\documentclass[12pt]{amsart}
\usepackage{graphicx}
\usepackage{hyperref}
\usepackage{setspace}
\usepackage{enumerate}
\usepackage{amsmath,amsfonts,amssymb,amscd}
\usepackage{bbm}
\newtheorem{thm}{Theorem}[section]

\newtheorem{lem}[thm]{Lemma}
\newtheorem{prop}[thm]{Proposition}

\theoremstyle{definition}
\newtheorem{defn}[thm]{Definition}

\theoremstyle{remark}
\newtheorem{rem}[thm]{Remark}
\numberwithin{equation}{section}

\let\emptyset\varnothing

\newcommand*{\medcap}{\mathbin{\scalebox{1.5}{\ensuremath{\cap}}}}

\newcommand{\tsp}{\thinspace}

\begin{document}

\title[Diophantine equations in primes]{Diophantine equations in primes:
density of prime points on affine hypersurfaces}

\author{Shuntaro Yamagishi}
\address{Mathematisch Instituut, Universiteit Utrecht, Budapestlaan 6, NL-3584 CD Utrecht, The Netherlands}
\email{s.yamagishi@uu.nl}
\indent

\date{Revised on \today}

\begin{abstract}
Let $F \in \mathbb{Z}[x_1, \ldots, x_n]$ be a homogeneous form of degree $d \geq 2$, and let $V_F^*$ denote the singular locus of the affine variety
$V(F) = \{ \mathbf{z} \in {\mathbb{C}}^n: F(\mathbf{z}) = 0 \}$. 
In this paper, we prove the existence of integer solutions with prime coordinates to the equation $F(x_1, \ldots, x_n) = 0$ provided $F$ satisfies suitable local conditions and $n - \dim V_F^* \geq  2^8 3^4 5^2 d^3 (2d-1)^2 4^{d}$. Our result improves on what was known previously due to Cook and Magyar (B. Cook and {\'A}. Magyar, `Diophantine equations in the primes'. Invent. Math. 198 (2014), 701-737), which required $n - \dim V_F^*$ to be an exponential tower in $d$.
\end{abstract}



\maketitle

\section{Introduction}
Solving Diophantine equations in primes is a fundamental problem in number theory.
There have been many significant results recently regarding solving linear equations in primes.
For example, the celebrated work of Green and Tao \cite{GT0} on arithmetic progressions in primes can be phrased as the statement that given any $n \in \mathbb{N}$ the system of linear equations
$$
x_{i+2} - {x_{i+1}} = x_{i+1} - x_{i} \ (1 \leq i \leq n)
$$
has a solution $(p_1, \ldots, p_{n+2})$ where each $p_i$ is prime 
and $p_1< p_2 < \cdots < p_{n+2}$.
A major achievement extending this result in which a more general system of linear equations is considered has been established by Green, Tao and Ziegler (see \cite{GT1, GT2, GTZ}).
Another important achievement in this area
includes 
the breakthrough on the problem of bounded gaps between primes by
 Maynard \cite{M},  Tao (see  \cite[pp. 385]{M}) and  Zhang \cite{Z}.
In particular, it was shown in \cite{M} that at least one of the equations
$$
x_1 - x_2 = 2j \ (1 \leq j \leq 300)
$$
has infinitely many integer solutions with prime coordinates.
There is also the work of Helfgott on the ternary Goldbach problem \cite{H0}.
It was proved by Vinogradov \cite{V} that the equation
$$
x_1 + x_2 + x_3 = N
$$
has an integer solution with prime coordinates for all sufficiently large odd $N \in \mathbb{N}$.
Helfgott proved that the assertion holds for all odd $N \in \mathbb{N}$ greater than or equal to 7,
establishing what is known as the ternary Goldbach problem.

In contrast to the great progress achieved for the linear case, the progress regarding solving general higher degree polynomial equations in primes
has been limited. Let $F \in \mathbb{Z}[x_1, \ldots, x_n]$ be a homogeneous form of degree $d \geq 2$.
We are interested in establishing the existence of prime solutions, which are integer solutions with prime coordinates, to the equation
\begin{equation}
\label{eqn main}
F(x_1, \ldots, x_n) = 0.
\end{equation}
For $d=2$ there are results due to Liu \cite{L2} and Zhao \cite{Zh}.
The first result in this direction for higher degrees was achieved by Cook and Magyar \cite{CM} (also for related work, see \cite{XY, Y2}). By applying the Hardy-Littlewood circle method, they established the existence of prime solutions to equations of the shape (\ref{eqn main}) under suitable local conditions; however, their result requires $n$ to be an exponential tower in $d$. In comparison to the situation for integer solutions, it is expected that an exponential in $d$ should be possible with current technology, and this is precisely what we establish in this paper.
First we introduce some notation in order to state our result.
Let $\wp$ denote the set of prime numbers. Let $\mathbb{Z}_p^{\times}$ be the units of $p$-adic integers. We consider the following condition.

Local conditions ($\star$): The equation (\ref{eqn main})
has a non-singular real solution in $(0,1)^{n}$, and also
has a non-singular solution in $( \mathbb{Z}_p^{\times})^n$ for every $p \in \wp$.
\newline
\newline
Let $V^*_{F}$ denote the singular locus of $V(F) = \{  \mathbf{z} \in \mathbb{C}^n:  F(\mathbf{z}) = 0  \}$, i.e. it is the affine variety in $\mathbb{A}^n_{\mathbb{C}}$ defined by 
\begin{equation}
\label{sing loc}
V_{F}^* = \left\{ \mathbf{z} \in \mathbb{C}^n:  \nabla F(\mathbf{z}) = \mathbf{0} \right\},
\end{equation}
where
$\nabla F = \left( \frac{\partial F}{\partial x_1}, \ldots,  \frac{\partial F}{\partial x_n} \right)$.
Let
\begin{eqnarray}
\label{defnoflamdastar}
\Lambda^*(x) =
\left\{
    \begin{array}{ll}
         \log x
         &\mbox{if } x \in \wp, \\
         0
         &\mbox{otherwise. }
    \end{array}
\right.
\end{eqnarray}
Given $\mathcal{H} \subseteq \mathbb{C}^n$ we let $\mathbbm{1}_{ \mathcal{H}}$ be the characteristic function of $\mathcal{H}$.
We define
$$
N_{\wp}(F; X) = \sum_{\mathbf{x} \in [0,X]^n } \Lambda^*(x_1) \cdots \Lambda^*(x_n) \mathbbm{1}_{ V(F) } (\mathbf{x}).
$$
In this paper, we establish the following result.
\begin{thm}
\label{mainthm1}
Let $F \in \mathbb{Z}[x_1, \ldots, x_n]$ be a homogeneous form of degree $d \geq 2$. Suppose $F$ satisfies the local conditions \textnormal{($\star$)} and
\begin{eqnarray}
\label{codim of F}
n -  \dim V_F^* \geq 2^8 3^4 5^2 d^3 (2d-1)^2 4^{d}.
\end{eqnarray}
Then we have
$$
N_{\wp}(F; X)  \gg X^{n-d}.
$$
\end{thm}
From the main result of Cook and Magyar \cite[Theorem 1]{CM} the same conclusion can be deduced provided $n - \dim V_F^* > \mathfrak{C}_d$,
where $\mathfrak{C}_d$ is a quantity depending only on $d$.
Let us define $\mathfrak{t}_1 = 1$ and recursively define $\mathfrak{t}_{j+1} = d^{\mathfrak{t}_j}$ for $j \geq 1$.
By going through the proof of \cite[Theorem 1]{CM}, it can be verified that a crude lower bound for $\mathfrak{C}_d$ is given by
$$
\mathfrak{C}_d > (\mathfrak{t}_{d-1} 2^{5d} d! - 1) (d-1)^{\mathfrak{t}_{d-1} 2^{5d} d! - 1}.
$$
This is significantly larger than what is required in the situation for integer solutions (In the work of Birch \cite[Theorem]{B}, the requirement is
$n -  \dim V_F^* > (d-1) 2^d$.); however, the importance of \cite[Theorem 1]{CM} is in establishing the existence of
such a quantity. In contrast, we see that the requirement on $n -  \dim V_F^*$ in Theorem \ref{mainthm1}
is comparable to that in \cite[Theorem]{B}. We obtain Theorem \ref{mainthm1} as a consequence of Theorem \ref{mainthm2} which we state in the next section.

Throughout we use $\ll$ and $\gg$ to denote Vinogradov's well-known notation, i.e.
the statement $f \ll g$ means there exists a positive constant $C$ (it may depend on parameters
which are regarded as fixed) such
that $|f| \leq  C g$ for all values under consideration, and  the statement $g \gg f$ is equivalent to $f \ll g$.
We also make use of the $O$-notation; the statement $f = O(g)$ is equivalent to $f \ll g$.

\textit{Acknowledgements.} The author would like to thank Tim Browning, Brian Cook, Liqun Hu (Nanchang University), Simon Rydin Myerson, Damaris Schindler and Trevor Wooley for many helpful discussions. The author is very grateful to Trevor Wooley for his encouragement and
Stanley Yao Xiao for pointing the author to the work of Cook and Magyar \cite{CM} when the author was still a graduate student.
A large portion of this work regarding the minor arcs was completed while the author was attending the Thematic Program on Unlikely Intersections, Heights, and Efficient Congruencing
at the Fields Institute, and the major arcs at the University of Bristol.
The author would like to thank the Fields Institute and the University of Bristol for providing excellent environments to work on this paper.
He also received additional support from M. Ram Murty and the Department of Mathematics and Statistics at Queen's University,
EPSRC grant \texttt{EP/P026710/1} and the NWO Veni Grant \texttt{016.Veni.192.047}.
All the generous hospitality and the support received while completing this work are gratefully acknowledged.
Finally, the author would like to thank the anonymous referees for many useful comments.


\section{Overview and notation}
In order to state Theorem \ref{mainthm2}, we need to introduce more notation.
We begin by defining the following class of smooth weights.
\begin{defn}
Let $\delta, \mathfrak{c} > 0$ and $M_0 \in \mathbb{Z}_{\geq 0}$.
We define $\mathcal{S}^+( \delta ;  M_0; \mathfrak{c})$ to be the set of smooth functions $\omega :\mathbb{R} \rightarrow [0, \infty)$ satisfying
\begin{enumerate}[(i)]
\item     $[-\delta/2, \delta/2] \subseteq \textnormal{supp}(\omega) \subseteq [-\delta, \delta]$,
\item for any $k \in \{ 0, \ldots, M_0 \}$ we have
$\| \partial^{k} \omega / \partial x^{k} \|_{L^\infty(\mathbb{R})} \leq \mathfrak{c}$.
\end{enumerate}
\end{defn}

Let $N > 1$. Throughout the remainder of the paper all implicit constants in $\ll, \gg$ and the $O$-notation are independent of $N$.
Let $\mathbf{x}_0 = (x_{0,1}, \ldots, x_{0,n}) \in (0,1)^n$ and we define
\begin{eqnarray}
\label{defn of varpi}
\varpi(\mathbf{x}) =  \prod_{j = 1}^n \varpi_j (x_j)
\ \ \textnormal{  where  } \  \  \varpi_j(x_j) = \omega \left( \frac{x_j}{N} - x_{0,j}  \right) .
\end{eqnarray}
Let $\Lambda$ denote the von Mangoldt function, where $\Lambda(x)$ is $\log p$ if $x$ is a power of $p \in \wp$ and $0$ otherwise.
Given $\mathbf{x} = (x_1, \ldots , x_n)$ we let $\Lambda(\mathbf{x}) = \Lambda(x_1) \cdots \Lambda(x_n).$
\begin{thm}
\label{mainthm2}
Let $F \in \mathbb{Z}[x_1, \ldots, x_n]$ be a homogeneous form of degree $d \geq 2$ satisfying (\ref{codim of F}) and the local conditions \textnormal{($\star$)}. Let $\mathbf{x}_0 \in (0,1)^n$ be a non-singular real solution to the equation $F(\mathbf{x}) = 0$.
Let $\delta, \mathfrak{c} > 0$ and $\omega \in \mathcal{S}^+( \delta ; n; \mathfrak{c})$, where $\delta$ is sufficiently small with respect to $F$ and $\mathbf{x}_0$, and let $\varpi$ be as in (\ref{defn of varpi}).
Then for any $A > 0$ we have
$$
\sum_{\mathbf{x} \in [0,N]^n } \varpi(\mathbf{x}) \Lambda(\mathbf{x})\mathbbm{1}_{ V(F) } (\mathbf{x}) = c(F; \omega, \mathbf{x}_0) \, N^{n-d} + O \left(  \frac{N^{n-d}}{(\log N)^A} \right),
$$
where $c(F; \omega, \mathbf{x}_0) > 0$ is a constant depending only on $F$, $\omega$ and $\mathbf{x}_0$.
\end{thm}
\begin{rem}
As it will be evident in the proof, it is sufficient to assume $F$ satisfies (\ref{codimcondnlong}) instead of (\ref{codim of F}).
\end{rem}
There are couple differences worth mentioning regarding Theorem \ref{mainthm2} compared to the main result of Cook and Magyar. First
their result establishes the asymptotic formula without the additional smooth weight $\varpi$ as in Theorem \ref{mainthm2}.
Also they establish the result for a system of polynomials of equal degrees, whereas Theorem \ref{mainthm2} is for a single homogeneous form.
Though we do not explore it here, we expect our method to be able to overcome these differences with additional technical effort.


Let us define the following exponential sum
\begin{equation}
\label{def S}
S(\alpha) = \sum_{\mathbf{x} \in [0, N]^n }  \varpi(\mathbf{x}) \Lambda(\mathbf{x}) e( \alpha F(\mathbf{x}) ).
\end{equation}
We define the \textit{major arcs} $\mathfrak{M}(\vartheta_0)$ to be the set of points
$\alpha \in [0,1)$ satisfying the following: there exist $1 \leq q \leq N^{\vartheta_0}$ and
$a \in \mathbb{Z}$ with
$$
\gcd(a, q) = 1 \ \ \ \text{  and  } \ \ \ \  | q \alpha - a | \leq N^{\vartheta_0 - d}.
$$
We define the \textit{minor arcs} to be the complement $\mathfrak{m}(\vartheta_0) =  [0,1) \backslash \mathfrak{M}(\vartheta_0)$.
By the orthogonality relation we have
\begin{eqnarray}
\label{orthog reln}
\sum_{\mathbf{x} \in [0,N]^n } \varpi(\mathbf{x})  \Lambda(\mathbf{x}) \mathbbm{1}_{ V(F) } (\mathbf{x}) = \int_{0}^1 S({\alpha}) \tsp {d}{\alpha}
= \int_{\mathfrak{M}(\vartheta_0) } S({\alpha}) \tsp {d}{\alpha} + \int_{\mathfrak{m}(\vartheta_0) } S({\alpha}) \tsp {d}{\alpha}.
\end{eqnarray}

We now describe the ideas behind our approach in a non-rigorous manner.
Let us loosely refer to $n - \dim V_F^*$ as the `rank' of $F$.
Consider a partition of variables $\mathbf{x} = (\mathbf{u}, \mathbf{v})$. Let us write
$F(\mathbf{u}, \mathbf{v}) = F_{\mathbf{u}}(\mathbf{u}) + G(\mathbf{u}, \mathbf{v}) + F_{\mathbf{v}}(\mathbf{v})$, where
$F_{\mathbf{u}}$ and $F_{\mathbf{v}}$ are portions of $F$ consisting only of monomials in $\mathbf{u}$ and $\mathbf{v}$ respectively, and
every monomial of $G$ consists of variables from both $\mathbf{u}$ and $\mathbf{v}$.
We apply the Cauchy-Schwarz inequality twice to $S(\alpha)$, first with respect to
$\mathbf{u}$ and then $\mathbf{v}$, to remove all the $\Lambda$-weights, and obtain
$$
|S(\alpha)|^4
\ll
(N \log N)^{2n} \Big{|} \sum_{\mathbf{u}, {\mathbf{u}'} \in [0, N]^{m}} \ \sum_{\mathbf{v}, {\mathbf{v}'} \in [0, N]^{n - m}} e(\alpha
(F(\mathbf{u}, \mathbf{v}) - F(\mathbf{u}, \mathbf{v}') - F(\mathbf{u}', \mathbf{v}) + F(\mathbf{u}', \mathbf{v}') ) )\Big{|}.
$$
It turns out that the rank of
$F(\mathbf{u}, \mathbf{v}) - F(\mathbf{u}, \mathbf{v}') - F(\mathbf{u}', \mathbf{v}) + F(\mathbf{u}', \mathbf{v}')$
is governed by the rank of $G$. Thus
if $G$ has a large rank, then we easily obtain a suitable estimate for $|S(\alpha)|$ by appealing to the work of Birch \cite{B}.
The challenge lies when such a convenient partition of the $\mathbf{x}$ variables, for which $G$ has a large rank, does not exist.
An extreme example where such a partition does not exist is the diagonal case
$F(\mathbf{x}) = A_1 x_1^d + \cdots + A_n x_n^d$, but in general we can not assume it has 
an additive structure as in this example. When no such partition exists, it essentially implies that the rank is `well-distributed'
amongst the variables, and this makes it possible to prove a pigeonhole principle type estimate regarding the rank; given any $k \in \mathbb{N}$ and a partition $\mathbf{x} = (\mathbf{u}_1, \ldots, \mathbf{u}_k)$, there exists at least one $1 \leq \ell \leq k$
such that $F |_{\mathbf{u}_j = \mathbf{0} \ (j \not = \ell)}$ has a `large' rank.
This condition becomes useful when we apply Vaughan's identity,
because it allows us to reduce our study of the exponential sum to the case where all of the weights are either of Type I or Type II component.
When all the weights are of Type I component, we appeal to the work of Schindler and Sofos \cite{SS}, where they handle similar exponential sums
defined over uneven boxes. On the other hand, when all the weights are of Type II component, the main input to treat this case is the work of the author regarding semiprime solutions to polynomial equations \cite{Y1};
the difference being the exponential sum is now over a product of hyperbolic regions instead of a box.
We overcome this technical challenge with a modest cost on the rank of $F$.
Our approach to the minor arcs estimate is completely different from
that of Cook and Magyar, which is based on mean value estimates. One way to consider
their approach is that they were able to make the strategy used for the additive case work
for this problem involving homogeneous forms, but with a large cost on the rank, while our approach
incorporates the work of Birch in this setting.
However, our method requires $\vartheta_0 \gg 1$ which is considerably larger compared to
$\vartheta_0 = \frac{C \log \log N}{ \log N}$ for some fixed $C > 0$ in \cite{CM}. For this reason we can not appeal to the Siegel-Walfisz theorem in our major arcs analysis as in their approach. Instead we make use of the Dirichlet characters and the zero-density estimates of the Dirichlet $L$-functions.
There has been much work regarding enlarging the major arcs for additive problems concerning primes (for example, see \cite{L, MV, R});
in contrast, we establish an analogous result for equations involving homogeneous forms.


\section{Preliminaries}
\label{prem}
Given a homogeneous form $F \in \mathbb{C}[x_1, \ldots, x_n]$, we let
$\textnormal{codim}_{\mathbb{A}^n_{\mathbb{C}}} V_{F}^* = n - \dim V_F^*$; this is the codimension of $V_{F}^*$ as
a subvariety of $\mathbb{A}^n_{\mathbb{C}}$. Let $a_1, \ldots, a_n \in \mathbb{C} \backslash \{ 0 \}$ and
$\mathcal{G}(\mathbf{x}) =  F(a_1 x_1, \ldots, a_n x_n)$. By the chain rule  it follows that
\begin{eqnarray}
\label{chain rule codim}
\textnormal{codim}_{\mathbb{A}^n_{\mathbb{C}}} V_{F}^* = \textnormal{codim}_{\mathbb{A}^n_{\mathbb{C}}} V_{\mathcal{G}}^*.
\end{eqnarray}
It can be verified easily that for any $r \geq n$ we have
\begin{eqnarray}
\label{dontmatter}
\textnormal{codim}_{\mathbb{A}^r_{\mathbb{C}}} V_{F}^* = \textnormal{codim}_{\mathbb{A}^{n}_{\mathbb{C}}} V_{F}^*,
\end{eqnarray}
where
$$
\textnormal{codim}_{\mathbb{A}^r_{\mathbb{C}}} V_{F}^* =  r - \dim
\Big{\{} (z_1, \ldots, z_r) \in \mathbb{C}^r: \frac{\partial F}{\partial x_i}(z_1, \ldots, z_n) = 0 \ (1 \leq i \leq r)   \Big{\}},
$$
i.e. on the left hand side of (\ref{dontmatter}) we are considering $F$ as an element of $\mathbb{C}[x_1, \ldots, x_r]$ which does not depend on $x_{r+1}, \ldots, x_n$.

We also have the following lemma regarding the codimension of $V_{F}^*$.
\begin{lem}
\label{Lemma on the B rank}
Let $F \in \mathbb{C}[x_1, \ldots, x_n]$ be a homogeneous form of degree $d \geq 2$. Let $0 \leq s < n$ and set
$$
\mathfrak{F}( x_{s + 1}, \ldots, x_{n}) =
F( 0, \ldots, 0,  x_{s + 1}, \ldots, x_{n}).
$$
Then we have
$$
\textnormal{codim}_{\mathbb{A}_{\mathbb{C}}^{n}} V_{F}^*  - 2 s \leq \textnormal{codim}_{\mathbb{A}_{\mathbb{C}}^{n-s}} V_{\mathfrak{F}}^* \leq \textnormal{codim}_{\mathbb{A}_{\mathbb{C}}^{n}} V_{F}^*.
$$
\end{lem}
\begin{proof}
We refer the reader to \cite[Lemma 3.1]{SS} for the lower bound. We now prove the upper bound when $s=1$; the upper bound when $s > 1$
follows easily from this case. It follows from the definition of the singular locus that
$$
V_{F}^* \cap \{ \mathbf{x} \in \mathbb{C}^n : x_1 = 0 \} \subseteq \{0\} \times V_{\mathfrak{F}}^* \subseteq \mathbb{A}_{\mathbb{C}}^n.
$$
Since the dimension of $V_{F}^* \cap \{ \mathbf{x} \in \mathbb{C}^n : x_1 = 0 \}$ is either $\dim V_F^* - 1$ or $\dim V_{F}^*$, we have
$\dim V_F^* - 1 \leq \dim V_{\mathfrak{F}}^*$. Therefore, we obtain
$$
\textnormal{codim}_{\mathbb{A}_{\mathbb{C}}^{n-1}} V_{\mathfrak{F}}^*  = n - 1 - \dim V_{\mathfrak{F}}^*     \leq  n - \dim V_F^* = \textnormal{codim}_{\mathbb{A}_{\mathbb{C}}^{n}} V_{F}^*.
$$

\end{proof}
Suppose we have a partition of variables $\mathbf{x} = (\mathbf{u}, \mathbf{v})$, where $\mathbf{u} = (u_1, \ldots, u_m)$
and $\mathbf{v} = (v_{m + 1}, \ldots, v_n)$. Let us denote $F_{\mathbf{u}}(\mathbf{u}) = F(\mathbf{u}, \mathbf{0})$,
$F_{\mathbf{v}}(\mathbf{v}) = F(\mathbf{0}, \mathbf{v})$ and $G_{\mathbf{u}, \mathbf{v}}(\mathbf{u}, \mathbf{v}) = F(\mathbf{u}, \mathbf{v}) - F(\mathbf{u}, \mathbf{0}) - F(\mathbf{0}, \mathbf{v})$.
With these notation we have the following lemma.
\begin{lem}
\label{lem on the split of B rank}
Let $F \in \mathbb{C}[x_1, \ldots, x_n]$ be a homogeneous form of degree $d \geq 2$.
Then we have
$$
\textnormal{codim}_{\mathbb{A}_{\mathbb{C}}^{n}} V^*_{F} \leq \textnormal{codim}_{\mathbb{A}_{\mathbb{C}}^{m}} V^*_{F_{\mathbf{u}}} + \textnormal{codim}_{\mathbb{A}_{\mathbb{C}}^{n}} V^*_{G_{\mathbf{u}, \mathbf{v}}} + \textnormal{codim}_{\mathbb{A}_{\mathbb{C}}^{n-m}} V^*_{F_{\mathbf{v}}}.
$$
\end{lem}
\begin{proof}
Given $f_1, \ldots, f_r \in \mathbb{C}[x_1, \ldots, x_n]$ let $(f_1, \ldots, f_r)$ denote the ideal in $\mathbb{C}[x_1, \ldots, x_n]$
generated by $f_1, \ldots, f_r$. Also given $I \subseteq \mathbb{C}[x_1, \ldots, x_n]$ we let $V(I) = \{ \mathbf{x} \in \mathbb{C}^n :
f(\mathbf{x}) = 0 \ (f \in I)  \}$.
We define the following homogeneous forms:
$$
\lambda_i (\mathbf{u}, \mathbf{v}) = \frac{\partial F}{ \partial u_i}  \ \ (1 \leq i \leq m), \ \ \text{       } \ \ \lambda_j (\mathbf{u}, \mathbf{v}) = \frac{\partial F}{ \partial v_j} \ \ (m + 1 \leq j \leq n),
$$
$$
\widetilde{\lambda}_i (\mathbf{u}) = \frac{\partial F_{\mathbf{u}}}{ \partial u_i}  \ \ (1 \leq i \leq m), \ \ \text{       } \ \ \widetilde{\lambda}_j (\mathbf{v}) = \frac{\partial F_{\mathbf{v}}}{ \partial v_j} \ \ (m + 1 \leq j \leq n),
$$
$$
\tau_i (\mathbf{u}, \mathbf{v}) = \frac{\partial G_{\mathbf{u}, \mathbf{v}}}{ \partial u_i}  \ \ (1 \leq i \leq m) \ \ \text{  and  } \ \ \tau_j (\mathbf{u}, \mathbf{v}) = \frac{\partial G_{\mathbf{u}, \mathbf{v}}}{ \partial v_j} \ \ (m + 1 \leq j \leq n).
$$
Clearly $\lambda_i = \widetilde{\lambda}_i + \tau_i$ $(1 \leq i \leq n)$.
Let us define the following ideals: $I_1 = (\widetilde{\lambda}_1, \ldots, \widetilde{\lambda}_m)$, $I_2 = (\widetilde{\lambda}_{m+1}, \ldots, \widetilde{\lambda}_n)$ and $J = (\tau_1, \ldots, \tau_n)$.
Then we have
$$
V(I_1 + I_2 + J) \subseteq V(\lambda_1, \ldots, \lambda_n) = V_F^* \subseteq \mathbb{A}_{\mathbb{C}}^n.
$$
Therefore, it follows that
\begin{eqnarray}
\notag
\textnormal{codim}_{\mathbb{A}_{\mathbb{C}}^{n}} V^*_{F} &\leq&  \textnormal{codim}_{\mathbb{A}_{\mathbb{C}}^{n}} V(I_1 + I_2 + J)
\\
\notag
&\leq& \textnormal{codim}_{\mathbb{A}_{\mathbb{C}}^{n}} V(I_1)  + \textnormal{codim}_{\mathbb{A}_{\mathbb{C}}^{n}} V(I_2)  + \textnormal{codim}_{\mathbb{A}_{\mathbb{C}}^{n}} V(J)
\\
\notag
&=& \textnormal{codim}_{\mathbb{A}_{\mathbb{C}}^{m}} V_{F_{\mathbf{u}}}^*  + \textnormal{codim}_{\mathbb{A}_{\mathbb{C}}^{n-m}} V_{F_{\mathbf{v}}}^* + \textnormal{codim}_{\mathbb{A}_{\mathbb{C}}^{n}} V_{G_{\mathbf{u}, \mathbf{v}}}^*,
\end{eqnarray}
where we obtained the final equality via (\ref{dontmatter}).
\end{proof}

We shall refer to $\mathfrak{B} \subseteq \mathbb{R}^m$ as a box if $\mathfrak{B}$ is of the form
$
\mathfrak{B} = I_1 \times \cdots \times I_m,
$
where each $I_j$ is a closed, open or half open/closed interval $(1 \leq j \leq m)$.
Let $\mathfrak{G}(\mathbf{u};\mathbf{v})$ be a polynomial in variables $\mathbf{u}= (u_1, \ldots, u_{m_1})$ and $\mathbf{v}= (v_1, \ldots, v_{m_2})$.
We say $\mathfrak{G}(\mathbf{u};\mathbf{v})$ is bihomogeneous of bidegree $(d_1, d_2)$ if
$$
\mathfrak{G}( s u_1, \ldots, s u_{m_1} ; t v_1, \ldots, t v_{m_2} ) = s^{d_1} t^{d_2} \mathfrak{G}(\mathbf{u} ; \mathbf{v}).
$$
We define the following affine varieties in $\mathbb{A}_{\mathbb{C}}^{m_1 + m_2}$:
\begin{equation}
\label{singlocbhmg}
V_{\mathfrak{G}, 1}^* = \left\{ (\mathbf{u}, \mathbf{v})  \in  \mathbb{C}^{m_1 + m_2} :
\frac{\partial \mathfrak{G}}{\partial u_i} (\mathbf{u}, \mathbf{v})
= 0 \ (1 \leq i \leq m_1) \right\}
\end{equation}
and
\begin{equation}
\label{singlocbhmg2}
V_{\mathfrak{G}, 2}^* = \left\{ (\mathbf{u}, \mathbf{v})  \in  \mathbb{C}^{m_1 + m_2} :
\frac{\partial \mathfrak{G}}{\partial v_i} (\mathbf{u}, \mathbf{v})
= 0 \ (1 \leq i \leq m_2) \right\};
\end{equation}
we take the partial derivatives with respect to the first set of variables (in the notation of $\mathfrak{G}$) for $V_{\mathfrak{G}, 1}^*$,
and with respect to the second set of variables for $V_{\mathfrak{G}, 2}^*$.
The following result was the key estimate in establishing the main result in \cite{Y1}.
\begin{prop}
\label{semiprimeprop} \cite[Theorem 5.1]{Y1}
Let $\mathfrak{F} \in \mathbb{C}[v_1, \ldots, v_m]$ be a homogeneous form of degree $d \geq 2$.
Let us define a bihomogeneous form
$$
\mathfrak{G}(\mathbf{u};\mathbf{v}) = \mathfrak{F}(u_1 v_1, \ldots, u_m v_m).
$$
Then we have
$$
\min \{ \textnormal{codim}_{\mathbb{A}^{2m}_{\mathbb{C}}}V^*_{\mathfrak{G}, 1}, \textnormal{codim}_{\mathbb{A}^{2m}_{\mathbb{C}}}V^*_{\mathfrak{G}, 2} \} \geq \frac{\textnormal{codim}_{\mathbb{A}^{m}_{\mathbb{C}}} V^*_{\mathfrak{F}}}{2}.
$$
\end{prop}
\begin{rem}
\label{dontmatter2}
In the statement of \cite[Theorem 5.1]{Y1}, the homogneous form in consideration has coefficients in $\mathbb{Z}$; however,
it can be seen from the proof that the result holds with $\mathbb{C}$ in place of $\mathbb{Z}$.
For a bihomogeneous form of the shape $\mathfrak{G}(\mathbf{u};\mathbf{v})$ as in Proposition \ref{semiprimeprop}, by symmetry we have
$V_{\mathfrak{G}, 1}^* \cong V_{\mathfrak{G}, 2}^*.$
In particular, it follows that
$$
\textnormal{codim}_{\mathbb{A}^{2m}_{\mathbb{C}}}V^*_{\mathfrak{G}, 1}  = 2m  - \dim V_{\mathfrak{G}, 1}^* =2m -  \dim V_{\mathfrak{G}, 2}^* = \textnormal{codim}_{\mathbb{A}^{2m}_{\mathbb{C}}}V^*_{\mathfrak{G}, 2}.
$$

\end{rem}
From here on we drop the subscript from the notation for the codimension; it will always be interpreted as a
subvariety of an appropriate affine space, and there is no ambiguity because of (\ref{dontmatter}).
The next lemma is an exponential sum estimate obtained by Schindler and Sofos in \cite{SS}; it handles a more general setting than in the work of Birch \cite{B}. Given $f \in \mathbb{C}[x_1, \ldots, x_m]$ let $\| f \|$ be the maximum of the absolute values of its coefficients,
and let $f^{[j]}$ denote the degree $j$ homogeneous portion of $f$.
\begin{lem}\cite[Lemma 2.5]{SS}
\label{Lemma 2.5 in SS}
Let $f \in \mathbb{Z}[v_1, \ldots, v_m]$ be a  polynomial of degree $d \geq 2$.
Let $\mathfrak{c}, \delta, \delta' > 0$ and $(z_1, \ldots, z_m ) \in (0,1)^m$, where $\delta$ is
sufficiently small with respect to $(z_1, \ldots, z_m )$.  Let $\omega \in \mathcal{S}^+( \delta ; 1; \mathfrak{c})$, $\epsilon \in \{0 , 1 \}$ and $\mathbf{u} \in [1, N^{\delta'}]^m$. We consider the exponential sum
$$
\widetilde{T}({\alpha}) = \sum_{ \substack{ 1 \leq  v_j \leq N/u_j  \\ (1 \leq j \leq m) }}
 \prod_{\ell=1}^m (\log v_{\ell})^{\epsilon} \
\omega \left( \frac{v_{\ell} u_{\ell}}{N} - z_{\ell}  \right) \cdot
e \left(  \alpha f(\mathbf{v} )  \right).
$$
Let
$$
\widetilde{K} = \frac{ \textnormal{codim} \tsp V^*_{  f^{[d]} } }{ 2^{d - 1} }.
$$
Let $0 < \vartheta < 1$. We denote $P_{\max} = \max_{1 \leq j \leq m} (N/u_j)$ and  $P_{\min} = \min_{1 \leq j \leq m} (N/ u_j)$.
Then for any $\varepsilon > 0$ at least one of the following alternatives holds:

\textnormal{i)} One has the upper bound
$$
|\widetilde{T}({\alpha})| \ll \left( \prod_{j=1}^m \frac{N}{u_j}  \right)^{1 + \varepsilon}
\left( \frac{ P_{\max}}{P_{\min}} \right)^{  \widetilde{K} } P_{\max}^{- \widetilde{K}  \vartheta}.
$$

\textnormal{ii)} There exist $1 \leq q \leq \| f^{[d]} \| P_{\max}^{(d-1) \vartheta}$ and $a \in \mathbb{Z}$ with $\gcd(a,q)=1$
such that
$$
| q \alpha - a | \leq P_{\min}^{-1} P_{\max}^{-(d-1) + (d-1) \vartheta}.
$$
Here the implicit constant is independent of $\vartheta$, $\mathbf{u}$ and the coefficients of $f$,
but it may depend on $\varepsilon$, $\mathfrak{c}$, $\delta,$ $\delta'$, $n$, $d$ and  $\mathbf{z}$.
\end{lem}
\begin{proof}
This is a slight variant of \cite[Lemma 2.5]{SS} and it can be obtained by making a minor modification to the argument.
If $\epsilon = 0$, then the result follows immediately from \cite[Lemma 2.5]{SS} by setting
$P_j = N / u_j$ $(1 \leq j \leq m)$. If $\epsilon = 1$, then the only difference
between $\widetilde{T}({\alpha})$ and the exponential sum in consideration in \cite[Lemma 2.5]{SS} is the
presence of the logarithmic weight.

The set-up for the proof of \cite[Lemma 2.5]{SS} begins after the statement of \cite[Theorem 2.1]{SS} (In particular, \cite[(2.2)]{SS} is not assumed at this point.). We follow their proof from this point on. The only modification required is the proof of \cite[Lemma 2.2]{SS},
where we have to keep track of the logarithmic weight. In \cite[Lemma 2.2]{SS}, the smooth weight is dealt with by partial summation at the end of the proof.
Our new weight, which is the same smooth weight multiplied by the logarithmic weight, can be dealt with in a similar manner
(except we need an additional factor of $P^{\varepsilon}$ for any $\varepsilon > 0$).
After \cite[Lemma 2.2]{SS}, we can simply follow their argument until the end of the proof of \cite[Lemma 2.5]{SS}
to establish our result. We note that as mentioned in \cite[Lines 11-13, Page 11]{SS}, the implicit constants do not depend on the
coefficients of $f$.
\end{proof}



\section{Structural dichotomy of $F$}
Let $F \in \mathbb{Z}[x_1, \ldots, x_n]$ be a homogeneous form.
Given partitions of variables $\mathbf{x} = (\mathbf{z}, \mathbf{w})$ and
$\mathbf{z} = (\mathbf{u}, \mathbf{v})$,  let us denote
$F_{\mathbf{z}}(\mathbf{z}) = F(\mathbf{z}, \mathbf{0}) = F(\mathbf{x}) |_{\mathbf{w} = \mathbf{0}}$
and
\begin{equation}
\label{den G 1}
\mathfrak{G}(\mathbf{u}, \mathbf{v} ) = F_{\mathbf{z}}(\mathbf{u}, \mathbf{v}) - F_{\mathbf{z}}(\mathbf{u}, \mathbf{0}) - F_{\mathbf{z}}(\mathbf{0}, \mathbf{v}).
\end{equation}
It is clear that $\mathfrak{G}(\mathbf{u}, \mathbf{0} )$ and $\mathfrak{G}(\mathbf{0}, \mathbf{v} )$ are both identically $0$.
With these notation we consider two cases based on the structure of $F$.
\begin{defn}
\label{dichotomy}
We define a structural dichotomy of $F$ with respect to $\mathcal{C}_0 > 0$ as follows:
\begin{enumerate}[(I)]
\item There exist partitions of variables $\mathbf{x} = (\mathbf{z}, \mathbf{w})$ and
$\mathbf{z} = (\mathbf{u}, \mathbf{v})$ such that $\textnormal{codim} \tsp V^*_{\mathfrak{G}} > \mathcal{C}_0$.
\item Given any partitions $\mathbf{x} = (\mathbf{z}, \mathbf{w})$ and
$\mathbf{z} = (\mathbf{u}, \mathbf{v})$, we have $\textnormal{codim} \tsp V^*_{\mathfrak{G}} \leq \mathcal{C}_0$.
\end{enumerate}
\end{defn}
Let $\mathcal{C}_0$ be the least integer satisfying
\begin{equation}
\mathcal{C}_0 > \frac{d(d-1) 8}{ \vartheta_0}  \  2^d,
\end{equation}
where $\vartheta_0 > 0$ is a constant to be chosen later.
Suppose $F$ satisfies (I) and let us denote $\mathbf{w} = (w_1, \ldots, w_s)$, $\mathbf{u} = (u_1, \ldots, u_{m})$ and $\mathbf{v} = (v_1, \ldots, v_{n - m -s})$. We also let $\mathbf{u}' = (u'_1, \ldots, u'_{m})$ and $\mathbf{v}' = (v'_1, \ldots, v'_{n - m -s})$. Let
\begin{eqnarray}
\mathfrak{g}_{\mathbf{w}}(\mathbf{u}, \mathbf{u}', \mathbf{v}, \mathbf{v}' ) &=& F(\mathbf{u}, \mathbf{v}, \mathbf{w}) - F(\mathbf{u}, {\mathbf{v}'}, \mathbf{w}) - F( {\mathbf{u}'}, \mathbf{v} , \mathbf{w}) + F({\mathbf{u}'}, {\mathbf{v}'} , \mathbf{w}).
\notag
\end{eqnarray}
Then we have
$$
\mathfrak{g}^{[d]}_{\mathbf{w}}(\mathbf{u}, {\mathbf{u}'}, \mathbf{v}, {\mathbf{v}'} ) =
\mathfrak{G}(\mathbf{u}, \mathbf{v}) - \mathfrak{G}(\mathbf{u}, {\mathbf{v}'}) - \mathfrak{G}( {\mathbf{u}'}, \mathbf{v}) + \mathfrak{G}({\mathbf{u}'}, {\mathbf{v}'})
$$
for each fixed $\mathbf{w}$.
By applying the Cauchy-Schwarz inequality twice, we obtain
\begin{eqnarray}
\label{CS twice ineq}
\\
\notag
|S(\alpha)| 
\ll
(N \log N)^{s + \frac{1}{2} (n-s)}  \max_{\mathbf{w} \in [0,N]^s} \Big{|} \sum_{\mathbf{u}, {\mathbf{u}'} \in [0, N]^{m}} \ \sum_{\mathbf{v}, {\mathbf{v}'} \in [0, N]^{n - m -s}} e(\alpha \mathfrak{g}_{\mathbf{w}}(\mathbf{u}, {\mathbf{u}'}, \mathbf{v}, {\mathbf{v}'} ))\Big{|}^{\frac{1}{4}}.
\end{eqnarray}
Since $\mathfrak{g}^{[d]}_{\mathbf{w}}(\mathbf{u}, \mathbf{0}, \mathbf{v}, \mathbf{0} ) = \mathfrak{G}(\mathbf{u}, \mathbf{v})$, it follows from Lemma \ref{Lemma on the B rank} that
\begin{equation}
\label{ineq codim 1'}
\textnormal{codim} \tsp V_{ \mathfrak{g}^{[d]}_{\mathbf{w}} }^* \geq \textnormal{codim} \tsp V_{\mathfrak{G}}^* > \mathcal{C}_0
\end{equation}
for each fixed $\mathbf{w}$. Then by \cite[Lemma 4.3]{B} we obtain the following.
\begin{lem}
\label{S est 1}
Let $F \in \mathbb{Z}[x_1, \ldots, x_n]$ be a homogeneous form of degree $d \geq 2$ satisfying \textnormal{(I)}. Let $0 < \vartheta < 1$
and $\varepsilon > 0$.
Then at least one of the following alternatives holds:

\textnormal{i)} One has the upper bound
$$
|S(\alpha)| \ll N^{n - \vartheta 2^{- 1 - d} \mathcal{C}_0 + \varepsilon}.
$$

\textnormal{ii)} There exist $1 \leq q \leq N^{(d-1) \vartheta}$ and $a \in \mathbb{Z}$ with $\gcd (a, q) = 1$
such that
$$
| q \alpha - a | \leq N^{-d + (d-1)\vartheta}.
$$
\end{lem}
\begin{proof}
If we are not in case ii), then it follows from \cite[Lemma 4.3]{B} and (\ref{ineq codim 1'}) that
$$
\Big{|} \sum_{\mathbf{u}, {\mathbf{u}'} \in [0, N]^{m}} \ \sum_{\mathbf{v}, {\mathbf{v}'} \in [0, N]^{n - m -s}} e(\alpha \mathfrak{g}_{\mathbf{w}}(\mathbf{u}, {\mathbf{u}'}, \mathbf{v}, {\mathbf{v}'} ))\Big{|}
\ll N^{2(n-s) - \vartheta 2^{1 - d} \mathcal{C}_0  + \varepsilon},
$$
where the implicit constant is independent of $\mathbf{w}$.
Therefore, by (\ref{CS twice ineq}) we obtain
$$
|S(\alpha)|  \ll  N^{s + \frac{1}{2}(n-s) + \varepsilon} N^{\frac{1}{2}(n-s) - \vartheta 2^{- 1 - d} \mathcal{C}_0 +  \varepsilon} = N^{n - \vartheta 2^{- 1 - d} \mathcal{C}_0 + 2 \varepsilon}.
$$
Since there is nothing to prove if we are in case ii), this completes the proof.
\end{proof}
This is all we need when $F$ satisfies (I). We now consider the case where $F$ satisfies (II).
Under this assumption, given any partition of the $\mathbf{x}$ variables it follows from Lemma \ref{lem on the split of B rank} that the rank must be concentrated in at least one of the subsets. We make this statement precise in the following lemma.
\begin{lem}
\label{lemma rank concn}
Let $F \in \mathbb{Z}[x_1, \ldots, x_n]$ be a homogeneous form of degree $d \geq 2$ satisfying \textnormal{(II)}. Suppose we have a partition of the $\mathbf{x}$ variables into $k$ sets $\mathbf{x} = (\mathbf{u}_1, \ldots, \mathbf{u}_{k})$. For each $1 \leq i \leq k$ let us denote
$$
F_{i} (\mathbf{u}_i) = F(\mathbf{x})|_{ \mathbf{u}_{\ell} = \mathbf{0} \tsp (\ell \not = i) }.
$$
Then there exists $1 \leq j \leq k$ satisfying
\begin{eqnarray}
\label{rank concn}
\textnormal{codim} \tsp V^*_{F_j} \geq    \frac{\textnormal{codim} \tsp V^*_{F}- (k-1) \mathcal{C}_0 }{k}.
\end{eqnarray}
\end{lem}
\begin{proof}
We prove the following claim by induction.

Claim: Given any $k \geq 2$ and partition $\mathbf{x} = (\mathbf{u}_1, \ldots, \mathbf{u}_{k})$, we have
$$
\textnormal{codim} \tsp V^*_{F} \leq \sum_{i=1}^k \textnormal{codim} \tsp V^*_{F_i} + (k-1) \mathcal{C}_0.
$$
\newline
It is clear that our result is an immediate consequence of this claim.
The base case $k = 2$ follows from our hypothesis and Lemma \ref{lem on the split of B rank}.
Suppose the statement holds for $k = k_0 - 1$.
We turn the partition $\mathbf{x} = (\mathbf{u}_1, \ldots, \mathbf{u}_{k_0})$ into a partition of $k_0 - 1$
sets by setting $\mathbf{z} = (\mathbf{u}_{k_0 - 1}, \mathbf{u}_{k_0})$.
Let us denote $\mathbf{w} = (\mathbf{u}_1, \ldots, \mathbf{u}_{k_0 -2})$ and
$F_{\mathbf{z}} (\mathbf{z}) = F(\mathbf{x}) |_{\mathbf{w} = \mathbf{0}}$.
Then by the inductive hypothesis we have
$$
\textnormal{codim} \tsp V^*_{F} \leq \sum_{i=1}^{k_0 - 2} \textnormal{codim} \tsp V^*_{F_i} + (k_0-2) \mathcal{C}_0
+
\textnormal{codim} \tsp V^*_{ F_{\mathbf{z}} }.
$$
By Lemma \ref{lem on the split of B rank} it follows that
$$
\textnormal{codim} \tsp V^*_{ F_{\mathbf{z}} } \leq \textnormal{codim} \tsp V^*_{F_{k_0 - 1}} + \textnormal{codim} \tsp V^*_{F_{k_0}}
+ \textnormal{codim} \tsp V^*_{ \mathfrak{G}_0 },
$$
where
$$
\mathfrak{G}_0(\mathbf{u}_{k_0 - 1}, \mathbf{u}_{k_0})
= F_{\mathbf{z}}(\mathbf{u}_{k_0 - 1}, \mathbf{u}_{k_0}) - F_{\mathbf{z}}(\mathbf{u}_{k_0 - 1}, \mathbf{0}) - F_{\mathbf{z}}(\mathbf{0}, \mathbf{u}_{k_0}).
$$
Since $\textnormal{codim} \tsp V^*_{ \mathfrak{G}_0 } \leq \mathcal{C}_0$, which follows from our hypothesis,
we obtain that the statement holds for $k = k_0$. This completes the proof of our claim.
\end{proof}

\section{Exponential sum estimate}
\label{sec exp sum}
We prove the following proposition to establish our minor arcs estimate in Section \ref{sec6.2}.
\begin{prop}
\label{main prop}
Let $F \in \mathbb{Z}[x_1, \ldots, x_n]$ be a homogeneous form of degree $d \geq 2$ satisfying \textnormal{(II)}.
Let $\varepsilon > 0$, $H$ be any positive even integer and  $\delta_1, \vartheta_0 > 0$ be such that
\begin{eqnarray}
\label{someconditions}
\max \left(  2 \delta_1 d +  (d-1) \frac{ 2 \delta_1 }{1 - 2 \delta_1},    \frac{1 - 2 \delta_1}{H} \right) < \vartheta_0 < 1.
\end{eqnarray}
Then for $N$ sufficiently large, at least one of the following alternatives holds:

\textnormal{i)} One has the upper bound
\begin{eqnarray}
\notag
|S({\alpha})|
&\ll&
N^{n + \varepsilon} \Big{(}
N^{- \frac{2}{3} \left( \frac{\vartheta_0 - d \delta_1  }{d-1} ( 1 -  \delta_1) -  \delta_1 \right) 2^{-d}(\textnormal{codim} \tsp V_{F}^* - 2\mathcal{C}_0)  }
\\
&+& N^{- \frac{2}{3} \left( \frac{\vartheta_0 - 2 \delta_1 d }{d-1} ( 1 -  2\delta_1) -  2\delta_1 \right) 2^{-d}(\textnormal{codim} \tsp V_{F}^* - {2} \mathcal{C}_0)  }
\notag
\\
&+&
N^{-\left( \frac{2 \vartheta_0 \delta_1}{2d-1} - \left( 1 + \frac{ 2 \delta_1 }{2d-1}  \right) \frac{1 - 2 \delta_1}{H} \right) \frac{ 2/3 }{2^{2d}H} \left(   \textnormal{codim} \tsp V_F^*  -
(3H - 1) \mathcal{C}_0 \right)
}
\Big{)}.
\notag
\end{eqnarray}

\textnormal{ii)}
There exist $1 \leq q \leq N^{\vartheta_0}$ and $a \in \mathbb{Z}$ with
$\gcd( a, q) = 1$ such that
\begin{eqnarray}
| q \alpha - a | \leq N^{\vartheta_0 - d} .
\notag
\end{eqnarray}
\end{prop}

One of the main ingredients to achieve Proposition \ref{main prop} is the following identity for $\Lambda$ known as Vaughan's identity \cite[\S 24]{D}.
\begin{lem}[Vaughan's identity]
\label{V's lem}
Let $U, V \geq 1$. Given any $x \in \mathbb{N}$ we have
\begin{eqnarray}
\notag
\Lambda(x)
=\Lambda(x)  \mathbbm{1}_{[1,U]}(x) + \sum_{st=x \atop{s\leq V}}\mu(s)\log t + \sum_{st=x \atop s\leq UV}
\nu_2(s)
+ \sum_{st=x\atop s>U,t>V}\Lambda(s)   \nu_3(t),
\end{eqnarray}
where
$$
\nu_2(s) = - \sum_{cd=s \atop c\leq V, d\leq U}\mu(c)  \Lambda(d) \ \ \text{  and  } \ \
\nu_3(t) =  - \sum_{c\mid t\atop c\leq V}\mu(c).
$$
\end{lem}
\begin{rem}
Given any $s, t \in \mathbb{N}$  we have
\begin{equation}
\label{triv bound on Chi}
|\nu_2(s)| \leq \sum_{d | s} \Lambda(d) = \log s    \ \ \text{  and  } \ \  |\nu_3(t)| \leq \sum_{c|t} 1 = \sigma_0(t).
\end{equation}
\end{rem}

We set
\begin{equation}
\label{defUV}
U = V = N^{\delta_1}.
\end{equation}
Recall the definition of $\varpi$ given in (\ref{defn of varpi}). In particular, for each  $1 \leq j \leq n$ we have
$$
\Lambda(x_j)  \mathbbm{1}_{[1,U]}(x_j) \varpi(\mathbf{x}) = 0  \  \ \,  (\mathbf{x} \in [0,N]^n).
$$
By Vaughan's identity we obtain
\begin{eqnarray}
S(\alpha) = \sum_{ \substack{ I_1, I_2, I_3 \subseteq \{1, \ldots, n \} \\ I_1 \cup I_2 \cup I_3 = \{1, \ldots, n\} \\ I_i \cap I_j = \emptyset \, (i \neq j)  } } \widetilde{S}_{I_1, I_2, I_3} (\alpha),
\label{3^nsum}
\end{eqnarray}
where
\begin{eqnarray}
\notag
\widetilde{S}_{I_1, I_2, I_3}(\alpha) &=&
\sum_{ \substack{ 1 \leq  s_{j_1} t_{j_1} \leq N \\ 1 \leq s_{j_1} \leq  V   \\ (j_1 \in I_1)  }} \
\sum_{ \substack{1 \leq s_{j_2} t_{j_2} \leq N \\ 1 \leq s_{j_2} \leq  U V  \\ (j_2 \in I_2) }}
\ \sum_{ \substack{1 \leq s_{j_3} t_{j_3} \leq N \\  s_{j_3} >  U, t_{j_3} > V  \\ (j_3 \in I_3) }} \
\prod_{j_1 \in I_1} \mu(s_{j_1}) (\log t_{j_1}) \varpi_{j_1}(s_{j_1} t_{j_1}) \cdot
\\
&&
\notag
\prod_{j_2 \in I_2} \nu_2(s_{j_2}) \varpi_{j_2}(s_{j_2} t_{j_2}) \cdot
\prod_{j_3 \in I_3} \Lambda(s_{j_3})  \nu_3(t_{j_3}) \varpi_{j_3}(s_{j_3} t_{j_3})
\cdot
\\
&&
\notag
e ( \alpha F(s_1 t_1, \ldots, s_n t_n) ).
\end{eqnarray}
Note in (\ref{3^nsum}), we allow the possibility of $I_j$ being the empty set.

We prove that given any partition $I_1, I_2$ and $I_3$ the sum $\widetilde{S}(\alpha) = \widetilde{S}_{I_1, I_2, I_3}(\alpha)$ satisfies
the statement of Proposition \ref{main prop} in place of $S(\alpha)$. 
Then Proposition \ref{main prop} follows by noting that there is nothing to prove if ii) is satisfied, and
by considering the partition that attains the largest value of $|\widetilde{S}(\alpha)|$ if ii) is not satisfied.
Consider the partition $\mathbf{x} = (\mathbf{u}_1, \mathbf{u}_2, \mathbf{u}_3)$
where $\mathbf{u}_j$ consists of all $x_i$'s with $i \in I_j$ $(1 \leq j \leq 3)$.
Then the rank is concentrated in at least one of $\mathbf{u}_1, \mathbf{u}_2$ and $\mathbf{u}_3$, i.e. there exists $1 \leq j \leq 3$ for which (\ref{rank concn}) is satisfied with $k=3$. We treat the three cases separately. To simplify the notation we let
\begin{equation}
\label{defnC_1}
\mathcal{C}_1 = \frac{ \textnormal{codim} \tsp V_F^* - 2 \mathcal{C}_0 }{3}.
\end{equation}

\textbf{Case 1}: The rank is concentrated in $\mathbf{u}_1$. Without loss of generality let $I_1 = \{ 1, \ldots,  m \}$.
Let us denote
$$
\mathfrak{d}_{\mathbf{s}, \mathbf{t}}(\mathbf{u}, \mathbf{v} ) = F(u_1 v_1, \ldots, u_m v_m, s_{m+1} t_{m+1}, \ldots, s_n t_n)
$$
and
\begin{eqnarray}
\label{def E case 2}
E_{\mathbf{s}, \mathbf{t}}(\alpha) = \sum_{ \substack{1 \leq u_{j} v_{j} \leq N \\  1 \leq u_{j} \leq V \\ (1 \leq j \leq m) }}  \mu(\mathbf{u}) \log (\mathbf{v})
 \prod_{\ell = 1 }^m \varpi_{\ell}(u_{\ell} v_{\ell})
\cdot e(\alpha  \mathfrak{d}_{\mathbf{s}, \mathbf{t}}(\mathbf{u}, \mathbf{v} ) ),
\end{eqnarray}
where $\mu(\mathbf{u}) = \mu(u_1) \cdots \mu(u_m)$ and $\log(\mathbf{v}) = \log(v_1) \cdots \log(v_m)$. Then we have
\begin{eqnarray}
\label{reduction case 2-1}
|\widetilde{S}(\alpha)|
&\leq& \sum_{ \substack{ 1 \leq s_{j_2} t_{j_2} \leq N \\ 1 \leq s_{j_2} \leq  U V \\ (j_2 \in I_2)  }}
\ \sum_{ \substack{ 1 \leq s_{j_3} t_{j_3} \leq N \\ s_{j_3} > U,  t_{j_3} > V \\ (j_3 \in I_3) }}
\ \prod_{j_2 \in I_2} (\log s_{j_2})  \varpi_{j_2}(s_{j_2} t_{j_2})
\cdot
\\
\notag
&&
\prod_{j_3 \in I_3} (\log N) \sigma_0 (t_{j_3}) \varpi_{j_3}(s_{j_3} t_{j_3})
\cdot
\left| E_{\mathbf{s}, \mathbf{t}}(\alpha) \right|
\\
\notag
&\ll&
N^{n - m + \varepsilon}
\max_{\mathbf{s}, \mathbf{t}}
\left| E_{\mathbf{s}, \mathbf{t}}(\alpha) \right|,
\end{eqnarray}
where the maximum is over the same range of $\mathbf{s}$ and $\mathbf{t}$ as in the above summation.
Let us fix $\mathbf{s}$ and $\mathbf{t}$, and we obtain an estimate for $E_{\mathbf{s}, \mathbf{t}}(\alpha)$ which is independent of the choice of
$\mathbf{s}$ and $\mathbf{t}$.
Let $\mathbf{V} = (V_1, \ldots, V_m)$, where each $V_j$ is a number greater than or equal to $\frac{1}{2}$ of the form $V/2^{k'}$ with $k' \in \mathbb{N}$. Then we have
\begin{eqnarray}
\label{def E case 2}
|E_{\mathbf{s}, \mathbf{t}}(\alpha)| \ll (\log N)^m  \ \max_{\mathbf{V}} \  (V_1 \cdots V_m) \ \max_{ \substack{ V_j < u_{j} \leq 2 V_j \\ (1 \leq j \leq m) }}
\left| \widetilde{E}_{\mathbf{s}, \mathbf{t}; \mathbf{u} }(\alpha) \right|,
\end{eqnarray}
where
$$
\widetilde{E}_{\mathbf{s}, \mathbf{t}; \mathbf{u} }(\alpha) = \sum_{ \substack{1 \leq v_{j} \leq N /u_j \\  (1 \leq j \leq m)  }} \log (\mathbf{v})
\prod_{\ell = 1}^{m}  \varpi_{\ell}(u_{\ell} v_{\ell})
 \cdot e(\alpha  \mathfrak{d}_{\mathbf{s}, \mathbf{t}}(\mathbf{u}, \mathbf{v} ) ).
$$
For a fixed $\mathbf{u} \in [1, V]^{m}$, it is clear that $\mathfrak{d}_{\mathbf{s}, \mathbf{t}}(\mathbf{u}, \mathbf{v})$ is a degree $d$ polynomial in $\mathbf{v}$ and its degree $d$ homogeneous portion $\mathcal{G}(v_1, \ldots, v_m)$ satisfies
$$
\mathcal{G}(v_1, \ldots, v_m) = F(u_1 v_1, \ldots, u_m v_m, 0, \ldots, 0)
$$
and $\| \mathcal{G} \| \leq V^d \|F \|$. Furthermore, $F_1(x_1, \ldots, x_m) =  F(x_1, \ldots, x_m, 0, \ldots, 0)$ is independent of $\mathbf{s}$ and $\mathbf{t}$, and it follows from (\ref{chain rule codim}) and (\ref{rank concn}), with $k=3$ and $j=1$, that
$$
\textnormal{codim} \tsp V_{\mathcal{G}}^* = \textnormal{codim} \tsp V_{F_1}^* \geq \mathcal{C}_1.
$$
Let $V_{\max} = \max_{1 \leq j \leq m} V_j$ and  $V_{\min} = \min_{1 \leq j \leq m} V_j$.
Let $0 < \vartheta < 1$ be such that
$$
- \vartheta + \delta_1 \vartheta + \delta_1 < 0 \ \  \textnormal{  and  } \ \   \delta_1 d + (d-1) \vartheta < \vartheta_0.
$$
By Lemma \ref{Lemma 2.5 in SS} it follows that either
we have
\begin{eqnarray}
\label{case 2 ineq 1}
|\widetilde{E}_{\mathbf{s}, \mathbf{t}; \mathbf{u} }(\alpha)| &\ll& \frac{N^{m + \varepsilon } }{\prod_{i = 1}^m u_i} N^{\delta_1 (\textnormal{codim} \tsp V_{\mathcal{G}}^*)/2^{d-1}}  N^{- (1 - \delta_1) \vartheta (\textnormal{codim} \tsp V_{\mathcal{G}}^*)/2^{d-1} }
\\
&\ll&
\notag
\frac{N^{m + \varepsilon} }{V_1 \cdots V_m } N^{( - \vartheta + \delta_1 \vartheta + \delta_1 ) \mathcal{C}_1/2^{d-1}},
\end{eqnarray}
or there exist $1 \leq q \leq \| F \| N^{\delta_1 d} N^{(d-1) \vartheta}$ and $a \in \mathbb{Z}$ with $\gcd(a,q)=1$ such that
$$
|q \alpha - a | \leq  \left( \frac{N}{ 2 V_{\max} } \right)^{-1}  \left( \frac{N}{ V_{\min} } \right)^{-(d-1) + (d-1) \vartheta} \leq
\frac{N^{\delta_1 (d - (d-1) \vartheta )  + (d-1) \vartheta} }{N^d } \leq N^{\vartheta_0 - d}.
$$
We also have
$\| F \| N^{\delta_1  d} N^{(d-1) \vartheta} \leq N^{\vartheta_0}$ for $N$ sufficiently large.
Finally, by (\ref{reduction case 2-1}), (\ref{def E case 2}) and  (\ref{case 2 ineq 1}) it follows that
\begin{eqnarray}
\label{S tilde est 2}
|\widetilde{S}(\alpha)| \ll N^{n  + ( - \vartheta + \delta_1 \vartheta + \delta_1 ) \mathcal{C}_1 2^{1 - d}  + \varepsilon},
\end{eqnarray}
and we obtain Proposition \ref{main prop} for this case.

\textbf{Case 2}: The rank is concentrated in $\mathbf{u}_2$. Without loss of generality let $I_2 = \{ 1, \ldots,  m \}$.
Let us denote
$$
\mathfrak{d}_{\mathbf{s}, \mathbf{t}}(\mathbf{u}, \mathbf{v} ) = F(u_1 v_1, \ldots, u_m v_m, s_{m+1} t_{m+1}, \ldots, s_n t_n)
$$
and
\begin{eqnarray}
\label{def E case 3}
E_{\mathbf{s}, \mathbf{t} }(\alpha) = \sum_{ \substack{1 \leq u_{j} v_{j} \leq N \\  1 \leq u_{j} \leq U V \\ (1 \leq j \leq m) }}   \nu_2(\mathbf{u})
\prod_{\ell=1}^{m}  \varpi_{\ell}(u_{\ell} v_{\ell})
\cdot e(\alpha  \mathfrak{d}_{\mathbf{s}, \mathbf{t}}(\mathbf{u}, \mathbf{v} ) ),
\end{eqnarray}
where $\nu_2(\mathbf{u}) = \nu_2(u_1) \cdots \nu_2(u_m)$.
Then we have
\begin{eqnarray}
\label{reduction case 3-1}
|\widetilde{S}(\alpha)|
&\leq& \sum_{ \substack{ 1 \leq s_{j_1} t_{j_1} \leq N \\ 1 \leq s_{j_1} \leq  V \\ (j_1 \in I_1)  }}
\ \sum_{ \substack{ 1 \leq s_{j_3} t_{j_3} \leq N \\ s_{j_3} > U,  t_{j_3} > V \\ (j_3 \in I_3) }}
\ \prod_{j_1 \in I_1} (\log t_{j_1})   \varpi_{j_1}(s_{j_1} t_{j_1})  \cdot
\\
\notag
&&
\prod_{j_3 \in I_3} (\log N) \sigma_0 (t_{j_3})  \varpi_{j_3}(s_{j_3} t_{j_3})  \cdot
\left| E_{\mathbf{s}, \mathbf{t}}(\alpha) \right|
\\
\notag
&\ll&
N^{n - m + \varepsilon} \max_{\mathbf{s}, \mathbf{t}}
\left| E_{\mathbf{s}, \mathbf{t}}(\alpha) \right|,
\end{eqnarray}
where the maximum is over the same range of $\mathbf{s}$ and $\mathbf{t}$ as in the above summation.
Let us fix $\mathbf{s}$ and $\mathbf{t}$, and we obtain an estimate for $E_{\mathbf{s}, \mathbf{t}}(\alpha)$ which is independent of
the choice of $\mathbf{s}$ and $\mathbf{t}$.
Let $\mathbf{V} = (V_1, \ldots, V_m)$, where each $V_j$ is a number greater than or equal to $\frac12$ of the form $UV/2^{k'}$ with $k' \in \mathbb{N}$. Then we have
\begin{eqnarray}
\label{ineq in case 3}
|E_{\mathbf{s}, \mathbf{t}}(\alpha)| \ll (\log N)^m  \max_{\mathbf{V}} \  (V_1 \cdots V_m) \ \max_{ \substack{ V_j < u_{j} \leq 2 V_j \\ (1 \leq j \leq m) } }
\left| \widetilde{E}_{\mathbf{s}, \mathbf{t}; \mathbf{u}}(\alpha) \right|,
\end{eqnarray}
where
$$
\widetilde{E}_{\mathbf{s}, \mathbf{t}; \mathbf{u}}(\alpha) = \sum_{ \substack{1 \leq v_{j} \leq N /u_j \\  (1 \leq j \leq m)  }}
\prod_{\ell=1}^{m}  \varpi_{\ell}(u_{\ell} v_{\ell}) \cdot
e(\alpha  \mathfrak{d}_{\mathbf{s}, \mathbf{t}}(\mathbf{u}, \mathbf{v} ) ).
$$
For a fixed $\mathbf{u} \in [1, UV]^{m}$, it is clear that $\mathfrak{d}_{\mathbf{s}, \mathbf{t}}(\mathbf{u}, \mathbf{v})$ is a degree $d$ polynomial in $\mathbf{v}$ and its degree $d$ homogeneous portion $\mathcal{G}(v_1, \ldots, v_m)$ satisfies
$$
\mathcal{G}(v_1, \ldots, v_m) =  F(u_1 v_1, \ldots, u_m v_m, 0, \ldots, 0)
$$
and $\| \mathcal{G} \| \leq (UV)^d \| F \|$. Furthermore, $F_2(x_1, \ldots, x_m) =  F(x_1, \ldots, x_m, 0, \ldots, 0)$ is independent of $\mathbf{s}$ and $\mathbf{t}$, and it follows from (\ref{chain rule codim}) and (\ref{rank concn}), with $k=3$ and $j=2$, that
$$
\textnormal{codim} \tsp V_{\mathcal{G}}^* = \textnormal{codim} \tsp V_{F_2}^* \geq \mathcal{C}_1.
$$
Let $V_{\max} = \max_{1 \leq j \leq m} V_j$ and  $V_{\min} = \min_{1 \leq j \leq m} V_j$.
Let $0 < \vartheta < 1$ be such that
$$
- \vartheta + 2 \delta_1 \vartheta + 2 \delta_1 < 0 \ \  \textnormal{  and  } \ \   2 \delta_1 d + (d-1) \vartheta < \vartheta_0.
$$
By Lemma \ref{Lemma 2.5 in SS} it follows that either
we have
\begin{eqnarray}
\label{case 3 ineq 1}
|\widetilde{E}_{\mathbf{s}, \mathbf{t}; \mathbf{u}}(\alpha)| &\ll& \frac{N^{m + \varepsilon}}{\prod_{i = 1}^m u_i} N^{2 \delta_1 (\textnormal{codim} \tsp V_{\mathcal{G}}^*)/2^{d-1}}  N^{- (1 - 2 \delta_1) \vartheta (\textnormal{codim} \tsp V_{\mathcal{G}}^* ) /2^{d-1} }
\\
\notag
&\ll&
\frac{N^{m + \varepsilon}}{ V_1 \cdots V_m } N^{( - \vartheta + 2 \delta_1 \vartheta + 2 \delta_1 ) \mathcal{C}_1/2^{d-1}},
\end{eqnarray}
or there exist $1 \leq q \leq \| F \| N^{2 \delta_1 d} N^{(d-1) \vartheta}$ and $a \in \mathbb{Z}$ with $\gcd(a,q)=1$ such that
$$
|q \alpha - a | \leq  \left( \frac{N}{ 2 V_{\max} } \right)^{-1}  \left( \frac{N}{ V_{\min} } \right)^{-(d-1) + (d-1) \vartheta} \leq
\frac{N^{2 \delta_1 (d - (d-1) \vartheta)  + (d-1) \vartheta} }{N^d} \leq N^{\vartheta_0 - d}.
$$
We also have
$\| F \| N^{2 \delta_1  d} N^{(d-1) \vartheta} \leq N^{\vartheta_0}$ for $N$ sufficiently large.
Finally, by (\ref{reduction case 3-1}), (\ref{ineq in case 3}) and (\ref{case 3 ineq 1}) it follows that
\begin{eqnarray}
\label{S tilde est 3}
|\widetilde{S}(\alpha)| \ll N^{n  + ( - \vartheta + 2 \delta_1 \vartheta + 2 \delta_1 ) \mathcal{C}_1 2^{1-d} + \varepsilon},
\end{eqnarray}
and we obtain Proposition \ref{main prop} for this case.

\textbf{Case 3}: The rank is concentrated in $\mathbf{u}_3$. Without loss of generality let $I_3 = \{ 1, \ldots,  m_0 \}$.
Let us denote
$$
\mathfrak{d}_{\mathbf{s}, \mathbf{t}} (\mathbf{a}, \mathbf{b} ) = F(a_1 b_1, \ldots, a_{m_0} b_{m_0}, s_{m_0+1} t_{m_0 + 1}, \ldots, s_n t_n)
$$
and
\begin{eqnarray}
E_{\mathbf{s}, \mathbf{t}}(\alpha) = \sum_{ \substack{ 1 \leq a_j b_j \leq N   \\  a_{j} >  U,  b_{j} > V \\ (1 \leq j \leq m_0) }}   \Lambda(\mathbf{a})  \nu_3(\mathbf{b})  \prod_{\ell = 1}^{m_0} \varpi_{\ell}(a_{\ell} b_{\ell}) \cdot e(\alpha  \mathfrak{d}_{\mathbf{s}, \mathbf{t}}(\mathbf{a}, \mathbf{b} ) ),
\end{eqnarray}
where $\nu_3(\mathbf{b}) = \nu_3(b_1) \cdots \nu_3(b_{m_0})$. Then we have
\begin{eqnarray}
\label{case 4 ineq 1}
|\widetilde{S}(\alpha)|
&\leq& \sum_{ \substack{  1 \leq s_{j_1} t_{j_1} \leq N  \\ 1 \leq s_{j_1} \leq  V  \\ (j_1 \in I_1)}} \
\sum_{ \substack{ 1 \leq s_{j_2} t_{j_2} \leq N \\ 1 \leq s_{j_2} \leq  U V  \\ (j_2 \in I_2) }}
\ \prod_{j_1 \in I_1} (\log t_{j_1}) \varpi_{j_1}(s_{j_1} t_{j_1}) \cdot
\\
&&
\prod_{j_2 \in I_2} (\log s_{j_2}) \varpi_{j_2}(s_{j_2} t_{j_2}) \cdot |E_{\mathbf{s}, \mathbf{t}}(\alpha)|
\notag
\\
&\ll&
N^{n - m_0 + \varepsilon} \max_{\mathbf{s}, \mathbf{t}} |E_{\mathbf{s}, \mathbf{t}}(\alpha)|,
\notag
\end{eqnarray}
where the maximum is over the same range of $\mathbf{s}$ and $\mathbf{t}$ as in the above summation.
It is clear that $\mathfrak{d}_{\mathbf{s}, \mathbf{t}}(\mathbf{a}, \mathbf{b})$ is a degree $2d$ polynomial in $\mathbf{a}$ and $\mathbf{b}$, and its degree $2d$ homogeneous portion is
$$
F_3(a_1 b_1, \ldots, a_{m_0} b_{m_0}) =  F(a_1 b_1, \ldots, a_{m_0} b_{m_0}, 0, \ldots, 0).
$$
Furthermore,  $F_3(x_1, \ldots, x_{m_0})$ is independent of $\mathbf{s}$ and $\mathbf{t}$, and
it follows from (\ref{rank concn}), with $k=3$ and $j=3$, that
$$
\textnormal{codim} \tsp V_{F_3}^* \geq \mathcal{C}_1.
$$

Recall we set $U = V = N^{\delta_1}$ in (\ref{defUV}).
Thus for each $1 \leq i \leq m_0$ the range of summation of $a_i$ in $E_{\mathbf{s}, \mathbf{t}}(\alpha)$ is $N^{\delta_1} < a_i < N^{1 - \delta_1}$.
Let
\begin{equation}
\label{delta2}
\delta_2 = (1 - 2 \delta_1)/ H.
\end{equation}
In order to estimate $E_{\mathbf{s}, \mathbf{t}}(\alpha)$, we consider
\begin{eqnarray}
\label{maxlambda}
\\
\notag
|E_{\mathbf{s}, \mathbf{t}}(\alpha) |
&\leq& H^{m_0}  \max_{ \boldsymbol{\lambda} } \Big{|}
\sum_{\substack{ N^{\lambda_i} < a_i \leq N^{\lambda_i + \delta_2}  \\ (1 \leq i \leq m_0) } } \
\sum_{\substack{ V <  b_j \leq N/a_j  \\ (1 \leq j \leq m_0) } }
\Lambda(\mathbf{a})
\nu_3(\mathbf{b})  \prod_{\ell=1}^{m_0} \varpi_{\ell}(a_{\ell} b_{\ell}) \cdot  e \left( \alpha \mathfrak{d}_{\mathbf{s}, \mathbf{t}} (\mathbf{a}, \mathbf{b} ) \right) \Big{|},
\notag
\end{eqnarray}
where the maximum is over all $\boldsymbol{\lambda} = (\lambda_1, \ldots, \lambda_{m_0})$ with
$$
\lambda_i \in \left\{ \delta_1, \delta_1 + \delta_2, \delta_1 + 2 \delta_2, \ldots,
\delta_1 + (H-1) \delta_2  \right\} \ (1 \leq i \leq m_0).
$$
Note $\delta_1 + H \delta_2 = 1 - \delta_1$ and $\delta_1 + \frac{H}{2} \delta_2 = \frac12$ (recall $H$ is even).

Let us fix $\boldsymbol{\lambda}$ for which the maximum occurs in (\ref{maxlambda}).
This gives a partition of $\mathbf{x} = (x_1, \ldots, x_{m_0})$ into $H$ sets $\mathbf{x} = (\mathbf{u}_0, \ldots, \mathbf{u}_{H-1})$,
where $\mathbf{u}_t$ is the collection of $x_i$'s such that  $i$ satisfies $\lambda_i = \delta_1 + t \delta_2$.
By Lemma \ref{lemma rank concn} there exists $j_0 \in \{ 0, 1, \ldots, H - 1 \}$ satisfying
\begin{eqnarray}
\label{codimcondF4}
\textnormal{codim} \tsp V^*_{F_{3,j_0}} \geq   \frac{\mathcal{C}_1 - (H -1) \mathcal{C}_0 }{H},
\end{eqnarray}
where
$$
F_{3,j_0}(\mathbf{u}_{j_0}) = F_{3}(\mathbf{x}) |_{\mathbf{u}_{\ell} = \mathbf{0} \tsp (\ell \not = j_0). }
$$
Let
\begin{equation}
\label{defdelta}
\lambda = \delta_1 + j_0 \delta_2.
\end{equation}
Without loss of generality let $\mathbf{u}_{j_0} = (x_1, \ldots, x_{m})$.
By estimating every other variable trivially, we obtain
\begin{eqnarray}
\label{reduced sum 1}
&&\Big{|} \sum_{\substack{ N^{\lambda_i} < a_i \leq N^{\lambda_i + \delta_2}  \\ (1 \leq i \leq m_0) } } \
\sum_{\substack{ V < b_j \leq N/a_j  \\ (1 \leq j \leq m_0) } }
\Lambda(\mathbf{a})  \nu_3(\mathbf{b})  \prod_{\ell=1}^{m_0} \varpi_{\ell}(a_{\ell} b_{\ell}) \cdot  e \left( \alpha \mathfrak{d}_{\mathbf{s}, \mathbf{t}} (\mathbf{a}, \mathbf{b} ) \right) \Big{|}
\\
&\ll&
N^{m_0 - m + \varepsilon}
\max_{ \widetilde{\mathbf{a}}, \widetilde{\mathbf{b}} }
\Big{|} \sum_{\substack{ N^{\lambda} < u_i \leq N^{\lambda + \delta_2}  \\ (1 \leq i \leq m) } } \
\sum_{\substack{ V <  v_j \leq N/u_j  \\ (1 \leq j \leq m) } }
\Lambda(\mathbf{u})
\nu_3(\mathbf{v})
\prod_{\ell = 1}^{m} \varpi_{\ell}(u_{\ell} v_{\ell})
\cdot  e \left( \alpha \mathfrak{d}_{j_0}(\mathbf{u},  \mathbf{v}) \right) \Big{|},
\notag
\end{eqnarray}
where the maximum is over the same range of
$\widetilde{\mathbf{a}} = (a_{m + 1}, \ldots, a_{m_0} )$ and $\widetilde{\mathbf{b}} = (b_{m + 1}, \ldots, b_{m_0} )$
as in the above summation,
and
\begin{eqnarray}
\notag
\mathfrak{d}_{j_0}(\mathbf{u}, \mathbf{v}) &=& \mathfrak{d}_{j_0}(u_1v_1, \ldots, u_m v_m)
\\
\notag
&=& \mathfrak{d}_{\mathbf{s}, \mathbf{t}} (u_1 v_1, \ldots, u_m v_m,  a_{m+1} b_{m+1}, \ldots, a_{m_0} b_{m_0})
\\
&=&  F( u_1 v_1, \ldots, u_m v_m,  a_{m+1} b_{m+1}, \ldots, a_{m_0} b_{m_0}, s_{m_0+1} t_{m_0 + 1}, \ldots, s_n t_n).
\notag
\end{eqnarray}
In particular, the coefficients of the lower degree terms of $\mathfrak{d}_{j_0}$ may depend on $\mathbf{s}, \mathbf{t}$, $\widetilde{\mathbf{a}}$ and $\widetilde{\mathbf{b}}$. The degree $2d$ homogeneous portion of $\mathfrak{d}_{j_0}(\mathbf{u}, \mathbf{v})$ is
$\mathfrak{d}_{j_0}^{[2d]}(\mathbf{u};\mathbf{v})  = F_{3, j_0}(u_1v_1, \ldots, u_{m}v_m)$.
With this set-up it follows from Proposition \ref{semiprimeprop} and (\ref{codimcondF4}) that
\begin{eqnarray}
\label{codim dj0}
\textnormal{codim} \tsp V^*_{ \mathfrak{d}_{j_0}^{[2d]}, 1} = \textnormal{codim} \tsp V^*_{ \mathfrak{d}_{j_0}^{[2d]}, 2} \geq \frac{
\textnormal{codim} \tsp V^*_{F_{3,j_0}}}{2} \geq  \frac{1}{2} \cdot  \frac{ \mathcal{C}_1 - (H -1) \mathcal{C}_0 }{H}.
\end{eqnarray}
We estimate the sum on the right hand side of (\ref{reduced sum 1}) in the next section.

\section{Weyl differencing and geometry of numbers}
\label{WD1}
Let $\delta_1$ and $\delta_2$ be as in Section \ref{sec exp sum} and let
$\lambda > 0$ be any real number. In order to estimate the sum on the right hand side of (\ref{reduced sum 1}), we consider the two cases $\lambda + \delta_2 \leq 1/2$ and $\lambda \geq 1/2$ separately.
The reason for treating these two cases separately is because our argument depends on which of the set of variables, $\mathbf{u}$ or $\mathbf{v}$, has a longer range of summation (The difference in the weights, namely $\Lambda(\mathbf{u})$ and $\nu_3 (\mathbf{v})$, does not affect the argument.). In the former case, $\mathbf{v}$ has a longer range of summation than $\mathbf{u}$, and vice versa for the latter case. However, the arguments are similar; the argument for the latter case becomes identical to that for the former case after a proper set-up.
We consider the case $\lambda \geq \delta_1$ and $\lambda + \delta_2 \leq 1/2$ in Section \ref{WD1'}, and
the case $\lambda \geq 1/2$ and $\lambda + \delta_2 \leq 1 - \delta_1$ in Section \ref{WD2}.

\subsection{$\lambda \geq \delta_1$ and $\lambda + \delta_2 \leq 1/2$}
\label{WD1'}
Let
\begin{eqnarray}
\label{WDI eqn1}
\mathfrak{T}(\alpha) = 
\sum_{\substack{ N^{ \lambda } <  u_j \leq N^{ \lambda + \delta_2}  \\ (1 \leq j \leq m)  } } \
\sum_{\substack{ V <  v_i \leq  N/u_i \\ (1 \leq i \leq m)  } }
\Lambda(\mathbf{u}) \nu_3 (\mathbf{v}) \prod_{\ell=1}^{m} \varpi_{\ell}(u_{\ell} v_{\ell})  \cdot  e \left( \alpha \mathfrak{d}_{j_0}(\mathbf{u}, \mathbf{v}) \right);
\end{eqnarray}
this is precisely the sum on the right hand side of (\ref{reduced sum 1}) when $\lambda$ is chosen as in (\ref{defdelta}).
Let $\mathbf{U} = (U_1, \ldots, U_m)$, where each $U_j$ is a number less than $N^{\lambda + \delta_2}$ of the form $2^{k'} N^{\lambda}$ with $k' \in \mathbb{Z}_{\geq 0}$. Then we have
\begin{eqnarray}
\label{WDI eqn2}
|\mathfrak{T}(\alpha)| \ll
(\log N)^{m} \ \max_{\mathbf{U}} |\mathfrak{T}_{\mathbf{U}} (\alpha)|,
\end{eqnarray}
where
\begin{eqnarray}
\mathfrak{T}_{\mathbf{U}}(\alpha)
\label{somesumT}
= \sum_{\substack{ U_j < u_j \leq 2 U_j \\ (1 \leq j \leq m)  } } \
\sum_{\substack{ V <  v_i \leq  N/u_i  \\ (1 \leq i \leq m)  } }
\Lambda(\mathbf{u})  \nu_3 (\mathbf{v})
\prod_{\ell=1}^{m} \varpi_{\ell}(u_{\ell} v_{\ell})
\cdot  e \left( \alpha g(\mathbf{v}, \mathbf{u}) \right)
\end{eqnarray}
and
$$
g(\mathbf{v}, \mathbf{u}) = \mathfrak{d}_{j_0}( \mathbf{u}, \mathbf{v} ).
$$
\begin{rem}
\label{remrange}
When $N^{\lambda + \delta_2}/2 \leq U_j < N^{\lambda + \delta_2}$, we replace the range of summation $U_j < u_j \leq 2 U_j$ in (\ref{somesumT}) with
$U_j < u_j \leq N^{\lambda + \delta_2}$.
\end{rem}
Clearly $g^{[2d]}(\mathbf{v}; \mathbf{u})$ is bihomogeneous of bidegree $(d,d)$.
It follows from (\ref{codim dj0}) that
\begin{eqnarray}
\label{codim ineq1}
\textnormal{codim} \tsp V^*_{g^{[2d]},1} = \textnormal{codim} \tsp V^*_{ \mathfrak{d}_{j_0}^{[2d]},2} \geq \frac{1}{2} \cdot \frac{\mathcal{C}_1 - (H -1)\mathcal{C}_0}{ H }.
\end{eqnarray}
Let us make the following changes in the notation: $\mathbf{x} = (x_1, \ldots, x_m) = \mathbf{v}$,  $\mathbf{y} = (y_1, \ldots, y_m) = \mathbf{u}$  and $Y_{i} = U_i$ $(1 \leq i \leq m)$;
these changes simplify the set-up in Section \ref{WD2}. By changing the order of summation in (\ref{somesumT}), we have
\begin{eqnarray}
\label{WDI eqn3}
\mathfrak{T}_{\mathbf{U}}(\alpha) &=&
\sum_{\substack{ V < x_i < N/ Y_i  \\ (1 \leq i \leq m)  } }
\nu_3 (\mathbf{x})
\sum_{\substack{ Y_j < y_j  \leq 2 Y_j \\ (1 \leq j \leq m)  } } \Lambda(\mathbf{y})   \psi(\mathbf{x} ; \mathbf{y})
 e ( \alpha g(\mathbf{x}, \mathbf{y}) ),
\end{eqnarray}
where
\begin{eqnarray}
\label{defn of weight WDI}
\psi(\mathbf{x} ; \mathbf{y})  =  \prod_{\ell=1}^{m} \varpi_{\ell}(x_{\ell} y_{\ell})  \mathbbm{1}_{[0, N/ y_{\ell}] } (x_{\ell}).
\end{eqnarray}

We let $\mathbf{j} = (j_1, \ldots, j_d)$ and $\mathbf{k} = (k_1, \ldots, k_d)$, and denote
\begin{eqnarray}
\label{defng2dsymm}
g^{[2d]}(\mathbf{x};\mathbf{y}) = \sum_{j_1 = 1}^{m} \cdots \sum_{j_d = 1}^{m}
\sum_{k_1 = 1}^{m} \cdots \sum_{k_d = 1}^{m}
G_{\mathbf{j}, \mathbf{k}} \tsp  x_{j_1} \cdots x_{j_d} \tsp y_{k_1} \cdots y_{k_d}
\end{eqnarray}
with each $G_{\mathbf{j}, \mathbf{k}} \in \mathbb{Q}$ symmetric in $(j_1, \ldots, j_d)$ and also in $(k_1, \ldots, k_d)$. Note we have $(d!)^2 G_{\mathbf{j}, \mathbf{k}} \in \mathbb{Z}$. Also $G_{\mathbf{j}, \mathbf{k}} = 0$ unless $(j_1, \ldots, j_d)$ is a permutation of $(k_1, \ldots, k_d)$; this is because  $g^{[2d]}(\mathbf{x};\mathbf{y}) = F_3(x_1y_1, \ldots, x_m y_m)$.

We denote
\begin{eqnarray}
\label{WDI eqn4}
\widetilde{ Y } = \prod_{i=1}^{m} Y_i, \ \ Y_{\max} = \max_{1 \leq i \leq m} Y_i \ \ \text{  and  } \ \ Y_{\min} = \min_{1 \leq i \leq m} Y_i.
\end{eqnarray}
For simplicity we also let $X_{\ell} = N/Y_{\ell}$ $(1 \leq \ell \leq m)$ and denote
\begin{eqnarray}
\label{WDI eqn5}
\widetilde{ X } = \prod_{i=1}^{m} X_i, \ \ X_{\max} = \max_{1 \leq i \leq m} X_i \ \ \text{  and  } \ \ X_{\min} = \min_{1 \leq i \leq m} X_i.
\end{eqnarray}
It follows from the construction that
\begin{eqnarray}
\label{construction1}
N^{\delta_1} \leq N^{\lambda} \leq Y_i \leq X_i \ \ (1 \leq i \leq m) \  \ \ \textnormal{  and  }  \  \ \ |Y_j/Y_i| \leq N^{\delta_2} \ \ (1 \leq i, j \leq m).
\end{eqnarray}
By H\"{o}lder's inequality we obtain
\begin{equation}
\label{ineq 1-1}
|\mathfrak{T}_{\mathbf{U}}(\alpha)|^{ 2^{d - 1} } \ll N^{\varepsilon} \widetilde{X}^{2^{d - 1} - 1} \sum_{\substack{ V <  x_i <  X_i \\ (1 \leq i \leq m)}}
| T_{\mathbf{x}} ({\alpha}) |^{2^{d - 1}},
\end{equation}
where
\begin{eqnarray}
\label{somesumTTT}
T_{\mathbf{x}} ({\alpha}) = \sum_{\substack{Y_j < y_j \leq  2 Y_j \\ (1 \leq j \leq m)}} \Lambda(\mathbf{y})   \psi(\mathbf{x} ; \mathbf{y})
e ( \alpha g( \mathbf{x}, \mathbf{y} ) ).
\end{eqnarray}
Here Remark \ref{remrange}, with $Y_j$, $y_j$ and (\ref{somesumTTT}) in place of $U_j$, $u_j$ and (\ref{somesumT}) respectively, applies.
Next we use a form of Weyl's inequality as in \cite[Lemma 11.1]{S} to bound $| T_{\mathbf{x}} (\alpha) |^{2^{d} - 1}$.
Given a subset $\mathcal{X} \subseteq \mathbb{R}^{m}$ we
denote $\mathcal{X}^D = \mathcal{X} - \mathcal{X} = \{ \mathbf{z} - \mathbf{z}' : \mathbf{z}, \mathbf{z}' \in \mathcal{X} \}$,
and also for any $\mathbf{z}_1, \ldots, \mathbf{z}_t \in \mathbb{R}^{m}$ we let
$$
\mathcal{X} (\mathbf{z}_1, \ldots, \mathbf{z}_t) = \medcap_{\epsilon_1 = 0}^1 \cdots \medcap_{\epsilon_t = 0}^1
(\mathcal{X} - \epsilon_1 \mathbf{z}_1 - \cdots - \epsilon_t \mathbf{z}_t).
$$
Let us define
\begin{eqnarray}
\mathcal{U} = (Y_1, 2 Y_1]   \times \cdots \times (Y_m, 2 Y_m]
\ \  \text{  and  } \  \ \mathcal{V} = (V, X_1)   \times \cdots \times (V, X_m).
\notag
\end{eqnarray}
Note when $N^{\lambda + \delta_2}/2 \leq Y_j < N^{\lambda + \delta_2}$, we replace $(Y_j, 2 Y_j]$ in the definition of $\mathcal{U}$ with $(Y_j, N^{\lambda + \delta_2}]$.

Let
\begin{eqnarray}
\label{defnF=g}
\mathcal{F}(\mathbf{y}) =  \alpha g( \mathbf{x}, \mathbf{y} ).
\end{eqnarray}
For each $t \in \mathbb{N}$ we denote
\begin{eqnarray}
\label{defnFd}
\mathcal{F}_{t}(\mathbf{y}_1, \ldots, \mathbf{y}_{t}) = \sum_{\epsilon_1 = 0}^1 \cdots \sum_{\epsilon_{t} = 0}^1 \,
(-1)^{\epsilon_1 + \cdots + \epsilon_{t} }
\mathcal{F}( \epsilon_1 \mathbf{y}_1 + \cdots + \epsilon_{t} \mathbf{y}_{t} ),
\end{eqnarray}
and $\mathcal{F}_0 = 0$ identically. We let
$\psi_{\mathbf{y}_1} (\mathbf{x};  \mathbf{z}) = \psi(\mathbf{x};  \mathbf{z}) \tsp \psi(\mathbf{x}; \mathbf{y}_1 +  \mathbf{z} )$
and recursively define
$$
\psi_{\mathbf{y}_1, \ldots, \mathbf{y}_{t}} (\mathbf{x};  \mathbf{z}) = \psi_{\mathbf{y}_1, \ldots, \mathbf{y}_{t - 1}} (\mathbf{x};  \mathbf{z}) \tsp  \psi_{\mathbf{y}_1, \ldots, \mathbf{y}_{t - 1}} (\mathbf{x};  \mathbf{y}_t + \mathbf{z} )
$$
for $t \geq 2$. We also define $\Lambda_{ \mathbf{y}_1, \ldots, \mathbf{y}_{t} }$ in a similar manner. By following the proof of
\cite[Lemma 11.1]{S}, while taking into account the weights, we obtain
\begin{eqnarray}
\notag
| T_{\mathbf{x}} ({\alpha}) |^{2^{t - 1} }
&\leq&  | \mathcal{U}^D |^{2^{t-1} - t} \sum_{\mathbf{y}_1 \in \mathcal{U}^D} \cdots
\sum_{\mathbf{y}_{t - 1} \in \mathcal{U}^D}
\\
\notag
&& \Big{|}  \sum_{\mathbf{y}_t \in \mathcal{U}(\mathbf{y}_1, \ldots, \mathbf{y}_{t - 1})  }
\Lambda_{\mathbf{y}_1, \ldots, \mathbf{y}_{t - 1}} (\mathbf{y}_t) \tsp  \psi_{\mathbf{y}_1, \ldots, \mathbf{y}_{t - 1}} (\mathbf{x};  \mathbf{y}_t)
\tsp e ( \mathcal{F}_{t}(\mathbf{y}_1, \ldots, \mathbf{y}_t) ) \Big{|}
\end{eqnarray}
for each $t \geq 1$. Then it follows by taking the square of the inequality (with $t = d-1$) and applying
the Cauchy-Schwarz inequality that
\begin{eqnarray}
\notag
| T_{\mathbf{x}} ({\alpha}) |^{2^{d - 1} }
&\leq&  | \mathcal{U}^D |^{2^{d - 1} - d} \sum_{\mathbf{y}_1 \in \mathcal{U}^D} \cdots
\sum_{\mathbf{y}_{d - 2} \in \mathcal{U}^D}
\\
\notag
&& \Big{|}  \sum_{\mathbf{z} \in \mathcal{U}(\mathbf{y}_1, \ldots, \mathbf{y}_{d - 2})  }
\Lambda_{\mathbf{y}_1, \ldots, \mathbf{y}_{d - 2}} (\mathbf{z}) \tsp  \psi_{\mathbf{y}_1, \ldots, \mathbf{y}_{d - 2}} (\mathbf{x};  \mathbf{z})
\tsp e ( \mathcal{F}_{d - 1}(\mathbf{y}_1, \ldots, \mathbf{y}_{d - 2}, \mathbf{z}) ) \Big{|}^2.
\end{eqnarray}
Given $\mathbf{z}, \mathbf{z}' \in \mathcal{U}(\mathbf{y}_1, \ldots, \mathbf{y}_{d - 2})$ let us set
$\mathbf{y}_{d - 1} = (\mathbf{z}' - \mathbf{z} ) \in \mathcal{U}(\mathbf{y}_1, \ldots, \mathbf{y}_{d - 2})^D$
and
$$
\mathbf{y}_{d} =  \mathbf{z} \in \mathcal{U}(\mathbf{y}_1, \ldots, \mathbf{y}_{d - 2}) \medcap (\mathcal{U}(\mathbf{y}_1, \ldots, \mathbf{y}_{d - 2}) -
\mathbf{y}_{d-1}) =  \mathcal{U}(\mathbf{y}_1, \ldots, \mathbf{y}_{d - 1}).
$$
Then it follows from (\ref{defnFd}) that
\begin{eqnarray}
&&
\mathcal{F}_{d-1}(\mathbf{y}_1, \ldots, \mathbf{y}_{d-2}, \mathbf{z}) - \mathcal{F}_{d-1}(\mathbf{y}_1, \ldots, \mathbf{y}_{d-2}, \mathbf{z}')
\notag
\\
&=&
\mathcal{F}_{d-1}(\mathbf{y}_1, \ldots, \mathbf{y}_{d-2}, \mathbf{y}_{d}) - \mathcal{F}_{d-1}(\mathbf{y}_1, \ldots, \mathbf{y}_{d-2}, \mathbf{y}_{d - 1} + \mathbf{y}_{d})
\notag
\\
&=& \mathcal{F}_{d}(\mathbf{y}_1, \ldots, \mathbf{y}_{d-1}, \mathbf{y}_{d}) - \mathcal{F}_{d-1}(\mathbf{y}_1, \ldots, \mathbf{y}_{d-2}, \mathbf{y}_{d - 1}).
\notag
\end{eqnarray}
Therefore, we obtain
\begin{eqnarray}
\label{ineq 1-2}
|T_{\mathbf{x}} ({\alpha}) |^{2^{d-1}}
&\leq&
| \mathcal{U}^D |^{2^{d - 1} - d}
\sum_{\mathbf{y}_1 \in \mathcal{U}^D} \cdots \sum_{\mathbf{y}_{d - 2} \in \mathcal{U}^D}
\ \sum_{\mathbf{y}_{d - 1} \in \mathcal{U}(\mathbf{y}_1, \ldots, \mathbf{y}_{d - 2})^D  }
\\
\notag
&&   \sum_{\mathbf{y}_{d} \in \mathcal{U}(\mathbf{y}_1, \ldots, \mathbf{y}_{d - 1})  }
\Lambda_{\mathbf{y}_1, \ldots, \mathbf{y}_{d - 1}} (\mathbf{y}_d)  \tsp
\psi_{\mathbf{y}_1, \ldots, \mathbf{y}_{d - 1}} (\mathbf{x}; \mathbf{y}_d) \cdot
\\
&& e ( \mathcal{F}_{d}(\mathbf{y}_1, \ldots, \mathbf{y}_{d}) - \mathcal{F}_{d-1}(\mathbf{y}_1, \ldots, \mathbf{y}_{d - 1}) ).
\notag
\end{eqnarray}
We note that $\mathcal{U}^D$, $\mathcal{U}(\mathbf{y}_1, \ldots, \mathbf{y}_{d - 2})^D$ are boxes contained in $(-Y_1, Y_1) \times \cdots \times (-Y_{m}, Y_{m})$, and $\mathcal{U}(\mathbf{y}_1, \ldots, \mathbf{y}_{d - 1})$ is a box contained in $(Y_1, 2Y_1] \times \cdots \times (Y_{m}, 2 Y_{m}]$.

By \cite[Lemma 11.4]{S} the polynomial $\mathcal{F}_{d}$ is the unique symmetric multilinear form associated to $\mathcal{F}^{[d]}$, the degree $d$ homogeneous portion (as a polynomial in $\mathbf{y}$) of $\mathcal{F}$, i.e. $\mathcal{F}_{d}$ satisfies
$$
\mathcal{F}_{d}(\mathbf{y}, \ldots, \mathbf{y}) = (-1)^d \,  d! \, \mathcal{F}^{[d]}(\mathbf{y}).
$$
Since $\mathcal{F}^{[d]}(\mathbf{y}) = \alpha g^{[2d]}(\mathbf{x};\mathbf{y})$,
it follows that $\mathcal{F}_{d}$  does not depend on the terms of $g$ with degrees less than $2d$.
Recalling the notation from (\ref{defng2dsymm}) and making use of \cite[Lemma 11.2]{S} and \cite[Lemma 11.4]{S}, it follows that
$$
\mathcal{F}_{d}(\mathbf{y}_1, \ldots, \mathbf{y}_d)
 =
(-1)^d d!  \alpha \sum_{\mathbf{j}} \sum_{\mathbf{k}} G_{ \mathbf{j}, \mathbf{k} } \tsp  x_{j_1} \cdots x_{j_{d}} \tsp y_{1, k_1} \cdots y_{{d}, k_{d}}
$$
and
$$
- \mathcal{F}_{d-1}(\mathbf{y}_1, \ldots, \mathbf{y}_{d - 1})
= \alpha \sum_{\mathbf{j}} \sum_{\mathbf{k}} G_{ \mathbf{j}, \mathbf{k} } \tsp  x_{j_1} \cdots x_{j_{d}} \tsp
\widetilde{\mathcal{H}}_{\mathbf{k}}(\mathbf{y}_1, \ldots, \mathbf{y}_{d-1})
+
\widetilde{\mathcal{H}}'_{\mathbf{x}}(\mathbf{y}_1, \ldots, \mathbf{y}_{d-1}),
$$
where the summations over $\mathbf{j}$ and $\mathbf{k}$ are as in (\ref{defng2dsymm}),
and $\widetilde{\mathcal{H}}_{\mathbf{k}}$ and $\widetilde{\mathcal{H}}'_{\mathbf{x}}$ are degree $d$ and $d-1$ homogeneous forms respectively,
both independent of $\mathbf{y}_{d}$.
We note that for fixed $\mathbf{y}_1, \ldots, \mathbf{y}_{d-1}$,
$\widetilde{\mathcal{H}}'_{\mathbf{x}}(\mathbf{y}_1, \ldots, \mathbf{y}_{d-1})$ is a degree $d-1$ homogeneous form in $\mathbf{x}$.
For simplicity let us denote
$$
\mathcal{H}_{\mathbf{k}}(\mathbf{y}_1, \ldots, \mathbf{y}_{d}) = (-1)^d d! \tsp y_{1, k_1} \cdots y_{{d}, k_{d}} + \widetilde{\mathcal{H}}_{\mathbf{k}}(\mathbf{y}_1, \ldots, \mathbf{y}_{d-1}).
$$

Let us write $\underline{\mathbf{y}} = (\mathbf{y}_1, \ldots, \mathbf{y}_{d})$.
We substitute the inequality (\ref{ineq 1-2}) into (\ref{ineq 1-1}), and we interchange the order of summation moving
the sum over $\mathbf{x}$ inside the sums over $\mathbf{y}_j$'s.
Then we apply H\"{o}lder's inequality to obtain
\begin{eqnarray}
\label{Hod1}
|\mathfrak{T}_{\mathbf{U}}(\alpha)|^{2^{d + d -2}} \ll
N^{\varepsilon (1 + 2^{d-1})}
 \widetilde{X}^{2^{d + d -2} - 2^{d -1} } \ \widetilde{Y}^{2^{d + d -2} - d} \
\sum_{\mathbf{y}_1} \cdots \sum_{\mathbf{y}_{d}}  | T_{ \underline{\mathbf{y}} } ({\alpha}) |^{2^{d - 1} },
\end{eqnarray}
where the range of summation of each $\mathbf{y}_j$ is the same as in (\ref{ineq 1-2}), and
$$
T_{ \underline{\mathbf{y}} } (\alpha) = \sum_{ \mathbf{x} \in \mathcal{V}}
\psi_{\mathbf{y}_1, \ldots, \mathbf{y}_{d - 1}} (\mathbf{x}; \mathbf{y}_d) \
e \left(  \alpha \sum_{\mathbf{j}} \sum_{\mathbf{k}} G_{\mathbf{j}, \mathbf{k} } \tsp  x_{j_1} \cdots x_{j_{d}} \tsp
\mathcal{H}_{\mathbf{k}}(\underline{\mathbf{y}})   +    \widetilde{\mathcal{H}}'_{\mathbf{x}}(\mathbf{y}_1, \ldots, \mathbf{y}_{d-1}) \right).
$$

Let us set $\widehat{\mathbf{y}} = (\mathbf{y}_1, \ldots, \mathbf{y}_{d - 1})$.
Let $\psi_{ \widehat{\mathbf{y}} }(\mathbf{x};  \mathbf{y}_d )
= \psi_{ \mathbf{y}_1, \ldots, \mathbf{y}_{d-1} }(\mathbf{x};  \mathbf{y}_d )$.
Similarly as before we let
$\psi_{\mathbf{x}_1; \widehat{\mathbf{y}}} (\mathbf{z};  \mathbf{y}_d) = \psi_{\widehat{\mathbf{y}}}(\mathbf{z};  \mathbf{y}_d)
\tsp \psi_{ \widehat{\mathbf{y}}}(\mathbf{x}_1 + \mathbf{z};  \mathbf{y}_d )$
and recursively define
$$
\psi_{\mathbf{x}_1, \ldots, \mathbf{x}_{t};  \widehat{\mathbf{y}}} (\mathbf{z};  \mathbf{y}_d)
= \psi_{\mathbf{x}_1, \ldots, \mathbf{x}_{t-1};  \widehat{\mathbf{y}}} (\mathbf{z};  \mathbf{y}_d) \tsp
\psi_{\mathbf{x}_1, \ldots, \mathbf{x}_{t-1};  \widehat{\mathbf{y}}} ( \mathbf{x}_{t} + \mathbf{z} ;  \mathbf{y}_d )
$$
for $t \geq 2$. We also set $\widehat{\mathbf{x}} = (\mathbf{x}_1, \ldots, \mathbf{x}_{d - 1})$ and $\underline{\mathbf{x}} = (\mathbf{x}_1, \ldots, \mathbf{x}_{d})$.
Let $\psi_{\widehat{\mathbf{x}} ; \widehat{\mathbf{y}} } (\mathbf{z}; \mathbf{y}_d) = \psi_{ \mathbf{x}_1, \ldots, \mathbf{x}_{d - 1} ; \widehat{\mathbf{y}} } (\mathbf{z}; \mathbf{y}_d)$. Applying the same differencing process as before to $T_{ \underline{\mathbf{y}} } (\alpha)$,
the inequality (\ref{Hod1}) becomes
\begin{eqnarray}
\label{ineq 2'}
|\mathfrak{T}_{\mathbf{U}}(\alpha)|^{2^{2 d -2}} &\ll&
N^{\varepsilon (1 + 2^{d-1})}
\widetilde{X}^{2^{2 d -2} - d } \widetilde{Y}^{2^{2 d -2} - d}
\sum_{\mathbf{y}_1} \cdots \sum_{\mathbf{y}_{d}} \sum_{\mathbf{x}_1 \in \mathcal{V}^D} \cdots \sum_{\mathbf{x}_{d - 2}  \in \mathcal{V}^D}
\\
\notag
&& \sum_{\mathbf{x}_{d - 1} \in \mathcal{V}(\mathbf{x}_1, \ldots, \mathbf{x}_{d-2})^D} \Big{|} \sum_{\mathbf{x}_{d} \in \mathcal{V}(\mathbf{x}_1, \ldots, \mathbf{x}_{d-1})}
\psi_{\widehat{\mathbf{x}} ; \widehat{\mathbf{y}} } (\mathbf{x}_d ; \mathbf{y}_d)
e(  \gamma(\underline{\mathbf{x}}; \underline{\mathbf{y}}))  \Big{|},
\end{eqnarray}
where
$$
\gamma(\underline{\mathbf{x}}; \underline{\mathbf{y}}) =
(-1)^d d! \alpha \sum_{\mathbf{j}} \sum_{\mathbf{k}} \tsp G_{ \mathbf{j}, \mathbf{k} }
\tsp x_{1, j_1} \cdots x_{d, j_{d}} \tsp  \mathcal{H}_{\mathbf{k}}( \underline{\mathbf{y}} ).
$$
We note that $\mathcal{V}^D$, $\mathcal{V}(\mathbf{x}_1, \ldots, \mathbf{x}_{d - 2})^D$ and $\mathcal{V}(\mathbf{x}_1, \ldots, \mathbf{x}_{d - 1})$
are boxes contained in $(-X_1, X_1) \times \cdots \times (-X_{m}, X_{m})$.
We now change the order of summation in (\ref{ineq 2'}), and bound the exponential sum
\begin{equation}
\label{sum 1}
\sum_{\mathbf{y}_{d} \in \mathcal{U}(\mathbf{y}_1, \ldots, \mathbf{y}_{d-1}) } \Big{|} \sum_{\mathbf{x}_{d} \in \mathcal{V}(\mathbf{x}_1, \ldots, \mathbf{x}_{d-1}) }
\psi_{ \widehat{\mathbf{x}} ; \widehat{\mathbf{y}} }  (\mathbf{x}_d ; \mathbf{y}_d)
e (  \gamma(\underline{\mathbf{x}}; \underline{\mathbf{y}})  )  \Big{|}.
\end{equation}
Recall the definitions of $\psi$  and $\varpi_j$ given in (\ref{defn of weight WDI}) and (\ref{defn of varpi}) respectively.
First we have $\psi_{ \widehat{\mathbf{x}} ; \widehat{\mathbf{y}} }  (\mathbf{x}_d ; \mathbf{y}_d)
= \prod_{\ell=1}^{m} \psi_{ \widehat{\mathbf{x}} ; \widehat{\mathbf{y}} ; \ell }  (x_{d, \ell} ; y_{d, \ell} )$,
where
\begin{eqnarray}
\psi_{ \widehat{\mathbf{x}} ; \widehat{\mathbf{y}} ; \ell }   (x_{d, \ell} ; y_{d, \ell} )
\notag
&=&
\prod_{ \substack{ 1 \leq j \leq d \\   0< i_1 < i_2 < \cdots < i_{j-1} < d } } \
\prod_{ \substack{ 1 \leq k \leq d \\   0< i'_1 < i'_2 < \cdots < i'_{k-1} < d } }
\\
&&
\notag
\omega \left( \frac{ (x_{i_1, \ell } + \cdots + x_{i_{j-1}, \ell} + x_{d, \ell}) (y_{i'_1, \ell } + \cdots + y_{i'_{k-1}, \ell} + y_{d, \ell})   }{N} - x_{0, \ell} \right) \cdot
\\
\notag
&& \mathbbm{1}_{[0, N/ ( y_{i'_1,  {\ell} } + \cdots + y_{i'_{k-1},  {\ell} } + y_{d, {\ell} } )] } \left(x_{i_1, \ell} + \cdots + x_{i_{j-1}, \ell} + x_{d, \ell} \right);
\end{eqnarray}
we interpret $x_{i_1, \ell } + \cdots + x_{i_{j-1}, \ell} + x_{d, \ell}$ as $x_{d, \ell}$ when $j=1$, and similarly
$y_{i'_1,  {\ell} } + \cdots + y_{i'_{k-1},  {\ell} } + y_{d, {\ell} }$ as $y_{d, \ell}$ when $k=1$.
Let $\mathcal{V}(\mathbf{x}_1, \ldots, \mathbf{x}_{d-1}) =
\prod_{\ell=1}^m  \mathcal{V}(\mathbf{x}_1, \ldots, \mathbf{x}_{d-1})^{(\ell)} \subseteq \mathbb{R}^m$,
where each $\mathcal{V}(\mathbf{x}_1, \ldots, \mathbf{x}_{d-1})^{(\ell)}$ is an interval contained in $(-X_{\ell}, X_{\ell})$.
With these notation we have
\begin{eqnarray}
\label{sum 1+usedtobe}
&&\sum_{\mathbf{y}_{d} \in \mathcal{U}(\mathbf{y}_1, \ldots, \mathbf{y}_{d - 1})} \Big{|} \sum_{\mathbf{x}_{d} \in \mathcal{V}(\mathbf{x}_1, \ldots, \mathbf{x}_{d-1}) } \psi_{ \widehat{\mathbf{x}} ; \widehat{\mathbf{y}}}   (\mathbf{x}_{d} ; \mathbf{y}_{d} )  e (  \gamma(\underline{\mathbf{x}}; \underline{\mathbf{y}})  )   \Big{|}
\\
\notag
&&\ll \sum_{\mathbf{y}_{d} \in \mathcal{U}(\mathbf{y}_1, \ldots, \mathbf{y}_{d - 1})} \prod_{\ell = 1}^{m}   \Big{|}
\sum_{x_{d, \ell} \in \mathcal{V}(\mathbf{x}_1, \ldots, \mathbf{x}_{d-1})^{(\ell)} }
\psi_{ \widehat{\mathbf{x}} ; \widehat{\mathbf{y}} ; \ell }   (x_{d, \ell} ; y_{d, \ell} )
e  (  \widetilde{\gamma}(\widehat{\mathbf{x}}, x_{d, \ell} \, \mathbf{e}_{\ell} ; \underline{\mathbf{y}} )  )  \Big{|},
\end{eqnarray}
where $\mathbf{e}_{\ell}$ is the $\ell$-th unit vector and $\widetilde{\gamma}$ is given by
$$
\widetilde{\gamma}(\underline{\mathbf{x}} ; \underline{\mathbf{y}} )
=(-1)^d  d! \alpha \sum_{ \mathbf{j} } \sum_{ \mathbf{k} } G_{ \mathbf{j}, \mathbf{k} } \tsp x_{1, j_1} \cdots x_{d, j_{d}} \tsp
\mathcal{H}_{\mathbf{k}}(\underline{\mathbf{y}}).
$$
Since the intersection of two intervals is still an interval (possibly the empty set),
the subset of $\mathcal{V}(\mathbf{x}_1, \ldots, \mathbf{x}_{d-1})^{(\ell)}$
for which the summand in the following sum
\begin{align}
\notag
\sum_{x_{d, \ell} \in \mathcal{V}(\mathbf{x}_1, \ldots, \mathbf{x}_{d-1})^{(\ell)}}
\prod_{ \substack{ 1 \leq j \leq d \\   0< i_1 < i_2 < \cdots < i_{j-1} < d } } \
&\prod_{ \substack{ 1 \leq k \leq d \\   0< i'_1 < i'_2 < \cdots < i'_{k-1} < d } }
\\
\notag
&\mathbbm{1}_{[0, N/ ( y_{i'_1,  {\ell} } + \cdots + y_{i'_{k-1},  {\ell} } + y_{d, {\ell} } )] } (x_{i_1, \ell} + \cdots + x_{i_{j-1}, \ell} + x_{d, \ell} )
\end{align}
is non-zero is an interval contained in $(-X_{\ell}, X_{\ell})$.
Given $z \in \mathbb{R}$ let $\| z \| = \min_{w \in \mathbb{Z}} |z - w|$.
Therefore, by partial summation we obtain
\begin{eqnarray}
\label{inequality17}
&&\Big{|}
\sum_{x_{d, \ell} \in \mathcal{V}(\mathbf{x}_1, \ldots, \mathbf{x}_{d-1})^{(\ell)} }
\psi_{ \widehat{\mathbf{x}} ; \widehat{\mathbf{y}} ; \ell }   (x_{d, \ell} ; y_{d, \ell} )
e  (  \widetilde{\gamma}(\widehat{\mathbf{x}}, x_{d, \ell} \, \mathbf{e}_{\ell} ; \underline{\mathbf{y}} )  )  \Big{|}
\\
&\ll&
\notag
\left(  \sup_{ x_{d, \ell} \in (-X, X) } \Omega (x_{d, \ell})
+
\int_{ -X_{\ell} }^{X_{\ell}} \Big{|} \frac{d \Omega}{dt}  (x_{d, \ell}) \Big{|} \, d x_{d, \ell} \right)  \cdot
\min \left( X_{\ell}, \|  \widetilde{\gamma}(\widehat{\mathbf{x}}, \mathbf{e}_{\ell} ; \underline{\mathbf{y}} )   \|^{-1}  \right),
\end{eqnarray}
where
\begin{eqnarray}
\Omega(t)
\notag
&=&
\prod_{ \substack{ 1 \leq j \leq d \\   0< i_1 < i_2 < \cdots < i_{j-1} < d } } \
\prod_{ \substack{ 1 \leq k \leq d \\   0< i'_1 < i'_2 < \cdots < i'_{k-1} < d } }
\\
&&
\notag
\omega \left( \frac{ (x_{i_1, \ell } + \cdots + x_{i_{j-1}, \ell} + t) (y_{i'_1, \ell } + \cdots + y_{i'_{k-1}, \ell} + y_{d, \ell})   }{N} - x_{0, \ell} \right).
\end{eqnarray}
Since
$$
\frac{ (|y_{i'_1,  {\ell} }| + \cdots + |y_{i'_{k-1},  {\ell} }| + |y_{d, {\ell} }| ) X_{\ell}  }{N}  \ll 1
$$
for any $\underline{\mathbf{y}}$  under consideration, it follows from (\ref{sum 1+usedtobe}) and (\ref{inequality17}) that
\begin{eqnarray}
\label{sum 1+}
&&\sum_{\mathbf{y}_{d} \in \mathcal{U}(\mathbf{y}_1, \ldots, \mathbf{y}_{d - 1})} \Big{|} \sum_{\mathbf{x}_{d} \in \mathcal{V}(\mathbf{x}_1, \ldots, \mathbf{x}_{d-1}) } \psi_{ \widehat{\mathbf{x}} ; \widehat{\mathbf{y}}}   (\mathbf{x}_{d} ; \mathbf{y}_{d} )  e (  \gamma(\underline{\mathbf{x}}; \underline{\mathbf{y}})  )   \Big{|}
\\
\notag
&&\ll \sum_{\mathbf{y}_{d} \in \mathcal{U}(\mathbf{y}_1, \ldots, \mathbf{y}_{d - 1})} \prod_{\ell = 1}^{m} \min \left( X_{\ell}, \|  \widetilde{\gamma}(\widehat{\mathbf{x}}, \mathbf{e}_{\ell} ; \underline{\mathbf{y}} )   \|^{-1}   \right),
\end{eqnarray}
where the implicit constant is independent of $\widehat{\mathbf{x}}$ and $\underline{\mathbf{y}}$.

For a real number $z$ we define its fractional part to be $\{z\} = z - \max_{ \substack{w \leq z \\ w \in \mathbb{Z}}} w$. Given $\mathbf{c} = (c_1, \ldots, c_{m}) \in \mathbb{Z}^{m}$ with $0 \leq c_{\ell} < X_{\ell}$ $(1 \leq \ell \leq m)$,
we let $\mathcal{R}(\widehat{\mathbf{x}}; \widehat{\mathbf{y}}; \mathbf{c})$ be the set of $\mathbf{y}_{d} \in \mathcal{U}(\mathbf{y}_1, \ldots, \mathbf{y}_{d - 1})$ satisfying
$$
\frac{c_{\ell}}{X_{\ell}} \leq \{  \widetilde{\gamma}(\widehat{\mathbf{x}}, \mathbf{e}_{\ell} ; \widehat{\mathbf{y}}, \mathbf{y}_{d} )   \} <
\frac{c_{\ell} + 1}{ X_{\ell} }
\ \ (1 \leq \ell \leq m).
$$
Then we obtain that the right hand side of (\ref{sum 1+}) is bounded by
$$
\ll \sum_{\mathbf{c}} \# \mathcal{R}(\widehat{\mathbf{x}}; \widehat{\mathbf{y}}; \mathbf{c})  \cdot
\prod_{\ell=1}^{m} \min \left( X_{\ell}, \ \max \left(  \frac{X_{\ell}}{c_{\ell} }, \ \frac{X_{\ell}}{ X_{\ell} - c_{\ell} - 1}  \right)   \right),
$$
where the summation is over all integral vectors $\mathbf{c}$ with $0 \leq c_{\ell} < X_{\ell}$ $(1 \leq \ell \leq m)$.
Next we obtain a bound for $ \# \mathcal{R}(\widehat{\mathbf{x}}; \widehat{\mathbf{y}}; \mathbf{c}) $.
We define the multilinear form
$$
\Gamma (\underline{\mathbf{x}} ; \underline{\mathbf{y}} ) = (d!)^2 \sum_{\mathbf{j}} \sum_{\mathbf{k}}
G_{ \mathbf{j}, \mathbf{k} } \tsp x_{1, j_1} \cdots x_{d, j_{d}} \tsp y_{1, k_1} \cdots y_{d, k_{d}}.
$$
When $\mathbf{x} = \mathbf{x}_1 = \cdots = \mathbf{x}_{d-1}$ and $\mathbf{y} = \mathbf{y}_1 = \cdots = \mathbf{y}_{d}$,
the multilinear form $\Gamma$ becomes
\begin{eqnarray}
\label{derivcond}
\Gamma( ( \mathbf{x}, \ldots , \mathbf{x}, \mathbf{e}_{\ell}) ; ( \mathbf{y}, \ldots , \mathbf{y})   ) = \frac{(d!)^2}{d} \cdot  \frac{\partial g^{[2d]}}{ \partial x_{\ell}}  (\mathbf{x};\mathbf{y})  \ \ (1 \leq \ell \leq m).
\end{eqnarray}
If $\#\mathcal{R}(\widehat{\mathbf{x}}; \widehat{\mathbf{y}}; \mathbf{c}) = 0$, then there is nothing to prove. Thus we suppose
$\#\mathcal{R}(\widehat{\mathbf{x}}; \widehat{\mathbf{y}}; \mathbf{c}) > 0$ and take $\mathbf{u} \in \mathcal{R}(\widehat{\mathbf{x}}; \widehat{\mathbf{y}}; \mathbf{c})$. Then for any $\mathbf{v} \in \mathcal{R}(\widehat{\mathbf{x}}; \widehat{\mathbf{y}}; \mathbf{c})$,  we have
$$
\mathbf{u} - \mathbf{v} \in \mathcal{U}(\mathbf{y}_1, \ldots, \mathbf{y}_{d - 1})^D \subseteq (-Y_1, Y_1) \times \cdots \times (- Y_{m}, Y_{m})
$$
and
\begin{eqnarray}
\label{counting 1''}
\|  \widetilde{\gamma}(\widehat{\mathbf{x}}, \mathbf{e}_{\ell} ; \widehat{\mathbf{y}}, \mathbf{u}  ) - \widetilde{\gamma}(\widehat{\mathbf{x}}, \mathbf{e}_{\ell} ; \widehat{\mathbf{y}}, \mathbf{v}  )   \| <  X_{\ell}^{-1} \ \ (1 \leq \ell \leq m).
\end{eqnarray}
Let $M'( \widehat{\mathbf{x}} ; \widehat{\mathbf{y}} )$ be the number of integral vectors $\mathbf{y} \in (-Y_1, Y_1) \times \cdots \times (- Y_{m}, Y_{m})$ such that
$$
\| \alpha \Gamma (  \widehat{\mathbf{x}}, \mathbf{e}_{\ell} ;  \widehat{\mathbf{y}}, \mathbf{y} )  \| < X_{\ell}^{-1} \ \ (1 \leq \ell \leq m).
$$
Since
$$
\widetilde{\gamma}(\widehat{\mathbf{x}}, \mathbf{e}_{\ell} ; \widehat{\mathbf{y}}, \mathbf{u}  ) - \widetilde{\gamma}(\widehat{\mathbf{x}}, \mathbf{e}_{\ell} ; \widehat{\mathbf{y}}, \mathbf{v}  ) = \alpha \Gamma( \widehat{\mathbf{x}}, \mathbf{e}_{\ell} ; \widehat{\mathbf{y}}, \mathbf{u} - \mathbf{v} ),
$$
it follows from (\ref{counting 1''}) that the vector $\mathbf{u} - \mathbf{v}$ is counted by $M'( \widehat{\mathbf{x}} ; \widehat{\mathbf{y}} )$
for all $\mathbf{v} \in \mathcal{R}(\widehat{\mathbf{x}}; \widehat{\mathbf{y}}; \mathbf{c})$; therefore, we have
$$
\# \mathcal{R}(\widehat{\mathbf{x}}; \widehat{\mathbf{y}}; \mathbf{c}) \leq M'( \widehat{\mathbf{x}} ; \widehat{\mathbf{y}} )
$$
for any $\mathbf{c}$ under consideration. This yields the bound
\begin{eqnarray}
\label{ineq 1-3}
\sum_{\mathbf{y}_{d} \in \mathcal{U}(\mathbf{y}_1, \ldots, \mathbf{y}_{d - 1})} \Big{|} \sum_{ \mathbf{x}_{d} \in \mathcal{V}(\mathbf{x}_1, \ldots, \mathbf{x}_{d-1}) } \psi_{ \widehat{\mathbf{x}} ; \widehat{\mathbf{y}}}   (\mathbf{x}_{d} ; \mathbf{y}_{d} )   e (  \gamma(\underline{\mathbf{x}}; \underline{\mathbf{y}})   )   \Big{|}
\ll M'( \widehat{\mathbf{x}} ; \widehat{\mathbf{y}} ) \prod_{\ell = 1}^{m} X_{\ell} \log X_{\ell}.
\end{eqnarray}
Let us define $M(\alpha; \mathbf{X}; \mathbf{Y}; P_1, \ldots, P_{m})$ to be the number of integral vectors
$$
\widehat{\mathbf{x}} \in \left( (-X_1, X_1) \times  \cdots \times (-X_{m}, X_{m}) \right)^{d-1}
\ \ \text{  and  }  \  \  \
\underline{\mathbf{y}}\in \left( (-Y_1, Y_1) \times  \cdots \times (-Y_{m}, Y_{m}) \right)^{d}
$$
such that
$$
\| \alpha \Gamma (\widehat{\mathbf{x}}, \mathbf{e}_{\ell} ; \underline{ \mathbf{y} }  )  \| < P_{\ell} \ \ (1 \leq \ell \leq m).
$$
By substituting the inequality (\ref{ineq 1-3}) into (\ref{ineq 2'}), we obtain
\begin{eqnarray}
\label{exp bound with count 1}
|\mathfrak{T}_{\mathbf{U}}({\alpha})|^{2^{2 d  -2}}
\ll
N^{\varepsilon(2 + 2^{d-1})}
\widetilde{X}^{2^{2 d  -2} - d + 1}  \widetilde{Y}^{2^{2 d  -2} - d}
M({\alpha}; \mathbf{X}; \mathbf{Y}; X_1^{-1}, \ldots, X_{m}^{-1}).
\end{eqnarray}

The following lemma on geometry of numbers was obtained in \cite{SS}; this is a slightly more general version of \cite[Lemma 12.6]{D1}.
\begin{lem} \cite[Lemma 2.4]{SS}
\label{shrink lemma}
Let $L_1, \ldots, L_s$ be symmetric linear forms given by $L_i = \gamma_{i,1}u_1 + \cdots + \gamma_{i,s} u_s$
$(1 \leq i \leq s)$, i.e. such that $\gamma_{i,j} = \gamma_{j,i}$ $(1 \leq i, j \leq s)$. Let $a_1, \ldots, a_s >1$
be real numbers. We denote by $\mathfrak{U}(Z)$ the number of integer solutions $u_1, \ldots, u_{2s}$ to the system of inequalities
$$
|u_i| < a_i Z \ \ (1 \leq i \leq s) \ \ \text{  and  } \ \ |L_i - u_{s + i}| < a_i^{-1} Z \ \ (1 \leq i \leq s).
$$
Then for $0 < Z_1 \leq Z_2 \leq 1$ we have
$$
\frac{\mathfrak{U}(Z_2)}{\mathfrak{U}(Z_1)} \ll \left( \frac{Z_2}{Z_1} \right)^{s},
$$
where the implicit constant depends only on $s$.
\end{lem}
Let
$$
\mathcal{T} = \max_{1 \leq j \leq m} \  \frac{X_{j}^{-1} }{  Y_{j}^{-1}} = \max_{1 \leq j \leq m} \ \frac{Y_{j} }{  X_{j}} \leq 1.
$$
Let $0 < Q_1, Q_2 \leq 1$ to be chosen later.
First by applying Lemma \ref{shrink lemma} $(d-1)$-times to $M(\alpha; \mathbf{X}; \mathbf{Y}; X_1^{-1}, \ldots, X_{m}^{-1})$,
we obtain
\begin{eqnarray}
\notag
M(\alpha; \mathbf{X}; \mathbf{Y}; X_1^{-1}, \ldots, X_{m}^{-1})
&\ll&
\notag
Q_1^{-m (d - 1)}
M(\alpha; Q_1 \mathbf{X};  \mathbf{Y}; Q_1^{d- 1}  X_1^{-1},
\ldots, Q_1^{d- 1}  X_{m}^{-1})
\\
&\ll&
Q_1^{-m (d - 1)} M(\alpha; Q_1 \mathbf{X};  \mathbf{Y}; Q_1^{d- 1} \mathcal{T} Y_1^{-1},
\ldots, Q_1^{d- 1} \mathcal{T} Y_{m}^{-1}),
\notag
\end{eqnarray}
where the second inequality is obtained by noting that $ X_{\ell}^{-1} \leq \mathcal{T} Y_{\ell}^{-1}$ $(1 \leq \ell \leq m)$.
Next we apply Lemma \ref{shrink lemma} $d$-times to the final expression in the above inequality, and obtain
\begin{eqnarray}
\label{ineq 1-4}
&&M(\alpha; \mathbf{X}; \mathbf{Y}; X_1^{-1}, \ldots, X_{m}^{-1})
\\
&\ll&
\notag
Q_1^{-m (d - 1)} Q_2^{-m d}
M(\alpha; Q_1 \mathbf{X}; Q_2 \mathbf{Y}; Q_1^{d- 1} Q_2^{d} \mathcal{T}  Y_1^{-1},
\ldots, Q_1^{d- 1} Q_2^{d} \mathcal{T}  Y_{m}^{-1}).
\end{eqnarray}
With this estimate we obtain the following lemma. For simplicity let us denote
\begin{equation}
\label{V notation}
V^*_1 = V^{*}_{g^{[2d]}, 1}.
\end{equation}
\begin{lem}
\label{6.2}
Let $0 < \theta < 1$.
Let $\varepsilon > 0$ be sufficiently small.
Then for $N$ sufficiently large, at least one of the following alternatives holds:

\textnormal{i)} One has the upper bound
\begin{eqnarray}
|\mathfrak{T}_{\mathbf{U}}(\alpha)|
&\ll&
\notag
N^{m - (\theta \delta_1 - \delta_2)  \frac{\textnormal{codim} \tsp V^*_{1}}{2^{2 d -2}} + \varepsilon}.
\end{eqnarray}

\textnormal{ii)}
There exist $1 \leq q \leq N^{\frac{2d-1}{2} \theta }$ and $a \in \mathbb{Z}$ with
$\gcd( {a}, q) = 1$ such that
\begin{eqnarray}
| q \alpha - a | \leq N^{-d + \delta_2 + \frac{2d-1}{2} \theta }.
\notag
\end{eqnarray}
\end{lem}
\begin{proof}
Let us set
$$
\theta' = \theta -  \varepsilon \frac{ 2^{2 d -2} }{2 \delta_1 \, \textnormal{codim} \tsp V^*_{1} }.
$$
Since the inequality in i) holds trivially if $(\theta' \delta_1 - \delta_2) \leq 0$, we assume $\theta' > \delta_2 / \delta_1$. Then
from
\begin{eqnarray}
\label{6.22''}
Y_{\max} \geq N^{\delta_1} \ \  \text{  and  } \  \  1 \leq Y_{\max}/Y_{\min} \leq N^{\delta_2},
\end{eqnarray}
it follows that $Y_{\max}^{\vartheta' - 1} Y_{\min} \geq 1$ for $N$ sufficiently large with respect to $\delta_1$ and $\delta_2$.
Consider the affine variety $\mathcal{Z} \subseteq \mathbb{A}^{m(2 d - 1)}_{\mathbb{C}}$ defined by
$$
\mathcal{Z} = \left\{  \widehat{\mathbf{x}} \in \mathbb{C}^{m(d-1)},  \underline{\mathbf{y}} \in \mathbb{C}^{m d} : \Gamma( \widehat{\mathbf{x}}, \mathbf{e}_{\ell} ; \underline{\mathbf{y}}  ) = 0 \ \ (1 \leq \ell \leq m) \right\}.
$$
Let us define
\begin{eqnarray}
\notag
&&\mathcal{N}(\mathcal{Z} ; Q_1 \mathbf{X} ; Q_2 \mathbf{Y}  )
\\
&=&
\notag
\left\{
(\widehat{\mathbf{x}}, \underline{\mathbf{y}}) \in \mathbb{Z}^{m(2 d  - 1)} \cap \mathcal{Z} :
\begin{array}{l}
- Q_1 X_i \leq  x_{t, i} \leq Q_1 X_i \ \ (1 \leq t \leq d-1, 1 \leq i \leq m) \\
- Q_2 Y_j \leq  y_{t', j} \leq Q_2 Y_j  \ \ (1 \leq t' \leq d, 1 \leq j \leq m)
\end{array}{}
\right\}.
\end{eqnarray}
In this proof, we set $Q_1 = X^{-1}_{\max} Y^{\theta'}_{\max}$ and $Q_2 = Y^{\theta' - 1}_{\max}$.

Suppose we have that every point counted by $M(\alpha; Q_1 \mathbf{X}; Q_2 \mathbf{Y}; Q_1^{d - 1} Q_2^{d} \mathcal{T} Y_1^{-1},
\ldots, Q_1^{d- 1} Q_2^{d} \mathcal{T} Y_{m}^{-1})$ is contained in $\mathcal{N}(\mathcal{Z} ; Q_1 \mathbf{X} ; Q_2 \mathbf{Y}  )$.
Let us dissect the region given by $\mathbf{x}_{t} \in [-Q_1 X_1, Q_1 X_1] \times \cdots \times [-Q_1 X_{m}, Q_1 X_{m}]$ $(1 \leq t \leq d-1)$ and $\mathbf{y}_{t'} \in [-Q_2 Y_1, Q_2 Y_1] \times \cdots \times [-Q_2 Y_{m}, Q_2 Y_{m}]$ $(1 \leq t' \leq d)$ into boxes whose sides are closed intervals of lengths $Y_{\max}^{\theta' - 1}Y_{\min}$ (at the boundaries we allow overlaps which results in a slight overcount). Then the number of these boxes is bounded by
\begin{equation}
\label{ineq 3}
\ll \left(  \prod_{i=1}^{m} \frac{ Y_{\max}^{\theta'} X_{\max}^{- 1} X_{i}}{Y_{\max}^{\theta' - 1}Y_{\min} }  \right)^{d-1}
\left(  \prod_{j=1}^{m} \frac{  Y_{\max}^{\theta' - 1} Y_{j}}{ Y_{\max}^{\theta' - 1}  Y_{\min}}  \right)^{d}.
\end{equation}
Here note we have
$$
Y_{\max}^{\theta'}  X^{-1}_{\max} X_i = Y_{\max}^{\theta'}  (N /Y_{\min})^{-1} (N/Y_i) = Y_{\max}^{\theta'}  Y_{\min} Y^{-1}_i \geq Y_{\max}^{\theta' - 1} Y_{\min}.
$$
We apply a linear transformation to each of these boxes moving the center of the box to the origin. Then we apply \cite[(3.1)]{Bro};
this bound is independent of the coefficients of the polynomials defining the affine variety (depending only on the dimension and the degree), hence it is uniform in the shift.
Therefore, we obtain from (\ref{ineq 3}) that
\begin{eqnarray}
\label{ineq 4 1}
\# \mathcal{N}(\mathcal{Z} ; Q_1 \mathbf{X} ; Q_2 \mathbf{Y}  ) \ll \left(  \prod_{i=1}^{m} \frac{  X_{\max}^{- 1} X_{i}}{Y_{\max}^{- 1}Y_{\min} }  \right)^{d-1} \left(  \prod_{j=1}^{m} \frac{Y_{j}}{Y_{\min}}  \right)^{d}   \left( Y_{\max}^{\theta' - 1}Y_{\min} \right)^{\dim \mathcal{Z}}.
\end{eqnarray}
Thus it follows from (\ref{exp bound with count 1}), (\ref{ineq 1-4}) and (\ref{ineq 4 1}) that
\begin{eqnarray}
\label{6.23}
|\mathfrak{T}_{\mathbf{U}}(\alpha)|^{2^{2 d -2}}
&\ll& N^{\varepsilon (2 + 2^{d-1})} \widetilde{X}^{2^{2 d -2} - d + 1} \widetilde{Y}^{2^{2 d -2} - d}
X_{\max}^{m (d - 1)} Y_{\max}^{m d} Y_{\max}^{- m \theta' (2d - 1)}
\\
\notag
&\cdot&
\left(  \prod_{i=1}^{m} \frac{X_{\max}^{- 1} X_{i}}{Y_{\max}^{- 1}Y_{\min} }  \right)^{d-1} \left(  \prod_{j=1}^{m} \frac{Y_{j}}{Y_{\min}}  \right)^{d}
\left( Y_{\max}^{\vartheta' - 1}Y_{\min} \right)^{\dim \mathcal{Z}}.
\notag
\end{eqnarray}

Let us define $\mathcal{D} \subseteq \mathbb{A}_{\mathbb{C}}^{m(2 d  -1 )}$ by
\begin{eqnarray}
\mathcal{D} = \{  (\widehat{\mathbf{x}}, \underline{\mathbf{y}}) \in \mathbb{C}^{m(2 d - 1)} :  \mathbf{x}_1 = \cdots = \mathbf{x}_{d-1},
\mathbf{y}_1 = \cdots = \mathbf{y}_{d} \}.
\end{eqnarray}
Then recalling (\ref{derivcond}) and the definition (\ref{singlocbhmg}), we have
\begin{equation}
\label{dimensionsss}
\dim V^{*}_{1} = \dim ( \mathcal{Z} \cap \mathcal{D} ) \geq \dim \mathcal{Z} -  (d - 2) m - (d - 1) m.
\end{equation}
Since we have $X_{\ell} = N/Y_{\ell}$ $(1 \leq \ell \leq m)$, (\ref{construction1}), (\ref{6.22''}) and (\ref{dimensionsss}), the right hand side of the inequality (\ref{6.23}) becomes
\begin{eqnarray}
&=& N^{\varepsilon (2 + 2^{d-1})} N^{m (2^{2 d -2})}
Y_{\max}^{-  m \theta' (d - 1)}  Y_{\max}^{- m (\theta' - 1) d}
\
(Y_{\max}^{-1} Y_{\min})^{-  m (d-1)} Y_{\min}^{-  m d}   \left( Y_{\max}^{\theta' - 1}Y_{\min} \right)^{\dim \mathcal{Z}}
\notag
\\
&=&  N^{\varepsilon (2 + 2^{d-1})} N^{ m (2^{2 d -2})}
\left( Y_{\max}^{\theta' - 1}Y_{\min} \right)^{- m (d - 1) -  m d +  \dim \mathcal{Z} }
\notag
\\
&\leq&  N^{\varepsilon (2 + 2^{d-1})}  N^{ m (2^{2 d -2})}
\left( Y_{\max}^{\theta' - 1} Y_{\min} \right)^{\dim V^*_{1} - 2 m}
\notag
\\
&\leq&
\notag
 N^{\varepsilon (2 + 2^{d-1})}  N^{m (2^{2 d -2})} N^{- \delta_1 \theta' \textnormal{codim} \tsp V^{*}_{1} } N^{\delta_2 \textnormal{codim} \tsp V^{*}_{1} },
\end{eqnarray}
and the estimate in i) follows immediately.

On the other hand, suppose there exists $(\widehat{\mathbf{x}}, \underline{\mathbf{y}})$ counted by
$M(\alpha;$  $Q_1 \mathbf{X};$ $ Q_2 \mathbf{Y}; Q_1^{d - 1} Q_2^{d} \mathcal{T} Y_1^{-1},$
$\ldots, Q_1^{d - 1} Q_2^{d} \mathcal{T} Y_{m}^{-1})$ which is not contained in $\mathcal{N}(\mathcal{Z} ; Q_1 \mathbf{X} ; Q_2 \mathbf{Y}  )$, i.e.
there exists $1 \leq \ell_0 \leq m$ such that
$$
\Gamma( \widehat{\mathbf{x}}, \mathbf{e}_{\ell_0} ; \underline{\mathbf{y}})  \not = 0.
$$
Let us write
$$
\alpha \Gamma( \widehat{\mathbf{x}}, \mathbf{e}_{\ell_0} ; \underline{\mathbf{y}}  ) = {a}_{\ell_0} + {\xi}_{\ell_0},
$$
where $a_{\ell_0} \in \mathbb{Z}$ and $|\xi_{\ell_0} | <  Q_1^{d - 1} Q_2^{d} \mathcal{T} Y_{\ell_0}^{-1}$.
Let $q$ be the absolute value of $\Gamma( \widehat{\mathbf{x}}, \mathbf{e}_{\ell_0} ; \underline{\mathbf{y}})$.
Then using the fact that $G_{\mathbf{j}, \mathbf{k}} = 0$ unless $(j_1, \ldots, j_d)$ is a permutation of $(k_1, \ldots, k_d)$,
we obtain
\begin{eqnarray}
\notag
1 \leq q \ll Q_1^{d - 1} Q_2^{d}  N^{d-1} Y_{\max}
= {Y_{\max}}^{\theta' (2d -1)} N^{d-1} X_{\max}^{-(d-1)} Y_{\max}^{-(d-1)} \leq {Y_{\max}}^{\theta' (2d -1)}. 
\end{eqnarray}
Since $Y_{\max}\leq N^{1/2}$ and $\theta' < \theta$, it follows that $1 \leq q  \leq N^{ \frac{2d-1}{2} \theta}$ for $N$ sufficiently large.
We also have the estimate
\begin{eqnarray}
| \xi_{\ell_0} | &<& Q_1^{d - 1} Q_2^{d} \ \mathcal{T} \max_{1 \leq \ell \leq m } Y_{\ell}^{-1}
\\
&=& X_{\max}^{-(d-1)} \tsp Y_{\max}^{-d} \ {Y_{\max}}^{\theta' (2 d  -1)} \tsp \mathcal{T} \ Y_{\min}^{-1}
\notag
\\
&\leq&
X_{\max}^{-(d-1)}  \left( \frac{N}{ X_{\min} } \right)^{-d}  {Y_{\max}}^{\theta' (2 d -1)} \frac{Y_{\max}}{X_{\min}} \tsp Y_{\min}^{-1}
\notag
\\
&\leq& N^{-d} N^{\delta_2} \tsp {Y_{\max}}^{\theta' (2 d -1)} 
\notag
\\
&<& N^{-d} N^{\delta_2} \tsp N^{\theta \frac{2 d -1}{2}}.
\notag
\end{eqnarray}
Therefore, we have obtained the statement in ii).
\end{proof}
Finally, we set $\theta$ to satisfy
\begin{eqnarray}
\vartheta_0 = \delta_2 + \frac{2d-1}{2} \theta.
\end{eqnarray}
Recall (\ref{defnC_1}), (\ref{delta2}), (\ref{codim ineq1}) and  (\ref{V notation}).
Then we have
\begin{eqnarray}
(\theta \delta_1 - \delta_2)  \frac{\textnormal{codim} \tsp V^*_{1}}{2^{2 d -2}}
\geq
\left( \frac{2 \vartheta_0 \delta_1}{2d-1} -  \left( 1 + \frac{ 2 \delta_1 }{2d-1}  \right)  \frac{1 - 2 \delta_1}{H}  \right) \frac{2}{2^{2d}} \left(  \frac{ \textnormal{codim} \tsp V_F^* }{3H} -
\frac{H - \frac13}{H} \mathcal{C}_0 \right),
\notag
\end{eqnarray}
and  we obtain Proposition \ref{main prop} for this case from (\ref{case 4 ineq 1}), (\ref{maxlambda}), (\ref{reduced sum 1}), (\ref{WDI eqn2}) and Lemma \ref{6.2}.
\begin{rem}
The argument here is based on the work of Birch \cite{B}, and we have taken the format which is a combination of the analogous components in \cite{DS} and \cite{SS}.
\end{rem}

\subsection{$\lambda \geq 1/2$ and  $\lambda + \delta_2 \leq 1 - \delta_1$}
\label{WD2}
In this case, we also consider the same exponential sum as in (\ref{WDI eqn1}). First we follow the same
steps as in the previous case, and obtain (\ref{WDI eqn2}) with
\begin{eqnarray}
\mathfrak{T}_{\mathbf{U}}(\alpha) = \sum_{\substack{ U_j < u_j \leq 2 U_j \\ (1 \leq j \leq m)  } } \
\sum_{\substack{ V <  v_i \leq  N/u_i  \\ (1 \leq i \leq m)  } }
\Lambda(\mathbf{u}) \nu_3(\mathbf{v})
\prod_{\ell=1}^{m} \varpi_{\ell}(u_{\ell} v_{\ell})
\cdot  e( \alpha g(\mathbf{u} , \mathbf{v}) ),
\label{somesumTTTTTTTTTT}
\end{eqnarray}
where
$$
g(\mathbf{u}, \mathbf{v}) = \mathfrak{d}_{j_0}( \mathbf{u}, \mathbf{v} );
$$
this is the same as in (\ref{somesumT}) except a very slight notational modification.
We note that here Remark \ref{remrange}, with (\ref{somesumTTTTTTTTTT}) in place of (\ref{somesumT}), applies.
It follows from (\ref{codim dj0}) that
\begin{eqnarray}
\textnormal{codim} \tsp V^*_{g^{[2d]},1} = \textnormal{codim} \tsp V^*_{ \mathfrak{d}_{j_0}^{[2d]},1} \geq \frac{1}{2} \cdot \frac{\mathcal{C}_1 - (H -1)\mathcal{C}_0}{ H }.
\end{eqnarray}
Next we let
$\mathbf{x} = \mathbf{u}$ and $\mathbf{y} = \mathbf{v}$;
the labels have been switched compared to the previous case, however $\mathbf{x}$ still corresponds to the longer sum and $\mathbf{y}$ to the shorter sum.
We also let $X_{i} = U_i$ $(1 \leq i \leq m)$.
With these notation, instead of (\ref{WDI eqn3}), we have
\begin{eqnarray}
\mathfrak{T}_{\mathbf{U}}(\alpha) = \sum_{\substack{ X_i < x_i \leq 2 X_i \\ (1 \leq i \leq  m)  } } \Lambda(\mathbf{x})
\sum_{\substack{ V <  y_j <  N/X_j  \\ (1 \leq j \leq m)  } }  \nu_3(\mathbf{y}) \psi(\mathbf{x}; \mathbf{y})   e( \alpha g( \mathbf{x}, \mathbf{y} ) ),
\notag
\end{eqnarray}
where
\begin{eqnarray}
\label{defn of weight WDI'}
\psi(\mathbf{x}; \mathbf{y}) = \prod_{\ell = 1}^m \varpi_{\ell}(x_{\ell} y_{\ell})  \mathbbm{1}_{[0, N/x_{\ell}] }(y_{\ell}).
\end{eqnarray}
Given any $x_{\ell}, y_{\ell} \in \mathbb{R}_{> 0}$ it is clear that $\mathbbm{1}_{[0, N/x_{\ell}] }(y_{\ell}) = \mathbbm{1}_{[0, N/y_{\ell}] }(x_{\ell})$.
Thus we see that the definition of $\psi(\mathbf{x}; \mathbf{y})$ in (\ref{defn of weight WDI'}) is
in fact identical to that in (\ref{defn of weight WDI}).
We also let $Y_{\ell} = N/X_{\ell}$ $(1 \leq \ell \leq m)$ and use the same notation as in (\ref{WDI eqn4}) and (\ref{WDI eqn5}).
It follows from the construction that $X_i < N^{\lambda + \delta_2} \leq  N^{1 - \delta_1}$ $(1 \leq i \leq m)$
and $|X_j/X_i| \leq N^{\delta_2}$ $(1 \leq i, j \leq m)$. Therefore, we have
\begin{eqnarray}
N^{\delta_1} \leq  Y_i \leq X_i \ \ (1 \leq i \leq m) \  \ \ \textnormal{  and  }  \  \ \ |Y_j/Y_i| \leq N^{\delta_2} \ \ (1 \leq i, j \leq m).
\end{eqnarray}
Then by H\"{o}lder's inequality we obtain
\begin{equation}
\label{ineq 1}
|\mathfrak{T}_{\mathbf{U}}({\alpha})|^{ 2^{d - 1} } \ll N^{\varepsilon} {\widetilde{X}}^{2^{d - 1} - 1} \sum_{\substack{ X_i  < x_i \leq 2 X_i \\ (1 \leq i \leq m)}}
| T_{\mathbf{x}} ({\alpha}) |^{2^{d - 1}},
\end{equation}
where
\begin{eqnarray}
\notag
T_{\mathbf{x}} ({\alpha}) = \sum_{\substack{ V <  y_j <  Y_j \\ (1 \leq j \leq m)}} \nu_3(\mathbf{y}) \psi(\mathbf{x} ; \mathbf{y})
e ( \alpha g( \mathbf{x}, \mathbf{y} ) ).
\end{eqnarray}
Here Remark \ref{remrange}, with $X_i$, $x_i$ and (\ref{ineq 1}) in place of $U_j$, $u_j$ and (\ref{somesumT}) respectively, applies.
Let
\begin{eqnarray}
\mathcal{U} =  (V,  Y_1)   \times \cdots \times (V,  Y_m) \ \  \text{  and  } \  \  \mathcal{V} = (X_1,  2 X_1]   \times \cdots \times (X_m,  2 X_m].
\end{eqnarray}
Note when $N^{\lambda + \delta_2}/2 \leq X_i < N^{\lambda + \delta_2}$, we replace $(X_i, 2 X_i]$ in the definition of $\mathcal{V}$ with $(X_i, N^{\lambda + \delta_2}]$. From here on, with this set-up the proof can be carried out in essentially the same manner as in the previous case. As it only requires very minor modifications we leave the remaining details to the reader, and we conclude the proof of Proposition \ref{main prop} for this case. This completes the proof of Proposition \ref{main prop}.

\subsection{Choosing the values of  $\vartheta_0$, $\delta_1$ and $H$ in Proposition \ref{main prop}}
\label{sec6.2}
Let $\vartheta_0 = \frac{1}{12} - \varepsilon_0$, where $\varepsilon_0 > 0$ is sufficiently small, and $\delta_1 = \frac{1}{ 12 (4d)}$.
Let us define $H$ to be the least positive even integer satisfying
\begin{align}
\notag
(1 + \varepsilon_1) \frac{ (2d-1) \left(  1 - 2 \delta_1 \right)  }
{  2 \vartheta_0  \delta_1 } \left( 1 + \frac{ 2 \delta_1 }{2d-1}  \right)&
\leq H
\\
\notag
 &<
(1 + \varepsilon_1 + \varepsilon_2) \frac{ (2d-1) \left(  1 - 2 \delta_1  \right)  }
{  2 \vartheta_0 \delta_1   } \left( 1 + \frac{ 2 \delta_1 }{2d-1}  \right),
\end{align}
where we set $\varepsilon_1 = 1$ and $\varepsilon_2 = 1/100$.
It can be verified that (\ref{someconditions}) is satisfied with these choices.
Then provided
\begin{eqnarray}
\label{codimcondnlong}
\\
\notag
\textnormal{codim} \tsp V_F^* > d^3 (2d-1)^2 2^{2d} 12^4 
\frac{6 \cdot 201^2}{ 100^2}
+      \left( 3 \cdot \frac{201}{100}  \cdot   \frac{12}{2} \cdot  12 (4d) (2d-1) - 1 \right) \mathcal{C}_0,
\end{eqnarray}
we have
\begin{eqnarray}
&& \left( \frac{2 \vartheta_0 \delta_1 H}{2d-1} -  \left( 1 + \frac{ 2 \delta_1 }{2d-1}  \right) (1 - 2 \delta_1) \right) \frac{2/3}{2^{2d} H^2 } \left(   \textnormal{codim} \tsp V_F^*  -
(3 H - 1) \mathcal{C}_0 \right)
\\
\notag
&>&  \varepsilon_1    \frac{ 4  \left( \frac{1}{12} \right)^2  \delta_1^2 }{  (2d-1)^2  (1 + \varepsilon_1 + \varepsilon_2)^2 } \cdot
\frac23
\cdot \frac{\left( \textnormal{codim} \tsp V_F^*  - (3 H - 1) \mathcal{C}_0 \right)}{2^{2d} }
\\
\notag
&=&
\frac{1}{d^2(2d - 1)^2} \left( \frac{1}{12} \right)^4  \frac{4 \cdot 100^2}{4^2 \cdot 201^2} \cdot
\frac23 \cdot
\frac{\left( \textnormal{codim} \tsp V_F^*  - (3 H - 1) \mathcal{C}_0 \right)}{2^{2d} }
\\
\notag
&>& d.
\end{eqnarray}
Similarly, it can be verified that assuming (\ref{codimcondnlong}) we have
\begin{align}
d < \min \Big{\{}   \frac{2}{3} \left( \frac{\vartheta_0 - d \delta_1}{d-1} ( 1 -  \delta_1) -  \delta_1 \right) 2^{-d}(\textnormal{codim} \tsp V_{F}^* - 2\mathcal{C}_0)  ,
\notag
\\
\frac{2}{3} \left( \frac{\vartheta_0 - 2 \delta_1 d}{d-1} ( 1 -  2\delta_1) -  2\delta_1 \right) 2^{-d}(\textnormal{codim} \tsp V_{F}^* - 2 \mathcal{C}_0)    \Big{\}}.
\notag
\end{align}
Finally, by combining the above inequalities, Proposition \ref{main prop} and Lemma \ref{S est 1} (with $\vartheta_0 = (d-1)\vartheta$), we obtain the following.
\begin{prop}
\label{minor arc est}
Let $F \in \mathbb{Z}[x_1, \ldots, x_n]$ be a homogeneous form of degree $d \geq 2$ satisfying (\ref{codimcondnlong}).
Let $\vartheta_0$, $\delta_1$ and $H$ be as above. Then there exists $\delta' > 0$ such that
\begin{eqnarray}
\int_{\mathfrak{m}(\vartheta_0) } |S(\alpha)| \tsp d \alpha \ll N^{n - d - \delta'}.
\end{eqnarray}
\end{prop}

\section{Preliminaries for the major arcs analysis}
Let $q \in \mathbb{N}$ and $\chi^0$ be the principal character modulo $q$. We consider $\chi^0$ as a primitive character modulo $q$ only when $q=1$, and not otherwise. We will use the following zero-free region estimate of the Dirichlet $L$-functions.
\begin{thm}
\label{thm zero free}
Let $s = \sigma + it$, $M \geq 3$, $T \geq 0$ and $\mathfrak{L} = \log M(T + 3)$.
Then there exists an absolute constant $c_1 > 0$ such that $L(s, \chi) \not = 0$ whenever
\begin{eqnarray}
\label{region}
\sigma \geq 1 - \frac{c_1}{\log M + ( \mathfrak{L} \log 2 \mathfrak{L}  )^{3/4}} \ \ \textnormal{  and  } \ \  |t| \leq T
\end{eqnarray}
for all primitive characters $\chi$ of modulus $q \leq M$, with the possible exception of at most one primitive character
$\widetilde{\chi}$ modulo $\widetilde{r}$.
If such $\widetilde{\chi}$ exists, then $L(s, \widetilde{\chi})$ has at most one zero in (\ref{region}) and the exceptional zero $\widetilde{\beta}$ is real and simple, and $\widetilde{r}$ satisfies $M \geq \widetilde{r} \gg_A (\log M)^A$ for any $A>0$.
\end{thm}
\begin{proof}
From the zero-free region estimate of the Riemann zeta function $\zeta(s) $ (for example,
\cite[pp.86]{D} when $|t| < 3$ and \cite[Theorem 5]{F} when $|t|\geq3$), it follows that there exists $c'_1 > 0$ such that
$\zeta(s) \not = 0$ whenever
\begin{eqnarray}
\notag 
\sigma \geq 1 - \frac{c'_1}{\log M + ( \mathfrak{L} \log 2 \mathfrak{L}  )^{3/4}} \ \ \textnormal{  and  } \ \  |t| \leq T.
\end{eqnarray}
By choosing $c_1 \leq c_1'$ this takes care of the result when $q=1$, because
$L(s, \chi^{0}) = \zeta(s)$ in this case. Let $c_1 = \min \{ c'_1, 4^{-1} 10^{-4},  c_{15} / 2 \}$, where $c_{15} > 0$ is an absolute constant from \cite[(7) pp. 93]{D}. Recall there are no primitive characters of modulus $2$.
Let $3  \leq q \leq M$. Then it follows from \cite[Theorem 2]{I} that there exists at most one primitive
character $\chi$ modulo $q$ and a number $\widetilde{\beta}$ in (\ref{region}) such that $L(\widetilde{\beta}, \chi) = 0$, and if there does exist such a character then it is real and the zero $\widetilde{\beta}$ is real and simple.
Let $3 \leq  q_1 < q_2 \leq M$. Suppose for $i \in \{1,2\}$ there exists a real primitive character $\chi_i$ modulo $q_i$ with a real zero $\widetilde{\beta}_i$ in (\ref{region}). Then we have
\begin{eqnarray}
\notag
\widetilde{\beta}_i \geq 1 -  \frac{   c_{15}/ 2  }{ \log M + ( \mathfrak{L} \log 2 \mathfrak{L}  )^{3/4} } >
1 -  \frac{   c_{15}  }{ \log M^2 } \geq 1 -  \frac{   c_{15}  }{ \log q_1 q_2 } \ \ \ (1 \leq i \leq 2),
\end{eqnarray}
which contradicts \cite[(7) pp. 93]{D}. Therefore, we obtain that there exists at most one primitive character with a zero in (\ref{region}).
From Siegel's theorem (for example, see \cite[\S 21]{D}) we have
\begin{equation}
\label{excepZ}
c_3(\varepsilon) \tsp \widetilde{r}^{- \varepsilon}  \leq 1 -  \widetilde{\beta}  \leq \frac{c_1}{\log M + ( \mathfrak{L} \log 2 \mathfrak{L}  )^{3/4}},
\end{equation}
where $c_3(\varepsilon)$ is a positive constant depending only on $\varepsilon > 0$, and the final assertion in the statement of the theorem follows by rearranging this inequality.
\end{proof}
\begin{rem} \label{remZFR}
Let $M = N^{\vartheta_0}$ for a fixed $\vartheta_0 > 0$ and $N$ sufficiently large. Then we have $\widetilde{r} \gg_A (\log N)^{A}$ for any $A > 0$
(the subscript in $\gg_A$ is to indicate that the implicit constant depends on $A$).
In particular, the exceptional zero will not occur for the Dirichlet $L$-functions associated to primitive characters of modulus $1 \leq q \leq (\log N)^D$ with $D > 0$. It also follows from (\ref{excepZ}) that
\begin{equation}
\label{excepZZ}
0 < 1 - \widetilde{\beta} < \frac12.
\end{equation}
\end{rem}
Let us denote $\eta(M,T) = c_1/ \left( \log M + ( \mathfrak{L} \log 2 \mathfrak{L}  )^{3/4} \right)$.
Let
$$
B_T = \{ s = \sigma + it : 0 \leq \sigma \leq 1 ,  |t| \leq T \}.
$$ Let
$N_{\chi}(\alpha, T)$ denote the number of zeros, with multiplicity, of $L(s, \chi)$ in the rectangle
$\alpha \leq \sigma \leq 1$ and $|t| \leq T$.
For the remainder of this paper, we let $\sum'_{\rho}$ denote the sum over the non-exceptional zeros
(with respect to Theorem \ref{thm zero free}), with multiplicity, of $L(s, \chi)$ in $B_T$ and let $\textnormal{Re }\rho = \beta$.
Let $\sum^*_{\chi (\textnormal{mod} \tsp r)}$  
denote the sum over the primitive characters $\chi$ modulo $r$.
\begin{lem}
\label{Gbdd}
Let $\lambda, D, A' > 0$. Let $M = N^{\vartheta_0}$ and $T = N^{\vartheta_1}$, where $\vartheta_0, \vartheta_1 > 0$ satisfy
$2\vartheta_0 + \vartheta_1 < \frac{5}{12}$. Then we have
\begin{align}
\textnormal{i)}& \ \ \ \ \  \ \ \ \ \ \ \  \notag \sum_{  1 \leq r \leq M} \  \sideset{}{^*}\sum_{\chi (\textnormal{mod} \tsp r)} \sideset{}{'}\sum_{\rho} (\lambda N)^{\beta - 1} \ll 1 ,
\\
\notag
\\
\notag
\textnormal{ii)}&  \ \  \ \ \ \
\sum_{  1 \leq r \leq (\log N)^D } \  \sideset{}{^*}\sum_{\chi (\textnormal{mod} \tsp r)} \sum_{ \substack{ \rho \in B_T \\ L(\rho, \chi) = 0  } } (\lambda N)^{\beta - 1} \ll  (\log N)^{-A'},
\end{align}
where the sum $\sum_{ \substack{ \rho \in B_{T}  \\  L(\rho, \chi) = 0  }}$ is over all the zeros, with multiplicity, of $L(s, \chi)$ in $B_{T}$.
Here the implicit constants may depend on $\lambda, D$ and $A'$.
\end{lem}
\begin{proof}
First recall there are no zeros of $L(s, \chi)$ on the lines Re$(s) = 0$ and Re$(s) = 1$ for any primitive character $\chi$.
By the zero-free region estimate (Theorem \ref{thm zero free}) we have
\begin{eqnarray}
\label{zerofree ineq0}
&& \sum_{  1 \leq r \leq M } \  \sideset{}{^*}\sum_{\chi (\textnormal{mod} \tsp r)} \sideset{}{'} \sum_{\rho} (\lambda N)^{\beta - 1}
\\
&=&
- \sum_{  1 \leq r \leq M } \  \sideset{}{^*}\sum_{\chi (\textnormal{mod} \tsp r)} \int_0^{1 - \eta(M, T)} (\lambda N)^{\alpha - 1} \
d_{\alpha}  (N_{\chi} (\alpha, T))
\notag
\\
\notag
&=&
\int_{0}^{1 - \eta(M, T)} (\lambda N)^{\alpha - 1} \log (\lambda N) \sum_{  1 \leq r \leq M } \  \sideset{}{^*}\sum_{\chi (\textnormal{mod} \tsp r)}   N_{\chi} (\alpha, T) \tsp d \alpha + \frac{1}{\lambda N} \sum_{  1 \leq r \leq M } \  \sideset{}{^*}\sum_{\chi (\textnormal{mod} \tsp r)} N_{\chi} (0, T),
\end{eqnarray}
where the integral in the second expression is the Riemann-Stieltjes integral with respect to $N_{\chi} (\alpha, T)$ as a function of $\alpha$,
and the last equality is obtained by integration by parts. Note this is the same calculation as in \cite[pp. 338]{G}.
In order to bound (\ref{zerofree ineq0}), we apply \cite[(1.1)]{Hux} when $0 \leq \alpha < 4/5$
(It is stated in \cite[pp. 438]{Hux} that \cite[(1.1)]{Hux} is valid for $\sigma \geq 1/2$; however, the validity of
\cite[(1.1)]{Hux} in the range $0 \leq \alpha < 1/2$ follows easily from the symmetry, across the line Re$(s) = 1/2$, of the zeros of
the Dirichlet $L$-functions.) and \cite[Theorem 1]{J} when $\alpha \geq 4/5$. As a result, we obtain that (\ref{zerofree ineq0}) is bounded by
\begin{eqnarray}
\label{zerofree ineq1}
&\ll&\int_0^{4/5} N^{ (\alpha - 1) } (\log N) (M^2T)^{ \frac{12}{5}(1 - \alpha) + \varepsilon} \tsp  d \alpha
\\
&+& \int_{4/5}^{1 - \eta(M,T)} N^{ (\alpha - 1) } (\log N) (M^2 T)^{(2 + \varepsilon)(1 - \alpha)} \tsp  d \alpha +  \frac{(M^2T)^{\frac{12}{5} + \varepsilon} }{N}
\notag
\\
&\ll&
\notag
\left( \frac{  (M^2T)^{  \frac{12}{5} + 5 \varepsilon  } }{N}
\right)^{ \frac15 } + \left( \frac{(M^2T)^{2 + \varepsilon  } }{N} \right)^{\eta(M, T)} + \frac{ (M^2T)^{\frac{12}{5} + \varepsilon } }{N},
\end{eqnarray}
where $\varepsilon > 0 $ is sufficiently small.
It can be verified that there exists $c'(\vartheta_0, \vartheta_1) >0$ depending only on $\vartheta_0$ and $\vartheta_1$ such that
$\eta(M,T) \geq c'(\vartheta_0, \vartheta_1) / \log N$. From this we easily see that (\ref{zerofree ineq1}) is bounded by $\ll 1$.

Next we consider the second sum in the statement. Let $M_1 = (\log N)^D$. It follows from Theorem \ref{thm zero free} that
\begin{eqnarray}
\sum_{  1 \leq r \leq  M_1} \  \sideset{}{^*}\sum_{\chi (\textnormal{mod} \tsp r)} \sum_{ \substack{ \rho \in B_T \\ L(\rho, \chi) = 0  } } (\lambda N)^{\beta - 1} =
(\lambda N)^{\widetilde{\beta} - 1}
+
\sum_{  1 \leq r \leq  M_1} \  \sideset{}{^*}\sum_{\chi (\textnormal{mod} \tsp r)} \sideset{}{'} \sum_{\rho} (\lambda N)^{\beta - 1}
\notag
\end{eqnarray}
if the exceptional zero exists, otherwise we have the equality without the term $(\lambda N)^{\widetilde{\beta} - 1}$.
Here the sum $\sum'_{\rho}$ over the non-exceptional zeros is with respect to Theorem \ref{thm zero free} with $M_1$ (in place of $M$).
Then the same argument as above yields
\begin{eqnarray}
\label{zerofree ineq1'}
&& \sum_{  1 \leq r \leq  M_1 } \  \sideset{}{^*}\sum_{\chi (\textnormal{mod} \tsp r)} \sum_{ \substack{ \rho \in B_T \\ L(\rho, \chi) = 0  } } (\lambda N)^{\beta - 1}
\\
&\ll&
\notag
N^{\widetilde{\beta} - 1}
+
\left( \frac{  (M_1^2T)^{  \frac{12}{5} + 5 \varepsilon  } }{N}
\right)^{ \frac15 } + \left( \frac{(M_1^2T)^{2 + \varepsilon  } }{N} \right)^{\eta(M_1, T)} + \frac{ (M_1^2T)^{\frac{12}{5} + \varepsilon } }{N}.
\end{eqnarray}
It can be verified that there exists $c''(\vartheta_1, D) >0$ depending only on $\vartheta_1$ and $D$ such that
$\eta(M_1,T) \geq c''(\vartheta_1, D)/ (\log N)^{3.1/4}$. Thus for any $\kappa_1 \geq 0$ we have
$$
N^{- \kappa_1 \eta(M_1,T)} = e^{- \kappa_1 (\log N) \tsp \eta(M_1,T)} \leq e^{  - \kappa_1 (\log N)^{0.9/4} c''(\vartheta_1, D) }.
$$
Without loss of generality let $0 < \varepsilon < 1/D$. Since $\widetilde{r} \leq M_1 = (\log N)^D$, by Siegel's theorem we know
$$
\widetilde{\beta} - 1 \leq - c_3(\varepsilon) \tsp \widetilde{r}^{- \varepsilon} \leq - c_3(\varepsilon) \tsp (\log N)^{- \varepsilon D}.
$$
Thus there exists $\kappa_2 > 0$ such that
$$
N^{\widetilde{\beta} - 1} \leq e^{ - c_3(\varepsilon) \frac{ \log N }{ (\log N)^{\varepsilon D}  } } \leq e^{ - c_3(\varepsilon) \tsp (\log N)^{\kappa_2}}.
$$
Therefore, we see that the right hand side of (\ref{zerofree ineq1'}) is bounded by $\ll (\log N)^{-A'}$ for any $A'>0$.
\end{proof}

Let $\phi$ be Euler's totient function.
For a positive integer $q$ let $\mathbb{U}_q = (\mathbb{Z}/q \mathbb{Z})^*$, the group of units in $\mathbb{Z}/q \mathbb{Z}$.
\begin{lem}
\label{exp sum modp}
Let $F \in \mathbb{Z}[x_1, \ldots, x_n]$ be a homogeneous form of degree $d \geq 2$.
Let $a \in \mathbb{Z}$ and $q \in \mathbb{N}$ be such that $\gcd(a,q)=1$.
Let $\chi_1, \ldots, \chi_n$ be any Dirichlet characters modulo $q$, and let
\begin{eqnarray}
S = \sum_{\mathbf{h} \in (\mathbb{Z}/q \mathbb{Z})^n  } \chi_1(h_1) \cdots \chi_n(h_n)
\ e \left( \frac{a}{q} F(h_1, \ldots, h_n)  \right).
\end{eqnarray}
Then for any $\varepsilon > 0$ we have
$$
|S| \ll q^{n - \frac{1}{2(2d-1) 4^d} \textnormal{codim} \tsp V_{F}^* + \varepsilon}.
$$
\end{lem}
\begin{proof}
Let $\mathbf{k} \in \mathbb{U}_q^n$. Then it is clear that
\begin{eqnarray}
S = \sum_{\mathbf{h} \in (\mathbb{Z}/q \mathbb{Z})^n  } \chi_1(h_1 k_1 ) \cdots \chi_n(h_n k_n)
\ e \left( \frac{a}{q} F(h_1 k_1, \ldots, h_n k_n)  \right).
\end{eqnarray}
Therefore, it follows that
\begin{eqnarray}
\phi(q)^n S = \sum _{\mathbf{k} \in (\mathbb{Z}/q \mathbb{Z})^n  } \sum_{\mathbf{h} \in (\mathbb{Z}/q \mathbb{Z})^n  } \prod_{j=1}^n \chi_j(h_j) \chi_j(k_j)  \cdot
e \left( \frac{a}{q} G(\mathbf{h}; \mathbf{k}) \right),
\end{eqnarray}
where  $G(\mathbf{h}; \mathbf{k}) = F(h_1 k_1, \ldots, h_n k_n).$
By applying the Cauchy-Schwarz inequality twice, we obtain
\begin{eqnarray}
\phi(q)^{4n} |S|^4 \leq \phi(q)^{4n}
\sum _{\mathbf{h}, \mathbf{h}' \in (\mathbb{Z}/q \mathbb{Z})^n  } \sum_{\mathbf{k}, \mathbf{k}' \in (\mathbb{Z}/q \mathbb{Z})^n  }
e \left( \frac{a}{q} \mathfrak{D}(\mathbf{h}, \mathbf{h}' ; \mathbf{k}, \mathbf{k}' )  \right),
\end{eqnarray}
where
$$
\mathfrak{D}(\mathbf{h}, \mathbf{h}' ; \mathbf{k}, \mathbf{k}' ) = G(\mathbf{h} ; \mathbf{k}) - G(\mathbf{h}' ; \mathbf{k}) -G(\mathbf{h} ; \mathbf{k}') +G(\mathbf{h}' ; \mathbf{k}').
$$
Note $\mathfrak{D}(\mathbf{h}, \mathbf{h}' ; \mathbf{k}, \mathbf{k}' )$ is bihomogeneous in the variables $(\mathbf{h}, \mathbf{h}')$ and $(\mathbf{k}, \mathbf{k}')$.
It follows from \cite[(2.11)]{Y1} and Proposition \ref{semiprimeprop} that
$$
\textnormal{codim} \tsp V_{\mathfrak{D}, 1}^* \geq \textnormal{codim} \tsp V_{G, 1}^* \geq  \frac{\textnormal{codim} \tsp V_{F}^*}{2}.
$$
By symmetry it also follows that $\textnormal{codim} \tsp V_{\mathfrak{D}, 2}^* \geq  \frac{\textnormal{codim} \tsp V_{F}^*}{2}.$
Thus from \cite[Corollary 2.5]{Y1}, which is a restatement of a result obtained in \cite{DS}, we obtain
\begin{eqnarray}
|S|^4 \leq \Big{|} \sum _{\mathbf{h}, \mathbf{h}' \in (\mathbb{Z}/q \mathbb{Z})^n  } \sum_{\mathbf{k}, \mathbf{k}' \in (\mathbb{Z}/q \mathbb{Z})^n  }
e \left( \frac{a}{q} \mathfrak{D}(\mathbf{h}, \mathbf{h}' ; \mathbf{k}, \mathbf{k}' )  \right) \Big{|}
\notag
\ll
q^{4n - \frac{2}{(2d-1) 4^d} \textnormal{codim} \tsp V_{F}^* + \varepsilon}.
\end{eqnarray}
\end{proof}
The following proposition on oscillatory integrals will be crucial. However, since the proof is long and technical we will not
get into the details here, instead we present the proof in \cite{Y}.
\begin{prop}
\label{prop osc int}
Let $F \in \mathbb{Z}[x_1, \ldots, x_n]$ be a homogeneous form of degree $d \geq 2$ satisfying $n - \dim V_{F}^* > 4$.
Let $r_j \in [-1, 0]$ and $t_j \in \mathbb{R}$ $(1 \leq j \leq n)$.
Suppose $\mathbf{x}_0 = (x_{0,1}, \ldots, x_{0,n}) \in (0,1)^n$ is a non-singular real solution to the equation $F(\mathbf{x}) = 0$.
Let $\omega \in \mathcal{S}^+(\delta; n; \mathfrak{c})$.
Then provided $\delta$ is sufficiently small, we have
$$
\Big{|} \int_{0}^{\infty} \cdots \int_{0}^{\infty}
\prod_{\ell=1}^n \omega(x_{\ell} - x_{0, \ell}) \cdot x_1^{r_1 + i t_1} \cdots x_n^{r_n + i t_n} \tsp  e( \tau F(\mathbf{x}))  \tsp  d \mathbf{x}  \Big{|} \ll \min \{ 1, |\tau|^{-1} \},
$$
where the implicit constant is independent of $r_1, \ldots, r_n$,
$t_1, \ldots, t_n$ and $\tau$.
\end{prop}
The key feature of the result is that the bound is uniform in $\mathbf{t}$; the result can be deduced easily for a fixed $\mathbf{t} \in \mathbb{R}^n$ (for example, by \cite[Lemma 10]{HB}),
but obtaining the uniformity is quite delicate and this is where the challenge lies. We make use of an explicit version of the inverse function theorem, the stationary phase method,
basic oscillatory integral estimates, some differential geometry and algebraic geometry over $\mathbb{R}$  to achieve this.
\section{Major arcs}
Let $F \in \mathbb{Z}[x_1, \ldots, x_n]$ be a homogeneous form of degree $d \geq 2$ satisfying (\ref{codimcondnlong}) and the local conditions \textnormal{($\star$)}. Let $\mathbf{x}_0$, $\delta$, $\omega$ and $\varpi$ be as in the statement of Theorem \ref{mainthm2}.
Recall the definition of $\Lambda^*$ given in (\ref{defnoflamdastar}).
Let us define
\begin{eqnarray}
S_1(\alpha) = \sum_{\mathbf{x} \in \mathbb{N}^{n} }  \varpi(\mathbf{x})    \Lambda^* (\mathbf{x})  e ( \alpha  F(\mathbf{x})  ).
\end{eqnarray}
Clearly we have
\begin{eqnarray}
\label{S and S1}
S(\alpha) = S_1(\alpha) + O(N^{n - \frac12}).
\end{eqnarray}

Let $\varepsilon_0 > 0$ be sufficiently small. We set
\begin{eqnarray}
\label{defnT}
\vartheta_0 = \frac{1}{12} - \varepsilon_0, \ \   M = N^{\vartheta_0}
\ \  \textnormal{  and  } \ \
T = N^{3 (\vartheta_0 + \varepsilon_0)},
\end{eqnarray}
and let $\widetilde{\chi}$, $\widetilde{r}$ and $\widetilde{\beta}$ be as in Theorem \ref{thm zero free}.
We define the following slightly modified major arcs
\begin{equation}
\label{defn major}
\mathfrak{M}'(\vartheta_0) = \bigcup_{1 \leq q \leq N^{\vartheta_0}} \bigcup_{\substack{ 0 \leq a \leq q  \\ \gcd(a, q) = 1}} \mathfrak{M}'_{q,a}(\vartheta_0),
\end{equation}
where
$$
\mathfrak{M}'_{q, a} (\vartheta_0) = \left\{ {\alpha} \in [0,1) :   \Big{|}  \alpha - \frac{a}{q}  \Big{|} <  N^{\vartheta_0 - d} \right\}.
$$
It can be verified that the arcs $\mathfrak{M}'_{q, a}(\vartheta_0)$'s are disjoint for $N$ sufficiently large.

Let $m, \ell \in \mathbb{Z}_{\geq 0}$ be such that $m + \ell \leq n$. We denote
$\mathbf{j} = (j_1, \ldots, j_m)$, $\mathbf{k} = (k_1, \ldots, k_{\ell})$ and $\mathbf{i} = (i_1, \ldots, i_{n - m - \ell})$
satisfying
\begin{eqnarray}
\label{cond on j and k}
\{ 1, \ldots, n \} = \{ i_1, \ldots, i_{n - m - \ell} \} \cup  \{ j_1, \ldots, j_m \} \cup \{ k_1, \ldots, k_{\ell} \}.
\end{eqnarray}
For each such triple $(\mathbf{i}, \mathbf{j}, \mathbf{k})$ and Dirichlet characters $\chi_1, \ldots, \chi_m$ modulo $q$, we define
\begin{eqnarray}
&&\mathcal{A}(q, a; \mathbf{i}; (j_1, \chi_1), \ldots,  (j_m, \chi_m); \mathbf{k} )
\\
&=& \sum_{\substack{ \mathbf{h} \in \mathbb{U}_q^n } } \
\overline{\chi_1}(h_{j_1}) \cdots \overline{\chi_m}(h_{j_m}) \tsp
\overline{\widetilde{\chi} \chi^0 }(h_{k_1}) \cdots \overline{\widetilde{\chi} \chi^0} (h_{k_{\ell}})
\ e \left( \frac{a}{q}  F(\mathbf{h} )   \right);
\notag
\end{eqnarray}
if $\widetilde{r} \nmid q$, then we only consider $(\mathbf{i}, \mathbf{j}, \mathbf{k})$  with $\mathbf{k} = \emptyset$, i.e. $\ell = 0$.
Similarly, we define
\begin{eqnarray}
\label{DefnW}
&&\mathcal{W}(\tau; \mathbf{i};  (j_1, \chi_1), \ldots,  (j_m, \chi_m); \mathbf{k}  )
\\
&=&
\int_0^{\infty} \cdots \int_0^{\infty}
\sum_{x_{j_1}, \ldots, x_{j_m} \in \mathbb{N}} \varpi(\mathbf{x}) \tsp  \chi_1(x_{j_1}) \Lambda^*(x_{j_1}) \cdots \chi_m(x_{j_m}) \Lambda^*(x_{j_m})
\cdot
\notag
\\
&& x_{k_1}^{\widetilde{\beta} - 1} \cdots x_{k_{\ell}}^{\widetilde{\beta} - 1}
\tsp e ( \tau F(\mathbf{x})  ) \tsp  d x_{i_1} \cdots d x_{i_{n - m - \ell}} d x_{k_1} \cdots d x_{k_{\ell}},
\notag
\end{eqnarray}
where for each $1 \leq v \leq m$ we replace $\chi_v(x_{j_v}) \Lambda^*(x_{j_v})$ with $\left( \chi^0(x_{j_v}) \Lambda^*(x_{j_v}) - 1 \right)$ when $\chi_{v} = \chi^0$, and with $\left( \widetilde{\chi} \chi^0(x_{j_v}) \Lambda^*(x_{j_v}) + x_{j_v}^{\widetilde{\beta} - 1} \right)$ when $\chi_v = \widetilde{\chi} \chi^0$.
With these notation we have the following.
\begin{lem}
\label{first sum lemma}
Let $\alpha \in [0,1)$, where $\alpha = a / q + {\tau}$, $0 \leq a \leq q \leq N^{\vartheta_0}$, $\gcd(a,q) = 1$
and $|\tau| < N^{\vartheta_0 - d}$.
Then we have
\begin{eqnarray}
\label{first sum}
S_1(\alpha) &=& \frac{1}{\phi(q)^n} \tsp  \mathcal{A} (q, a; (1, \ldots, n) ; \emptyset ;  \emptyset ) \tsp \mathcal{W}(\tau;(1, \ldots, n) ;  \emptyset ; \emptyset )
\\
\notag
&+& \sum_{(\mathbf{j}, \mathbf{k}) } \frac{(-1)^{\ell}}{\phi(q)^n} \tsp S_{\mathbf{j}, \mathbf{k}}(\alpha)
+ O(N^{n + \vartheta_0 -1}),
\end{eqnarray}
where
\begin{eqnarray}
S_{\mathbf{j}, \mathbf{k}}(\alpha) =  \sum_{\chi_1, \ldots, \chi_m (\textnormal{mod} \tsp q)}
\mathcal{A}(q, a; \mathbf{i} ;  (j_1, \chi_1), \ldots,  (j_m, \chi_m); \mathbf{k} ) \tsp  \mathcal{W}(\tau; \mathbf{i};  (j_1, \chi_1), \ldots,  (j_m, \chi_m); \mathbf{k} ),
\notag
\end{eqnarray}
and the sum  $\sum_{(\mathbf{j}, \mathbf{k}) }$ in (\ref{first sum}) is over all $(\mathbf{j}, \mathbf{k})$ satisfying $(m, \ell) \not = (0,0)$
and (\ref{cond on j and k}) with an additional condition $\ell = 0$ if $\widetilde{r} \nmid q$ or
the exceptional zero does not exist.
\end{lem}

\begin{proof}
Here we only consider the case where $\widetilde{r}|q$ and the exceptional zero does exist; the proof for the cases
$\widetilde{r} \nmid q$ or the exceptional zero does not exist
are identical to this case with only slight modifications. First note if $x_u \in \wp \cap [(x_{0,u} - \delta)N, (x_{0,u} + \delta)N]$, then $\gcd(x_u, q) =1$ because $q \leq N^{\vartheta_0}$. Therefore, we obtain the following via the orthogonality relation of
the Dirichlet characters
\begin{eqnarray}
\notag
S_1(\alpha) &=& \frac{1}{\phi(q)^n} \sum_{\chi_1, \ldots, \chi_n (\textnormal{mod} \tsp q)} \
\sum_{\mathbf{h} \in \mathbb{U}_q^n} \overline{\chi_1}(h_1) \cdots \overline{\chi_n}(h_n) \  e \left( \frac{a}{q} F(\mathbf{h}) \right) \cdot
\\
&&
\notag
\sum_{\mathbf{x} \in \mathbb{N}^n} \varpi(\mathbf{x}) \prod_{u=1}^n \chi_u(x_u) \Lambda^*(x_u) \cdot  e ( \tau F(\mathbf{x}) ).
\end{eqnarray}

Let $\mathbf{i}' = (i'_1, \ldots, i'_{s'})$, $\mathbf{j}' = (j'_1, \ldots, j'_{m'})$, $\mathbf{k}' = (k'_1, \ldots, k'_{\ell'})$,
$\mathbf{k} = (k_1, \ldots, k_{\ell})$ and $\mathbf{i} = (i_1, \ldots, i_{n - s' - m' - \ell' - \ell})$, where
\begin{eqnarray}
\label{paritionnn}
 \{1, \ldots, n\} &=& \{i_1, \ldots, i_{n - s' - m' - \ell' - \ell} \} \cup \{ i'_1, \ldots, i'_{s'} \} \cup  \{ j'_1, \ldots, j'_{m'} \}
\\
\notag
&\cup&  \{ k'_1, \ldots, k'_{\ell'} \} \cup  \{ k_1, \ldots, k_{\ell} \}.
\end{eqnarray}
Now we break up the summands of the inner sum $\sum_{\mathbf{x} \in \mathbb{N}^n}$ 
using the identities
$$
\chi^0(x_{u}) \Lambda^*(x_{u}) = \left( \chi^0(x_{u}) \Lambda^*(x_{u}) - 1 \right) + 1
$$
when $\chi_u = \chi^0$, and
$$
\widetilde{\chi} \chi^0(x_{u}) \Lambda^*(x_{u})  = \left( \widetilde{\chi} \chi^0(x_{w}) \Lambda^*(x_{u}) + x_u^{\widetilde{\beta} - 1} \right) - x_u^{\widetilde{\beta} - 1}
$$
when $\chi_u = \widetilde{\chi} \chi^0$ (If $ \widetilde{r} \nmid q$ or the exceptional zero does not exist, we simply ignore this second identity.); we do this for each $1 \leq u \leq n$ and obtain
\begin{eqnarray}
\notag
S_1(\alpha) &=& \sum_{ (\mathbf{i}, \mathbf{i'}, \mathbf{j'}, \mathbf{k'}, \mathbf{k}) }
\frac{1}{\phi(q)^n} \sum_{\chi_{j'_1}, \ldots, \chi_{j'_{m'}} (\textnormal{mod} \tsp q)} \
\sum_{\mathbf{h} \in \mathbb{U}_q^n} \overline{\chi_{j'_1} }(h_1) \cdots \overline{\chi_{j'_{m'}} }(h_{j'_{m'}}) \cdot
\\
\notag
&& \prod_{u \in \{ k'_1, \ldots, k'_{\ell'} \} \cup  \{ k_1, \ldots, k_{\ell} \}}  \overline{  \widetilde{\chi} \chi^0  }(h_{u})
\cdot  e \left( \frac{a}{q} F(\mathbf{h}) \right) \cdot \sum_{\mathbf{x} \in \mathbb{N}^n} \varpi(\mathbf{x})
\prod_{w \in \{ i'_1, \ldots, i'_{s'} \}  } \chi^0(x_w) \Lambda^*(x_w)
\cdot
\\
&&
\notag
\prod_{v \in \{ j'_1, \ldots, j'_{m'} \}  }  \chi_v(x_v) \Lambda^*(x_v) \cdot
\prod_{u \in \{ k'_1, \ldots, k'_{\ell'} \} \cup  \{ k_1, \ldots, k_{\ell} \}  } \widetilde{\chi} \chi^0(x_u) \Lambda^*(x_u) \cdot
\\
\notag
&&
(-1)^{\ell} x_{k_1}^{\widetilde{\beta} - 1} \cdots x_{k_{\ell}}^{\widetilde{\beta} - 1}
\tsp   e ( \tau F(\mathbf{x}) ),
\notag
\end{eqnarray}
where the sum $\sum_{ (\mathbf{i}, \mathbf{i'}, \mathbf{j'}, \mathbf{k'}, \mathbf{k}) }$ is over all $(\mathbf{i}, \mathbf{i'}, \mathbf{j'}, \mathbf{k'}, \mathbf{k})$ satisfying (\ref{paritionnn}), and the convention regarding $\chi^0(x_w) \Lambda^*(x_w)$ and  $\widetilde{\chi} \chi^0(x_u) \Lambda^*(x_u)$
described in the sentence following (\ref{DefnW}) is being used.
Then we set $\mathbf{j} = (\mathbf{i}', \mathbf{j}', \mathbf{k}')$, and it follows that
\begin{eqnarray}
\label{sum'''}
S_1(\alpha) &=& \frac{1}{\phi(q)^n} \tsp  \mathcal{A} (q, a; (1, \ldots, n) ; \emptyset ;  \emptyset )
\sum_{\mathbf{x} \in \mathbb{N}^n} \varpi(\mathbf{x})  e ( \tau F(\mathbf{x})  )
\\
\notag
&+& \sum_{(\mathbf{j}, \mathbf{k}) } \frac{(-1)^{\ell}}{\phi(q)^n}
\sum_{\chi_1, \ldots, \chi_m (\textnormal{mod} \tsp q)}
\mathcal{A}(q, a; \mathbf{i} ;  (j_1, \chi_1), \ldots,  (j_m, \chi_m); \mathbf{k} ) \cdot
\\
&& \sum_{\mathbf{x} \in \mathbb{N}^n} \varpi(\mathbf{x}) \tsp  \chi_1(x_{j_1}) \Lambda^*(x_{j_1}) \cdots \chi_m(x_{j_m}) \Lambda^*(x_{j_m})
\tsp  x_{k_1}^{\widetilde{\beta} - 1} \cdots x_{k_{\ell}}^{\widetilde{\beta} - 1}
\tsp e ( \tau F(\mathbf{x})  ),
\notag
\end{eqnarray}
where the sum $\sum_{(\mathbf{j}, \mathbf{k}) }$ is over all $(\mathbf{j}, \mathbf{k})$ satisfying $(m, \ell) \not = (0,0)$ and (\ref{cond on j and k}).
Each summand of the sum $\sum_{(\mathbf{j}, \mathbf{k}) }$ in (\ref{sum'''}) can also be expressed as
\begin{eqnarray}
\notag
&& \frac{ (-1)^{\ell} }{\phi(q)^n}
\sum_{\mathbf{h} \in \mathbb{U}_q^n} \overline{\widetilde{\chi} \chi^0 }(h_{k_1}) \cdots
\overline{\widetilde{\chi} \chi^0}(h_{k_{\ell}}) \
e \left( \frac{a}{q} F(\mathbf{h}) \right) \cdot
\\
&&
\notag
\sum_{ \substack{ x_{j_1}, \ldots, x_{j_m} \in \mathbb{N} } } \
\prod_{  1 \leq v \leq m } \Lambda^*(x_{j_v}) \sum_{\chi_v (\textnormal{mod} \tsp q)}   \overline{\chi_v}(h_{j_v})  \chi_v(x_{j_v})
\cdot
\sum_{ \substack{ x_{u} \in \mathbb{N} \\ (u \not \in\{j_1, \ldots, j_m\}) } }
\varpi(\mathbf{x}) \tsp x_{k_1}^{\widetilde{\beta} - 1} \cdots x_{k_{\ell}}^{\widetilde{\beta} - 1} e ( \tau F(\mathbf{x}) ).
\end{eqnarray}
Next we apply the following estimate which can be deduced from the mean value theorem along with (\ref{excepZZ}): for each $x_{j_1}, \ldots, x_{j_m} \in \mathbb{N}$ we have
\begin{eqnarray}
\notag
&&  \sum_{ \substack{ x_{u} \in \mathbb{N} \\ (u \not \in\{j_1, \ldots, j_m\}) } }
\varpi(\mathbf{x}) \tsp x_{k_1}^{\widetilde{\beta} - 1} \cdots x_{k_{\ell}}^{\widetilde{\beta} - 1} e ( \tau F(\mathbf{x})  )
\\
&=&
\notag
\int_0^{\infty} \cdots \int_0^{\infty}
\varpi(\mathbf{x}) \tsp x_{k_1}^{\widetilde{\beta} - 1} \cdots x_{k_{\ell}}^{\widetilde{\beta} - 1}
 e ( \tau F(\mathbf{x})  ) \tsp  d x_{i_1} \cdots d x_{i_{n - m - \ell}} d x_{k_1} \cdots d x_{k_{\ell}}
+ O(N^{n - m + \vartheta_0 - 1}),
\end{eqnarray}
where the implicit constant is independent of $x_{j_1}, \ldots, x_{j_m}$.
Therefore, by substituting these expressions into (\ref{sum'''}), we see that
we have obtained the result apart form the error term. Since
$$
 \sum_{\chi_v (\textnormal{mod} \tsp q)}   \overline{\chi_v}(h_{j_v})  \chi_v(x_{j_v}) \geq 0,
$$
for each $(\mathbf{j}, \mathbf{k})$ we have by the orthogonality relation of the Dirichlet characters and the prime number theorem that
\begin{eqnarray}
&& \frac{  N^{n - m + \vartheta_0 - 1} }{\phi(q)^n}
 \sum_{\mathbf{h} \in \mathbb{U}_q^n} \, \sum_{ 1 \leq  x_{j_1} \leq N} \cdots \sum_{ 1 \leq  x_{j_m} \leq N}
\prod_{  1 \leq v \leq m } \Lambda^*(x_{j_v}) \, \Big{|} \sum_{\chi_v (\textnormal{mod} \tsp q)}   \overline{\chi_v}(h_{j_v})  \chi_v(x_{j_v}) \Big{|}
\\
&=&
\notag
\frac{  N^{n - m + \vartheta_0 - 1} }{\phi(q)^m}
\prod_{1 \leq v \leq m} \,
\sum_{h_{j_v} \in \mathbb{U}_q} \,  \sum_{ \substack{ 1 \leq x_{j_v} \leq N } }  \Lambda^*(x_{j_v}) \sum_{\chi_v (\textnormal{mod} \tsp q)}   \overline{\chi_v}(h_{j_v})  \chi_v(x_{j_v})
\\
&=&
\notag
N^{n - m + \vartheta_0 - 1}
\prod_{1 \leq v \leq m} \,
\sum_{h_{j_v} \in \mathbb{U}_q} \, \sum_{ \substack{ 1 \leq x_{j_v} \leq N \\  x_{j_v} \equiv h_{j_v} (\textnormal{mod} \tsp q) } }  \Lambda^*(x_{j_v})
\\
&\leq&
\notag
N^{n - m + \vartheta_0 - 1}
\left( \sum_{ 1 \leq x \leq N }  \Lambda(x) \right)^m
\\
&\ll&
\notag
N^{n + \vartheta_0 - 1},
\end{eqnarray}
and from this bound it follows that the error term is as in the statement of the lemma.
\end{proof}

It follows from Lemma \ref{first sum lemma} that
\begin{eqnarray}
\label{MAJOR1}
&&\int_{\mathfrak{M}'(\vartheta_0)} S_1(\alpha) \tsp d \alpha
\\
\notag
&=&
\sum_{1 \leq q \leq N^{\vartheta_0}} \sum_{\substack{ 1 \leq a \leq q \\  (a,q) = 1 }} \frac{1}{\phi(q)^n} \tsp \mathcal{A} (q, a; (1, \ldots, n) ; \emptyset ;  \emptyset )
\cdot
\int_{|\tau| < N^{\vartheta_0 - d}}
\mathcal{W} (\tau;(1, \ldots, n) ;  \emptyset ; \emptyset ) \tsp d \tau
\\
\notag
&+& \sum_{1 \leq q \leq N^{\vartheta_0}} \sum_{\substack{ 1 \leq a \leq q \\  (a,q) = 1 }} \sum_{(\mathbf{j}, \mathbf{k}) } \frac{(-1)^{\ell}}{\phi(q)^n} \tsp \int_{|\tau| < N^{\vartheta_0 - d}} S_{\mathbf{j}, \mathbf{k}} \left( \frac{a}{q} + \tau \right) \tsp d \tau
+ O(N^{n - d + 4 \vartheta_0 - 1}),
\end{eqnarray}
where the sum $\sum_{(\mathbf{j}, \mathbf{k}) }$ is as in the statement of Lemma \ref{first sum lemma}.
We prove that the first term on the right hand side of (\ref{MAJOR1}) contributes the main term, while the remaining terms are error terms. We change the order of summation in the second term, and consider the contribution from each $( \mathbf{j}, \mathbf{k})$ separately. We deal with the case $\ell = 0$ in Section \ref{sec m>0}, and the case $\ell > 0$ in Section \ref{secm=0}.
Finally, Theorem \ref{mainthm2} is established in Section \ref{secmajormainterm}.

\subsection{Case $\ell = 0$}
\label{sec m>0}
Since $(m, \ell) \not = (0,0)$, we necessarily have $m > 0$. Without loss of generality let $\mathbf{j} = (1, \ldots, m)$.
In this case, we have
\begin{eqnarray}
\label{eqn 1 m>0}
&& \Big{|} \sum_{1 \leq q \leq N^{\vartheta_0}} \sum_{ \substack{ 1 \leq a \leq q \\  \gcd(a,q) = 1 } } \frac{1}{\phi(q)^n}  \int_{|\tau| < N^{\vartheta_0 - d}}  S_{\mathbf{j}, \emptyset } \left( \frac{a}{q} + \tau \right)  d \tau \Big{|}
\\
\notag
&=& \Big{|} \sum_{1 \leq q \leq N^{\vartheta_0}}  \sum_{ \substack{ \chi'_1, \ldots, \chi'_m \\ (\textnormal{mod} \tsp q) } }  \sum_{\substack{ 1 \leq a \leq q \\  \gcd(a,q) = 1 }} \frac{1}{\phi(q)^n} \tsp \mathcal{A}(q, a; \mathbf{i};  (1, \chi'_1), \ldots, (m, \chi'_m) ; \emptyset  ) \cdot  \\
&& \int_{|\tau| < N^{\vartheta_0 - d}}  \mathcal{W}(\tau; \mathbf{i}; (1, \chi'_1), \ldots, (m, \chi'_m) ; \emptyset  )  \tsp d \tau  \Big{|}.
\notag
\end{eqnarray}
Let us denote $\chi'_v = \chi_v \chi^0$, where $\chi_v$ is the primitive character modulo $r_v$ which induces $\chi'_v$,
and $\chi^0$ is the principal character modulo $q$. We also denote $R= \textnormal{lcm}(r_1, \ldots, r_m)$.
We let 
$$
\delta_{\min} =  \min_{1 \leq u \leq n} (x_{0,u} - \delta).
$$
Let $p \in [(x_{0,v} - \delta)N, (x_{0,v} + \delta) N]$ be a prime.
Since $q \leq N^{\vartheta_0} < \delta_{\min} N$, we have $\gcd(p,q) = 1$ and it follows that
$\chi_v(p) \chi^{0}(p) = \chi_v(p)$.
Consequently, we obtain
$$
\mathcal{W}(\tau; \mathbf{i}; (1, \chi'_1), \ldots, (m, \chi'_m) ; \emptyset  )
= \mathcal{W}(\tau; \mathbf{i}; (1, \chi_1), \ldots, (m, \chi_m) ; \emptyset  ).
$$
Let $\kappa$ be any real number satisfying
$$
2 < \kappa < \frac{\textnormal{codim} \tsp V_F^*}{2(2d-1)4^d} -1.
$$
Since $q / \phi(q) \ll_{\varepsilon} q^{\varepsilon}$ for any $\varepsilon > 0$, it follows from Lemma \ref{exp sum modp} that
\begin{eqnarray}
\label{usinglem7.4}
\sum_{\substack{ 1 \leq q \leq N^{\vartheta_0} \\  R |q   }}  \sum_{\substack{ 1 \leq a \leq q \\  \gcd(a,q) = 1 }} \frac{1}{\phi(q)^n} | \mathcal{A}(q, a; \mathbf{i}; (1, \chi_1 \chi^0), \ldots, (m, \chi_m \chi^0) ; \emptyset ) | \ll R^{- \kappa}.
\end{eqnarray}
Therefore, the term on the right hand side of (\ref{eqn 1 m>0}) can be rewritten as
\begin{eqnarray}
\label{eqn 2 m>0}
&& \Big{|} \sum_{   1 \leq r_1, \ldots, r_m \leq N^{\vartheta_0} } \  \sideset{}{^*}\sum_{\substack{\chi_v (\textnormal{mod} \tsp r_v)  \\  (1 \leq v \leq m) } }
\sum_{\substack{ 1 \leq q \leq N^{\vartheta_0} \\  R |q   }}   \sum_{\substack{ 1 \leq a \leq q \\  \gcd(a,q) = 1 }} \frac{1}{\phi(q)^n} \tsp
\mathcal{A}(q, a; \mathbf{i}; (1, \chi_1 \chi^0), \ldots, (m, \chi_m \chi^0) ;\emptyset ) \cdot
\\
\notag
&& \int_{|\tau| < N^{\vartheta_0 - d}} \mathcal{W}(\tau; \mathbf{i}; (1, \chi_1), \ldots, (m, \chi_m) ; \emptyset ) \tsp d \tau
\Big{|}
\\
\notag
&\ll& \sum_{   1 \leq r_1, \ldots, r_m \leq N^{\vartheta_0} } R^{-\kappa}  \sideset{}{^*}\sum_{\substack{\chi_v (\textnormal{mod} \tsp r_v) \\ (1 \leq v \leq m) } }
\int_{|\tau| < N^{\vartheta_0 - d}}
| \mathcal{W}(\tau; \mathbf{i}; (1, \chi_1), \ldots, (m, \chi_m) ; \emptyset  )  |  \tsp d \tau.
\end{eqnarray}

We will make use of the following explicit formula (for example, it can be deduced from \cite[\S 17 and \S 19]{D}): Let $1 \leq q \leq X$ and $2 \leq T' \leq X^{1/2}$. For any primitive character $\chi$ modulo $q$, we have
\begin{eqnarray}
\label{explicit formula}
\sum_{y \leq X} \Lambda^*(y) \chi(y)  = \delta_{\chi = \chi^0} X - \sum_{ \substack{ \rho \in B_{T'}  \\  L(\rho, \chi) = 0  }} \frac{X^{\rho}}{\rho} + E(X),
\end{eqnarray}
where $\delta_{\chi = \chi^0} = 1$ if $\chi = \chi^0$ and $0$ otherwise, and
\begin{eqnarray}
\label{explicit error}
| E(X) |  \ll \frac{X (\log X)^2}{T'}.
\end{eqnarray}
Here the sum $\sum_{ \substack{ \rho \in B_{T'}  \\  L(\rho, \chi) = 0  }}$ is over all the zeros, with multiplicity, of $L(s, \chi)$ in $B_{T'}$.
We will be using (\ref{explicit formula}) with $T' = T$ and $N \ll X \ll N$.

With these notation we obtain the following lemma. Note we prove the lemma without assuming $\ell = 0$. 
\begin{lem}
\label{bound on W}
Without loss of generality let $\mathbf{j} = (1, \ldots, m)$ and $\mathbf{k} = (m+1, \ldots, m + \ell)$,
where $m > 0$ and $\ell \geq 0$. Suppose $\chi_v$ is a primitive character modulo $r_v$ $(1 \leq v \leq m)$.
Let $|\tau| < N^{\vartheta_0 - d}$.
Then we have
\begin{eqnarray}
\label{eqn lem bdd W}
&&\Big{|}  \mathcal{W}(\tau; \mathbf{i};  (1, \chi_1), \ldots,  (m, \chi_m); \mathbf{k} ) \Big{|}
\\
\notag
&\ll& \Big{|} \int_{0}^{\infty} \cdots \int_0^{\infty}
\prod_{v=1}^m  \sideset{}{'}\sum_{\rho_v} x_v^{\rho_v - 1}
\cdot
x_{m+1}^{\widetilde{\beta} - 1} \cdots x_{m + \ell}^{\widetilde{\beta} -1}
\tsp \varpi(\mathbf{x}) \tsp  e (\tau F(\mathbf{x}) ) \tsp d \mathbf{x} \Big{|}
+ \widetilde{E},
\end{eqnarray}
where
\begin{eqnarray}
\notag
\widetilde{E} =  \sum_{ \boldsymbol{\epsilon} \in \{ 0, 1 \}^{m} \backslash \{ \mathbf{0} \} } N^{n} N^{(\vartheta_0 - 1) (\epsilon_1 + \cdots + \epsilon_m)} E(N)^{\epsilon_1 + \cdots + \epsilon_m} \prod_{ \epsilon_j = 0  }  \sideset{}{'}\sum_{\rho_j} (\delta_{\min} N)^{\beta_j - 1}.
\end{eqnarray}
\end{lem}

\begin{proof}
Let $\varsigma(\mathbf{t}) = \varpi(\mathbf{t})   e (\tau F(\mathbf{t}) )$.
For each $1 \leq i \leq m$ let
$$
E'(x_i) =
\left\{
    \begin{array}{ll}
         E(x_i) + \{x_i\}
         &\mbox{if } \chi_i = \chi^0 ,\\
         E(x_i) - (x_i^{\widetilde{\beta}} / \widetilde{\beta} ) + \sum_{1 \leq t \leq x_i} t^{\widetilde{\beta} - 1}
         &\mbox{if } \chi_i = \widetilde{\chi} \chi^0, \\
         E(x_i) &\mbox{otherwise.}
    \end{array}
\right.
$$
Since we have (\ref{excepZZ}), it follows that
$$
|  - \frac{x_i^{\widetilde{\beta}} }{ \widetilde{\beta}} + \sum_{1 \leq t \leq x_i} t^{\widetilde{\beta} - 1} | \ll 1
$$
for any $1 \leq x_i \leq N$. Recall the definition of $\mathcal{W}(\tau; \mathbf{i};  (1, \chi_1), \ldots,  (m, \chi_m); \mathbf{k} )$ given in (\ref{DefnW}) and the convention
described in the sentence following it. By applying partial summation, the explicit formula (\ref{explicit formula}) and integration by parts to the integrand of $\mathcal{W}(\tau; \mathbf{i};  (1, \chi_1), \ldots,  (m, \chi_m); \mathbf{k} )$, we obtain
\begin{eqnarray}
\notag
&&\sum_{x_{1}, \ldots, x_{m} \in \mathbb{N}}
\chi_1(x_{1}) \Lambda^*(x_{1}) \cdots \chi_m(x_{m}) \Lambda^*(x_{m}) \tsp \varsigma(\mathbf{x})
\\
&=&
\notag
(-1) \sum_{x_{2}, \ldots, x_{m} \in \mathbb{N}}
 \chi_2(x_{2}) \Lambda^*(x_{2}) \cdots \chi_m(x_{m}) \Lambda^*(x_{m})
\int_0^{\infty} \frac{\partial \varsigma}{\partial t_1} (\mathbf{x}) E'(x_1) dx_1
\\
&+&
\notag
(-1) \sum_{x_{2}, \ldots, x_{m} \in \mathbb{N}}
\chi_2(x_{2}) \Lambda^*(x_{2}) \cdots \chi_m(x_{m}) \Lambda^*(x_{m})
\int_0^{\infty} \frac{\partial \varsigma}{\partial t_1} (\mathbf{x}) \sideset{}{'}\sum_{\rho_1} x_1^{\rho_1 - 1}  dx_1.
\end{eqnarray}
We repeat this procedure with respect to each $x_2, \ldots, x_m$, and obtain
\begin{eqnarray}
\mathcal{W}(\tau; \mathbf{i}; (1, \chi_1), \ldots,  (m, \chi_m); \mathbf{k} )
&=&
\notag
(-1)^m \sum_{ \boldsymbol{\epsilon} \in \{ 0, 1 \}^{m}  }
\int_{0}^{\infty} \cdots \int_{0}^{\infty}  \frac{\partial^{\epsilon_1 + \cdots + \epsilon_m}  \varsigma }
{ \partial^{\epsilon_1} t_{1} \cdots  \partial^{\epsilon_m} t_{m}  }(\mathbf{x})   \cdot
\\
\notag
&&x_{m+1}^{\widetilde{\beta} - 1} \cdots x_{m + \ell}^{\widetilde{\beta} -1}  \ \prod_{ \epsilon_j = 0  }
\sideset{}{'}\sum_{\rho_j} x_j^{\rho_j - 1} \cdot \prod_{\epsilon_i = 1} E'(x_i) \  d\mathbf{x}.
\notag
\end{eqnarray}
Clearly the summand with $\boldsymbol{\epsilon} = \mathbf{0}$ corresponds to the first term on the right hand side of (\ref{eqn lem bdd W}).
It can be verified that for any $|\tau| < N^{\vartheta_0 - d}$ and $\mathbf{x} \in \textnormal{supp}(\varpi)$, we have
$$
\Big{|} \frac{\partial^{\epsilon_1 + \cdots + \epsilon_m} \varsigma}{ \partial^{\epsilon_1} t_{1} \cdots  \partial^{\epsilon_m} t_{m}  }  (\mathbf{x})  \Big{|}
\ll  N^{(\vartheta_0 - 1)(\epsilon_1 + \cdots + \epsilon_m)}.
$$
Therefore, we obtain
\begin{eqnarray}
&&\Big{|} \int_{0}^{\infty} \cdots \int_{0}^{\infty}  \frac{\partial^{\epsilon_1 + \cdots + \epsilon_m} \varsigma}
{ \partial^{\epsilon_1} t_{1} \cdots  \partial^{\epsilon_m} t_{m}  }  (\mathbf{x})  \
x_{m+1}^{\widetilde{\beta} - 1} \cdots x_{m + \ell}^{\widetilde{\beta} -1} \prod_{ \epsilon_j = 0  } \sideset{}{'}\sum_{\rho_j} x_j^{\rho_j - 1} \cdot \prod_{\epsilon_i = 1} E'(x_i) \  d \mathbf{x} \ \Big{|}
\notag
\\
&\ll&
\notag
N^{n} N^{(\vartheta_0 - 1) (\epsilon_1 + \cdots + \epsilon_m)} E(N)^{\epsilon_1 + \cdots + \epsilon_m} \prod_{ \epsilon_j = 0  }  \sideset{}{'}\sum_{\rho_j} (\delta_{\min} N)^{\beta_j - 1},
\end{eqnarray}
which proves our lemma.
\end{proof}

We apply Proposition \ref{prop osc int} to the first term on the right hand side of (\ref{eqn lem bdd W}) with $\ell = 0$,
and obtain an upper bound as follows
\begin{eqnarray}
\label{eqn 3 m>0}
&&
\int_{|\tau| < N^{\vartheta_0 - d} } \Big{|}  \int_{0}^{\infty} \cdots \int_0^{\infty}
\prod_{v=1}^m  \sideset{}{'}\sum_{\rho_v} x_v^{\rho_v - 1}
\cdot \varpi(\mathbf{x}) \tsp  e (\tau F(\mathbf{x}) )   \tsp d \mathbf{x}  \Big{|}  \tsp  d \tau
\\
&=&
N^{n-d} \int_{|\tau| < N^{\vartheta_0} } \Big{|}  \int_{0}^{\infty} \cdots \int_0^{\infty}
\prod_{v=1}^m  \sideset{}{'}\sum_{\rho_v} (N x_v)^{\rho_v - 1}
\cdot  \prod_{u=1}^n \omega(x_{u} - x_{0, u}) \cdot   e (\tau F(\mathbf{x}) )   \tsp d \mathbf{x}  \Big{|}  \tsp  d \tau
\notag
\\
\notag
&\ll&
N^{n - d } \sideset{}{'}\sum_{\rho_1, \ldots, \rho_m} | N^{\rho_1 - 1} \cdots N^{\rho_m - 1} |   \cdot
\\
&&
\notag
\int_{|\tau| < N^{\vartheta_0}} \Big{|} \int_{0}^{\infty} \cdots \int_0^{\infty}
x_1^{\rho_1 - 1} \cdots x_m^{\rho_m - 1} \prod_{u=1}^n \omega(x_{u} - x_{0, u}) \cdot  e (\tau F(\mathbf{x}) )   \tsp d \mathbf{x}   \Big{|} \tsp  d \tau
\\
\notag
&\ll&
N^{n - d} (\log N) \sideset{}{'}\sum_{\rho_1, \ldots, \rho_m} N^{\beta_1 - 1} \cdots N^{\beta_m - 1}.
\end{eqnarray}
Therefore, by combining  (\ref{eqn 1 m>0}), (\ref{eqn 2 m>0}), (\ref{eqn lem bdd W}) and (\ref{eqn 3 m>0}), we obtain
\begin{eqnarray}
\label{eqn 4 m>0}
&& \Big{|} \sum_{1 \leq q \leq N^{\vartheta_0}} \sum_{\substack{ 1 \leq a \leq q \\  \gcd(a,q) = 1 }} \int_{|\tau| < N^{\vartheta_0 - d}}  S_{\mathbf{j}, \emptyset} \left( \frac{a}{q} + \tau \right)  d \tau \Big{|}
\\
\notag
&\ll& \sum_{ 1 \leq r_1, \ldots, r_m \leq N^{\vartheta_0} } R^{-\kappa}  \sideset{}{^*}\sum_{\substack{ \chi_v (\textnormal{mod} \tsp r_v)  \\   (1 \leq v \leq m)  } } N^{n - d} (\log N) \sideset{}{'}\sum_{\rho_1, \ldots, \rho_m} N^{\beta_1 - 1} \cdots N^{\beta_m - 1}
\\
\notag
&+&
\sum_{ 1 \leq r_1, \ldots, r_m \leq N^{\vartheta_0} }
R^{-\kappa}  \sideset{}{^*}\sum_{ \substack{ \chi_v (\textnormal{mod} \tsp r_v)  \\   (1 \leq v \leq m)  } }
\int_{|\tau| < N^{\vartheta_0 - d}}  \widetilde{E}  \ d \tau. 
\end{eqnarray}

We begin by bounding the first term on the right hand side of (\ref{eqn 4 m>0}). Clearly we have $R^{-\kappa} \leq r_1^{-\kappa}$.
Let $D > 1$ and $A'>0$. Then  by Lemma \ref{Gbdd} and Remark \ref{remZFR} we obtain that it is bounded by
\begin{eqnarray}
&\ll&
N^{n - d} (\log N) \left( \sum_{ \substack{ 1 \leq r_1 \leq N^{\vartheta_0}} } r_1^{-\kappa}  \sideset{}{^*}\sum_{\chi_1 (\textnormal{mod} \tsp r_1)} \sideset{}{'}\sum_{\rho_1 } N^{\beta_1 - 1} \right)
\prod_{v=2}^m  \
\sum_{ \substack{ 1 \leq r_v \leq N^{\vartheta_0}}} \  \sideset{}{^*}\sum_{\chi_v (\textnormal{mod} \tsp r_v)}
\sideset{}{'}\sum_{\rho_v} N^{\beta_v - 1}
\notag
\\
\notag
&\ll&  N^{n-d} (\log N) \sum_{  1 \leq r_1 \leq N^{\vartheta_0} } r_1^{-\kappa} \sideset{}{^*}\sum_{\chi_1 (\textnormal{mod} \tsp r_1)} \sideset{}{'}\sum_{\rho_1} N^{\beta_1 - 1}
\\
\notag
&\ll&  N^{n-d} (\log N)  \sum_{  1 \leq r_1 \leq (\log N)^D } \  \sideset{}{^*}\sum_{\chi_1 (\textnormal{mod} \tsp r_1)} \sum_{\substack{\rho_1 \in B_T \\ L(\rho_1, \chi_1) = 0 }} N^{\beta_1 - 1}
\\
\notag
&+&
N^{n-d} (\log N)^{- D \kappa + 1} \sum_{  1 \leq r_1 \leq N^{\vartheta_0} } \  \sideset{}{^*}\sum_{\chi_1 (\textnormal{mod} \tsp r_1)} \sideset{}{'}\sum_{\rho_1} N^{\beta_1 - 1}
\\
\notag
&\ll& N^{n-d} (\log N)^{- A'} + N^{n - d} (\log N)^{- D \kappa + 1}.
\end{eqnarray}
Next we bound the second term on the right hand side of (\ref{eqn 4 m>0}). Given $ \boldsymbol{\epsilon} \in \{ 0, 1 \}^{m} \backslash \{ \mathbf{0} \}$
let $\iota(\boldsymbol{\epsilon})$ be the smallest number $\iota$ such that $\epsilon_{\iota} =1$.
Recall $\kappa > 2$, (\ref{defnT}) and (\ref{explicit error}). By Lemma \ref{Gbdd} we have
\begin{eqnarray}
\notag
&& \sum_{  1 \leq r_1, \ldots, r_m \leq N^{\vartheta_0} } R^{-\kappa}  \sideset{}{^*}\sum_{\substack{\chi_v (\textnormal{mod} \tsp r_v) \\  (1 \leq v \leq m)}}
\int_{|\tau| < N^{\vartheta_0 - d}}
\widetilde{E} \ d \tau
\\
&\ll& \sum_{ \boldsymbol{\epsilon} \in \{ 0, 1 \}^{m} \backslash \{ \mathbf{0} \} }  N^{n - d + \vartheta_0} N^{(\vartheta_0 - 1) (\epsilon_1 + \cdots + \epsilon_m)} E(N)^{\epsilon_1 + \cdots + \epsilon_m} \sum_{  1 \leq r_{\iota(\boldsymbol{\epsilon})} \leq N^{\vartheta_0}} r_{ \iota(\boldsymbol{\epsilon}) }^{- \kappa} \  \sideset{}{^*}\sum_{\chi_{\iota(\boldsymbol{\epsilon})} (\textnormal{mod} \tsp r_{\iota(\boldsymbol{\epsilon})})} 1 \cdot
\notag
\\
&&
\prod_{ \substack{ \epsilon_{i} = 1 \\ i \not =  \iota(\boldsymbol{\epsilon}) } }  \sum_{  1 \leq r_{i} \leq N^{\vartheta_0}} \ \sideset{}{^*}\sum_{\chi_i (\textnormal{mod} \tsp r_i)} 1
\ \cdot \
\prod_{ \epsilon_j = 0  } \,  \sum_{  1 \leq r_j \leq N^{\vartheta_0}} \ \sideset{}{^*}\sum_{\chi_j (\textnormal{mod} \tsp r_j)} \sideset{}{'}\sum_{\rho_j} (\delta_{\min} N)^{\beta_j - 1}
\notag
\\
&\ll&
\sum_{ \boldsymbol{\epsilon} \in \{ 0, 1 \}^{m} \backslash \{ \mathbf{0} \} }  N^{n - d + \vartheta_0} N^{(\vartheta_0 - 1) (\epsilon_1 + \cdots + \epsilon_m)} \left( N^{2 \vartheta_0}E(N) \right)^{\epsilon_1 + \cdots + \epsilon_m}  N^{- 2 \vartheta_0} \cdot
\notag
\\
&&
\prod_{ \epsilon_j = 0  } \,  \sum_{  1 \leq r_j \leq N^{\vartheta_0}} \  \sideset{}{^*}\sum_{\chi_j (\textnormal{mod} \tsp r_j)} \sideset{}{'}\sum_{\rho_j} (\delta_{\min} N)^{\beta_j - 1}
\notag
\\
&\ll&
N^{n - d  - \vartheta_0 - 2 \varepsilon_0 }.
\notag
\end{eqnarray}

\subsection{Case $\ell > 0$}
\label{secm=0}
In this case, we only need to consider $q$ divisible by $\widetilde{r}$.
Without loss of generality let $\mathbf{j} = (1, \ldots, m)$ and $\mathbf{k} = (m+1, \ldots, m+{\ell})$.
First we suppose $m=0$.
By (\ref{usinglem7.4}) (with $\widetilde{r}$ in place of $R$) and Proposition \ref{prop osc int} we obtain
\begin{eqnarray}
\notag
&& \Big{|} \sum_{ \substack{1 \leq q \leq N^{\vartheta_0} \\ \widetilde{r} | q }}    \sum_{\substack{ 1 \leq a \leq q \\  (a,q) = 1 }} \int_{|\tau| < N^{\vartheta_0 - d}}  S_{\emptyset, \mathbf{k}} \left( \frac{a}{q} + \tau \right)  d \tau \Big{|}
\\
\notag
&\leq& \sum_{ \substack{1 \leq q \leq N^{\vartheta_0} \\ \widetilde{r} | q }}  \sum_{\substack{ 1 \leq a \leq q \\  (a,q) = 1 }}   \frac{1}{\phi(q)^n} | \mathcal{A}(q, a; \mathbf{i}; \emptyset ; (1, \ldots, \ell)  ) | \cdot
\\
\notag
&&
 \int_{|\tau| < N^{\vartheta_0 - d}} \Big{|}  \int_{0}^{\infty} \cdots \int_0^{\infty}
x_{1}^{\widetilde{\beta} - 1} \cdots x_{\ell}^{\widetilde{\beta} - 1} \tsp \varpi(\mathbf{x}) \tsp  e (\tau F(\mathbf{x}) )   \tsp d \mathbf{x} \Big{|}  d \tau
\\
&\ll&
\widetilde{r}^{-\kappa} N^{n - d + \ell ( \widetilde{\beta} - 1) }
\int_{|\tau| < N^{\vartheta_0} } \Big{|} \int_{0}^{\infty} \cdots \int_0^{\infty}
x_{1}^{\widetilde{\beta} - 1} \cdots x_{\ell}^{\widetilde{\beta} - 1} \ \prod_{u=1}^n \omega(x_{u} - x_{0, u}) \cdot e (\tau F(\mathbf{x}) )   \tsp d \mathbf{x}  \Big{|} \tsp   d \tau
\notag
\\
&\ll&
\widetilde{r}^{-\kappa} N^{n - d} (\log N),
\notag
\end{eqnarray}
and from Remark \ref{remZFR} we have $\widetilde{r}^{-1} \ll (\log N)^{-A}$ for any $A>0$. 
When $m > 0$, we proceed in a similar manner as in Section \ref{sec m>0} and obtain
$$
\Big{|} \sum_{ \substack{1 \leq q \leq N^{\vartheta_0} \\ \widetilde{r} | q }}    \sum_{\substack{ 1 \leq a \leq q \\  (a,q) = 1 }} \int_{|\tau| < N^{\vartheta_0 - d}}  S_{\mathbf{j}, \mathbf{k}} \left( \frac{a}{q} + \tau \right)  d \tau \Big{|}
\ll N^{n - d} (\log N)^{-A'}
\notag
$$
for any $A'>0$. Since it only requires minimal modifications, we omit the details. 

\begin{rem}
There is work regarding enlarging the major arcs for problems related to Goldbach's conjecture, for example
\cite{MV} and \cite{R}; these were used as a guidance for developing the argument in this section.
\end{rem}

\subsection{Proof of Theorem \ref{mainthm2}} \label{secmajormainterm}
We combine (\ref{S and S1}) and (\ref{MAJOR1}) with the estimates from Sections \ref{sec m>0} and \ref{secm=0}.
As a result, we obtain
\begin{eqnarray}
\notag
&&\int_{\mathfrak{M}'(\vartheta_0)} S(\alpha) \tsp d\alpha
\\
&=&
\notag
\sum_{1 \leq q \leq N^{\vartheta_0}} \sum_{\substack{ 1 \leq a \leq q \\  (a,q) = 1 }} \frac{1}{\phi(q)^n} \tsp \mathcal{A} (q, a; (1, \ldots, n) ; \emptyset ;  \emptyset )
\cdot
\int_{|\tau| < N^{\vartheta_0 - d}}
\mathcal{W} (\tau;(1, \ldots, n) ;  \emptyset ; \emptyset ) \tsp d \tau
\\
\notag
&+& O \left(N^{n - d + 4 \vartheta_0 - 1} +  N^{n - d + 3 \vartheta_0 - \frac12} + N^{n - d  - \vartheta_0 - 2 \varepsilon_0 } +  \frac{N^{n - d}}{(\log N)^{A_1} } \right)
\end{eqnarray}
for any $A_1 > 0$.
By a standard argument (for example, see \cite[Section 5]{BP} or \cite[\S 2]{SS}) it follows that there exists $\delta'>0$ such that
\begin{eqnarray}
\notag
&&\int_{|\tau| < N^{\vartheta_0 - d}}
\mathcal{W}(\tau;(1, \ldots, n) ;  \emptyset ; \emptyset ) \tsp d \tau
\\
\notag
&=&
N^{n-d} \int_{|\tau| < N^{\vartheta_0}} \int_{0}^{\infty} \cdots \int_{0}^{\infty}
\prod_{u=1}^n \omega(x_{u} - x_{0, u}) \cdot e( \tau F(\mathbf{x}))  \tsp  d \mathbf{x} \tsp d \tau
\\
\notag
&=& c_1(F; \omega, \mathbf{x}_0) \, N^{n - d} + O(N^{n-d - \delta'}),
\end{eqnarray}
where $c_1(F; \omega, \mathbf{x}_0) \geq  0$ depends only on $F$, $\omega$ and $\mathbf{x}_0$. Also the work in \cite[Section 7]{CM} implies that there exists $\delta'' >0$
such that
\begin{eqnarray}
\notag
\sum_{1 \leq q \leq N^{\vartheta_0}} \sum_{\substack{ 1 \leq a \leq q \\  (a,q) = 1 }} \frac{1}{\phi(q)^n} \tsp \mathcal{A} (q, a; (1, \ldots, n) ; \emptyset ;  \emptyset ) &=&
\sum_{1 \leq q \leq N^{\vartheta_0}} \sum_{\substack{ 1 \leq a \leq q \\  (a,q) = 1 }} \frac{1}{\phi(q)^n}
\sum_{\substack{ \mathbf{h} \in \mathbb{U}_q^n } }  e \left( \frac{a}{q}  F(\mathbf{h} )   \right)
\\
\notag
&=& c_2(F) + O(N^{-\delta''}),
\end{eqnarray}
where $c_2(F) \geq 0$ depends only on $F$. 
The constant $c_1(F; \omega, \mathbf{x}_0) c_2(F)$ is a product of local densities; in fact, $c_1(F; \omega, \mathbf{x}_0) c_2(F) > 0$ provided $F$ satisfies the local conditions \textnormal{($\star$)} and $\varpi$ is as in this section. Furthermore, it follows from Proposition \ref{minor arc est} that
$$
\Big{|} \int_{\mathfrak{M}'(\vartheta_0) \backslash \mathfrak{M}(\vartheta_0) } S(\alpha) \tsp d\alpha \Big{|} \leq  \int_{\mathfrak{m}(\vartheta_0) } |S(\alpha)| \tsp d\alpha \ll \frac{N^{n-d}}{(\log N)^{A_2}}
$$
for any $A_2>0$. Therefore, by combining these estimates and Proposition \ref{minor arc est} with (\ref{orthog reln}), we obtain Theorem \ref{mainthm2} with
$c(F; \omega, \mathbf{x}_0) = c_1(F; \omega, \mathbf{x}_0) c_2(F)$.

\section{Reduction from prime power solutions to prime solutions}

\begin{thm}
\label{thm in B}
Let $\mathbf{y} = (y_1, \ldots, y_{m})$ and $\mathfrak{c}' > 0$. 
Let $f \in \mathbb{Z}[y_1, \ldots, y_m]$ be a polynomial of degree $d \geq 2$ such that 
$\| f^{[d - i]} \| \leq \mathfrak{c}' N^{i}$ $(0 \leq i \leq d)$ and
$$
\textnormal{codim} \tsp V_{f^{[d]}}^*   > (d-1) 2^d.
$$
Then we have
$$
\sum_{ \mathbf{y} \in [0,N]^m } \mathbbm{1}_{V(f)} (\mathbf{y}) \ll N^{m - d},
$$
where the implicit constant may depend on $f^{[d]}$ and $\mathfrak{c}'$ but not on $(f - f^{[d]})$.
\end{thm}
This is a slight variant of the main theorem of \cite{B}, the difference being we have a polynomial whose absolute values of its coefficients of lower degree terms may be large. However, the restriction given in the statement ensures the argument in the major arcs analysis to still go through (the lower degree terms play no role in the minor arcs estimate), and the result follows.
Thus provided $n - \textnormal{codim} \tsp V_F^* - 2 > (d-1) 2^d$, which in particular is satisfied
assuming (\ref{codim of F}), in combination with Lemma \ref{Lemma on the B rank} and Theorem \ref{thm in B} we have
\begin{eqnarray}
\notag
\sum_{i=1}^n  \sum_{\substack{ \mathbf{x} \in [0,N]^{n} \\ x_i \not \in \wp } } \Lambda(\mathbf{x})
\mathbbm{1}_{V(F)}(\mathbf{x})
&\ll& (\log N)^{n-1} \sum_{i=1}^n \sum_{ \substack{ z \in [0,N]  \\ z \not \in \wp } } \Lambda(z) \sum_{ \substack{ \mathbf{x} \in [0,N]^{n} \\ x_i = z } }
\mathbbm{1}_{V(F)}(\mathbf{x})
\\
&\ll& (\log N)^n N^{1/2} N^{(n-1) - d}.
\notag
\end{eqnarray}
Therefore, from this estimate and Theorem \ref{mainthm2} we obtain Theorem \ref{mainthm1}.


\begin{thebibliography}{9}

\bibitem{B}  B. J. Birch, \textit{Forms in many variables}. Proc. Roy. Soc. Ser. A 265 1961/1962, 245--263.


\bibitem{Bro} T. D. Browning, \textit{Quantitative arithmetic of projective varieties}, Progress in Mathematics, vol. 277, Birkh\"auser Verlag, Basel, 2009.



\bibitem{BP} T. D. Browning, and S. M. Prendiville, \textit{Improvements in Birch's theorem on forms in many variables}.
J. reine angew. Math. 731 (2017), 203--234.


\bibitem{CM} B. Cook and {\'A}. Magyar, \textit{Diophantine equations in the primes}. Invent. Math. {198} (2014), 701--737.


\bibitem{D1} H. Davenport, \textit{Analytic methods for Diophantine equations and Diopantine inequalities}.
Second edition. Cambridge University Press, Cambridge, 2005.

\bibitem{D} H. Davenport, \textit{Multiplicative Number Theory}. Third edition.
Springer-Verlag, New York, 2000.




\bibitem{G} P. X. Gallagher, A large sieve density estimate near $\sigma=1$. Invent. Math. 11 (1970), 329--339.


\bibitem{F} K. Ford, \textit{Zero-free regions for the Riemann zeta function}.
Number theory for the millennium, II (Urbana, IL, 2000), 25--56, A K Peters, Natick, MA, 2002.

\bibitem{GT0} B. Green and  T. Tao, \textit{The primes contain arbitrarily long arithmetic progressions}. Ann. of Math. (2) 167 (2008), no. 2, 481--547.

\bibitem{GT1} B. Green and  T. Tao, \textit{Linear equations in primes}.  Ann. of Math. (2) 171 (2010), no. 3, 1753--1850.

\bibitem{GT2} B. Green and  T. Tao, \textit{The M\"obius function is asymptotically orthogonal to nilsequences}.  Ann. of Math. 175 (2012), 541--566.

\bibitem{GTZ} B. Green,  T. Tao and T. Ziegler, \textit{An inverse theorem for the Gowers $U^{s+1}[N]$-norm}.  Ann. of Math. 176 (2012), no. 2, 1231--1372.


\bibitem{HB} D. R. Heath-Brown, \textit{A new form of the circle method, and its application to quadratic forms}. J. reine angew. Math. 481 (1996), 149--206.





\bibitem{Hux} M. N. Huxley,
\textit{Large values of Dirichlet polynomials. III}.  Acta Arith. 26 (1974/75), no. 4, 435--444.


\bibitem{H0} H. A. Helfgott, \textit{The ternary Goldbach problem}, to appear in Annals of Mathematics Studies.




\bibitem{I} H. Iwaniec, \textit{On Zeros of Dirichlet's $L$ Series}. Invent. Math. 23 (1974), 97--104.

\bibitem{J} M. Jutila, \textit{On Linnik's constant}. Math. Scand. 41 (1977), no. 1, 45--62.

\bibitem{L2} J. Liu, \textit{Integral points on quadrics with prime coordinates}.  Monatsh.  Math. 164 (2011), no.4,  439--465.

\bibitem{L} J. Liu, \textit{On Lagrange's theorem with prime variables}. Q. J. Math. 54 (2003), no. 4, 453--462.

\bibitem{M} J. Maynard, \textit{Small gaps between primes}. Ann. of Math. (2) 181 (2015), no. 1, 383--413.

\bibitem{MV} H. L. Montgomery and R. C. Vaughan,
\textit{The exceptional set in Goldbach's problem}. Acta Arith. 27 (1975), 353--370.

\bibitem{R} X. M. Ren, \textit{The major arcs in the ternary Goldbach problem}. Acta Math. Hungar. 98 (2003), no. 1-2, 39--58.

\bibitem{DS} D. Schindler, \textit{Bihomogeneous forms in many variables}. J. Th\'{e}orie Nombres Bordeaux 26 (2014), 483--506

\bibitem{SS} D. Schindler and E. Sofos, \textit{Sarnak's saturation problem for complete intersections}. Mathematika 65 (2019), no. 1, 1--56.

\bibitem{S} W. M. Schmidt, \textit{The density of integer points on homogeneous varieties}.
Acta Math. {154} (1985), no. 3-4, 243--296.


\bibitem{V} I. M. Vinogradov. \textit{Representation of an odd number as a sum of three primes}. {Dokl. Akad. Nauk. SSSR} {15} (1937), 291--294.


\bibitem{XY} S. Y. Xiao and S. Yamagishi, \textit{Zeroes of polynomials with prime inputs and Schmidt's $h$-invariant}.
Canadian J. Math. 72 (2020), 805--833.

\bibitem{Y1} S. Yamagishi, \textit{Diophantine equations in semiprimes}. Discrete Analysis 2019:17, 21 pp.

\bibitem{Y} S. Yamagishi, \textit{On an oscillatory integral involving a homogeneous form}.
Funct. Approx. Comment. Math. 62 (2020), 21--58.


\bibitem{Y2} S. Yamagishi, \textit{Prime solutions to polynomial equations in many variables and differing degrees}. Forum Math. Sigma 6 (2018), e19, 89 pp.

\bibitem{Z}  Y. Zhang, \textit{Bounded gaps between primes}. Ann. of Math. (2) 179 (2014), no. 3, 1121--1174.


\bibitem{Zh} L. Zhao, \textit{The quadratic form in nine prime variables}. Nagoya Math. J., 223 (1) (2016), 21--65.


\end{thebibliography}
\end{document}